\documentclass[bj,authoryear]{imsart}

\usepackage{float,comment,natbib}
\usepackage[ruled,vlined]{algorithm2e}
\usepackage{multirow}
\usepackage{bbm}

\startlocaldefs

\numberwithin{equation}{section}

\def \bfs{\bm{s}}

\def \bft{\mathbf{t}}

\def \ev{E}
\def \pr{\mathbb{P}}

\def \bft{\mathbf{t}}

\def \cip{\xrightarrow[\text{}]{\text{$\mathbb{P}$}}}
\newcommand{\vo}{\vec{o}\@ifnextchar{^}{\,}{}}

\theoremstyle{plain}

\newtheorem{remark}{Remark}

\newtheorem{theorem}{Theorem}[section]
\newtheorem{lemma}[theorem]{Lemma}
\newtheorem{corollary}[theorem]{Corollary}
\newtheorem{proposition}[theorem]{Proposition}

\theoremstyle{remark}

\endlocaldefs

\begin{document}

\begin{frontmatter}
\title{Inverse regression for  spatially distributed functional data}
\runtitle{Spatial Inverse Regression}

\begin{aug}
\author[A]{\inits{F.}\fnms{Suneel Babu}~\snm{Chatla}\ead[label=e1]{sbchatla@utep.edu}}
\author[B]{\inits{S.}\fnms{Ruiqi}~\snm{Liu}\ead[label=e2]{ruiqliu@ttu.edu}}
\address[A]{Department of Mathematical Sciences, University of Texas at El Paso, TX 79968, USA\printead[presep={,\ }]{e1}}

\address[B]{Department of Mathematics and Statistics, Texas Tech University, TX 79409, USA\printead[presep={,\ }]{e2}}

\end{aug}

\begin{abstract}
  Spatially distributed functional data are prevalent in many statistical applications such as meteorology, energy forecasting, census data,  disease mapping, and neurological studies. Given their complex and high-dimensional nature, functional data often require dimension reduction methods to extract meaningful information. Inverse regression is one such approach that has become very popular in the past two decades. We study the inverse regression in the framework of functional data observed at irregularly positioned spatial sites. The functional predictor is the sum of a spatially dependent functional effect and a spatially independent functional nugget effect, while the relation between the scalar response and the functional predictor is modeled using the inverse regression framework. For estimation, we consider local linear smoothing with a general weighting scheme, which includes as special cases the schemes under which equal weights are assigned to each observation or to each subject. This framework enables us to present the asymptotic results for different types of sampling plans over time such as non-dense, dense, and ultra-dense.  We discuss the domain-expanding infill (DEI) framework for spatial asymptotics, which is a mix of the traditional expanding domain and infill frameworks. The DEI framework overcomes the limitations of traditional spatial asymptotics in the existing literature.  Under this unified framework, we develop asymptotic theory and identify conditions that are necessary for the estimated eigen-directions to achieve optimal rates of convergence. Our asymptotic results include pointwise and $L_2$ convergence rates.  Simulation studies using synthetic data and an application to a real-world dataset confirm the effectiveness of our methods.
\end{abstract}


\end{frontmatter}

\section{Introduction}

Sufficient dimension reduction methods have become popular owing to their usefulness in extracting useful information from high-dimensional data and have found a wide range of applications. For a brief overview of sufficient dimension reduction methods, we refer to \citet{li2018sufficient} and \citet{cook2018introduction}. For recent advancements on post-reduction inference, we refer to \citet{kim2020post}. In this proposal, we present a method for estimating the sufficient dimension reduction space where the scalar response and the functional predictors are spatially distributed and related through linear indices as assumed in \citet{muller2005generalized}, \citet{ferre2003functional}, \citet{ferre2005smoothed}, \citet{hsing2009rkhs} and \citet{jiang2014inverse}. Spatially distributed functional data are prevalent in many statistical applications such as meteorology, energy forecasting, census data,  disease mapping, and neurological studies.  One well-known application is the spatial prediction of the Canadian temperature and precipitation data  \citep{ramsey2005functional,giraldo2010continuous}. The dataset includes the daily average temperatures at 35 weather stations over 30 years. The locations of these 35 stations serve as spatial coordinates for the spatial prediction of weather data.

Although there are model-based and model-free sufficient dimension reduction methods, the latter is more robust and flexible as it relies on weaker assumptions. The inverse regression or sliced inverse regression (SIR) proposed by \citet{li1991sliced} and \citet{duan1991slicing} is one important approach in this paradigm. The SIR approach assumes that the scalar response variable in a regression model depends on the multivariate predictor through an unknown number of linear projections. Formally,
\begin{align*}
    Y \perp \!\!\! \perp  \bm{X} ~\vert ~ \bm{\beta}^T\bm{X},
\end{align*}
where $\bm{X}$ is a $p$-dimensional random vector, $Y$ is a random variable, and the coefficient matrix $\bm{\beta}$ is a $p \times K$ matrix such that $K$ is much smaller than $p$. The matrix $\bm{\beta}$ is not identifiable. However, the space spanned by the columns of $\bm{\beta}$, the effective dimension reduction (e.d.r.) space, is of interest. \citet{li1991sliced} showed that the e.d.r. directions $\{b_j\}_{j=1}^K$ are estimated from the generalized eigen-analysis
\begin{align}
    \text{cov}[E(\bm{X}|Y)] b_j &= \lambda_j \text{cov}(\bm{X})b_j,
\label{eqn:intro-gea}
\end{align}
where the $\lambda_j$'s are eigenvalues, under the following design condition:
\begin{align*}
    E(\bm{c}^T \bm{X} |\bm{\beta}^T\bm{X}) \text{ is linear in } \bm{\beta}^T\bm{X}, \text{ for any direction } \bm{c}\in \mathbb{R}^p.
\end{align*}
For more variants of the above framework, we refer to, for example, \citet{cook1991sliced}, \citet{cook2002dimension}, \citet{cook2007fisher}, and \citet{cook2010necessary}.  

Proposing the SIR approach to functional data $X(\cdot)$, which are stochastic processes observed over a time interval $\mathcal{I} \subset \mathbb{R}$, is not simple due to the complication of inverting a covariance operator on $L_2(\mathcal{I})$. The first papers to take functional SIR into consideration were \citet{,ferre2003functional,ferre2005smoothed}. They showed that it is feasible to estimate the e.d.r. space under regularity conditions as long as the true dimension of the space is known \citep{,forzani2007note, ferre2007reply}. For some further refinements and alternatives, we refer to \citet{,hsing2009rkhs}, \citet{cook2010necessary}, \citet{chen2011single}, \citet{li2017nonlinear,song2019sufficient}, \citet{li2022dimension} and \citet{lee2022functional}. An assumption commonly made in functional data analysis is that the complete trajectories for a sample of $n$ random functions are fully observed. This assumption, however, may not hold in practice because the measurements are obtained at discrete and scattered time points. As mentioned in \citet{jiang2014inverse}, one way to overcome this limitation is to borrow information across all sample functions by applying smoothing to obtain estimators. Unfortunately, the spatial points are often irregularly positioned in space which complicate the smoothing of functional data.

One important issue with spatial data is that they require a $d$-dimensional index space $\{\bm{s} \in \mathbb{R}^d, d \ge 2\}$ and one needs to take the dependence across all directions into consideration. Instead of defining a random field through specific model equations, the dependence conditions are often imposed by the decay of covariances or the use of mixing conditions. Different asymptotics holds in the presence of clusters and for sparsely distributed points. Traditionally, there exist two different approaches for spatial data: expanding domain and infill asymptotics \citep{cressie2015statistics}. The expanding domain asymptotics requires that the domain of spatial measurement locations tends to infinity. It is similar to time series and more suited for a situation where the measurements are on a grid. In the infill asymptotics, the total domain is kept fixed, but the density of measurements in that domain is allowed to increase.  By proposing a general sampling framework,  \citet{,hormann2013consistency} established conditions for which sample mean and sample covariance of the spatially distributed functional data are consistent. Their framework, however, assumes that the complete trajectories of the functions are known.

This study proposes a framework of inverse regression for functional data observed at irregularly positioned spatial locations as well as sparsely observed in time. The data consist of a scalar response $Y(\bm{s}_j)$ and a functional covariate $X(\bm{s}_j) = \{X(\bm{s}_j;t), t \in \mathcal{I}\}$, observed at spatial points $\bm{s}_j \in \mathbb{R}^2$, $j=1,\ldots,n$. We assume that each observed functional covariate has a location-specific random process which is interpreted as the “nugget'' effect. The nugget effects are uncorrelated across functions at different spatial locations. Our contribution is twofold, which is summarized as follows.
\begin{enumerate}[label=(\roman*)]
    \item Methodologically, we consider local linear smoothing \citep{fan2018local} to estimate the inverse regression function $E(X(\bm{s};t)| Y(\bm{s}))$ and the covariance function of  $X(\bm{s};t)$ by borrowing the information across all the subjects. The method is applicable to the unified weighting scheme proposed in \citet{zhang2016sparse}, including non-dense (cannot attain the $\sqrt{n}$ rate), dense (attain $\sqrt{n}$ rate with non-negligible bias) and ultra-dense (attain $\sqrt{n}$ rate without asymptotic bias) categories of functions. The proposed approach allows both the density of observations and the domain to increase, and therefore retains the advantages of both spatial data frameworks. 
    \item Theoretically, we identify the sufficient conditions under which the spatial-temporal covariance, the inverse regression function and the covariance function are consistent in Theorems \ref{th:Rhat}, \ref{th:1} and \ref{thm:rate-phi}, respectively.  As they are derived for irregular spatial locations, each of these results is new to the literature and  hence is of independent interest. Our asymptotic results include both pointwise rates and $L_2$ rates which are new to the literature. It is also worth mentioning that the theoretical results of spatially correlated data are not trivial extensions from the independent case. For example, Lemma S.1.5 in the supplement \citet{suppl}  proves a uniform convergence for kernel estimators with spatially correlated observations, which makes use Bradley's coupling lemma (Lemma 1.2 in \citealp{bosq2012nonparametric}). Applying  the uniform convergence results, we give the convergence rate for the e.d.r. directions in Theorem \ref{th:main}. The rate for e.d.r. direction depends on the convergence rates of spatio-temporal covariance, inverse regression function, and the number of basis functions used for the truncation of covariance operator.
\end{enumerate}

The remainder of the paper is organized as follows. We outline the inverse regression procedure for spatially distributed functional data in Section \ref{sec:IRforSpatial}. The formulation of the domain-expanding infill framework is included in Section \ref{sec:sample-method}, which also covers the estimation of covariance operators and e.d.r. directions under local linear smoothing framework.  In Section \ref{sec:asymp}, we outline the conditions that are required for our asymptotic study and derive the convergence rates for covariance operators and e.d.r. directions.  Due to space constraints, the proofs of the main results, a validation of the theoretical findings through simulations, and a real data analysis are deferred to the supplement \citet{suppl}.


\textbf{Notation:}  For functions $f(x)$ and $g(x,y)$ defined on a compact set, let $\Vert f \Vert_{L_2}= \left[\int f^2(x) dx\right]^{1/2}$ and $\Vert g \Vert_{HS}=\left[\int \int g^2(x,y) dx dy \right]^{1/2}$. For any vector $\bm{x}$, denote its $L_2$ norm as $\Vert \bm{x} \Vert$. For any positive sequences $\{a_n\}$ and $\{b_n\}$, we write $a_n \lesssim b_n$ if $a_n/b_n$ is bounded by a constant, and write $a_n \asymp b_n$ if $C_1 \le a_n/b_n \le C_2$ for some $C_1, C_2 >0$.

\section{Inverse regression for spatial function} \label{sec:IRforSpatial}
In this section, we formulate the inverse regression framework for spatially distributed functional data.
We assume $\{(Y(\bm{s}), X(\bm{s})), \bm{s} \in \mathcal{D}_n \subseteq \mathbb{R}^2\}$ is a stationary random field with $Y(\bm{s})$ being a scalar response variable and a functional covariate $X(\bm{s})$ taking values in $L^2(\mathcal{I})$ where $\mathcal{I} \subset \mathbb{R}$. The space $L^2(\mathcal{I})$ is a collection of square-integrable functions on $\mathcal{I}$. The value of the function $X(\bm{s})$ at $t\in \mathcal{I}$ is denoted by $X(\bm{s};t)$. We consider the following model
\begin{align}
  Y(\bm{s}) &= f(\langle\beta_1(\bm{s}),X(\bm{s})\rangle, \ldots, \langle\beta_K(\bm{s}),X(\bm{s})\rangle, \epsilon(\bm{s})),
  \label{eqn:secM-i1}
\end{align}
where the $K$ unknown functions $\beta_1(\cdot), \ldots, \beta_K(\cdot)$ are in $L^2(\mathcal{I})$, $f$ is an arbitrary unknown function on $\mathbb{R}^{K+1}$, $\epsilon(\bm{s})$ is a random error that is independent of $X(\bm{s})$, and the inner product is defined as
\begin{align*}
  \langle\beta_K(\bm{s}),X(\bm{s})\rangle &= \int_{\mathcal{I}} \beta_K(\bm{s};t) X(\bm{s};t) dt.
\end{align*}
Assuming that the spatial dependence is second-order stationary and isotropic, we write the covariance function of $X(\bm{s};t)$ as 
\begin{align}
  R(\Vert \bm{s}_1- \bm{s}_2 \Vert; t_1, t_2) &= \text{cov}\{ X(\bm{s}_1; t_1),  X(\bm{s}_2;t_2) \}  
\end{align}
for any  $(\bm{s}_1, t_1),~ (\bm{s}_2, t_2) \in \mathcal{D}_n \times \mathcal{I}$. The covariance operator $R(\cdot, \cdot, \cdot)$ exists and bounded given that
\begin{align}
  E\Vert X(\bm{s}) \Vert_{L_2}^4 < \infty \qquad  \text{ for  all } \bm{s} \in \mathcal{D}_n,
  \label{eqn:xfun-4m}
\end{align}
where the $L_2$ norm is defined as
\begin{align*}
  \Vert X(\bm{s}) \Vert_{L_2} &:= \left\{ \int_{\mathcal{I}} X^2(\bm{s};t) dt \right\}^{1/2}.
\end{align*}
Similarly, we define the conditional  covariance operator  
\begin{align}
R_e(\Vert \bm{s}_1-\bm{s}_2 \Vert;t_1,t_2) = \text{cov}\{ E(X(\bm{s}_1; t_1)| Y(\bm{s}_1)), E(X(\bm{s}_2; t_2)| Y(\bm{s}_2))\}, 
\end{align}
for any $(\bm{s}_1, t_1),~ (\bm{s}_2, t_2) \in \mathcal{D}_n \times \mathcal{I}$. We note that, $R(0;t_1,t_2)$ and $R_e(0;t_1,t_2)$ are the variance operators of $X(\bm{s})$ and $E(X(\bm{s})| Y(\bm{s}))$, respectively.

Following \citet{li1991sliced} and \citet{ferre2003functional}, we obtain that
\begin{align}
  & E( \langle b(\bm{s}),X(\bm{s})\rangle |\langle\beta_1(\bm{s}),X(\bm{s})\rangle, \ldots, \langle\beta_K(\bm{s}),X(\bm{s})\rangle) \text{ is linear in } \nonumber \\
  & \qquad \langle\beta_1(\bm{s}),X(\bm{s})\rangle, \ldots, \langle\beta_K(\bm{s}),X(\bm{s})\rangle \text{ for any direction $b(\bm{s})$ in }L^2(\mathcal{I}).
\end{align}
In other words, under model (\ref{eqn:secM-i1}), the functions $\langle\beta_j(\bm{s}),X(\bm{s})\rangle$, for $j=1,\ldots,K$, summarize all the information contained in the functional variable $X(\bm{s})$ for predicting $Y(\bm{s})$. The individual coefficient functions $\beta_j(\bm{s})$, for $j=1,\ldots,K$, are not identifiable since the link function $f(\cdot)$ in (\ref{eqn:secM-i1}) is unknown. However, our interest lies in estimating the effective dimension reduction (e.d.r.) space which is spanned by the individual coefficient functions and is identifiable.
Moreover, the operator $R_e(0)$ is degenerate in any direction $R(0)-$ orthogonal to the  e.d.r. space. This implies we can recover the basis of e.d.r. space through the $R(0)-$ orthogonal eigenvectors of $R_e(0)$ associated with the $K$ largest eigenvalues
\begin{align}
  R_e(0)\beta_j(\bm{s}) &= \lambda_j(\bm{s}) R(0) \beta_j(\bm{s}),
  \label{eqn:secM-ed}
\end{align} 
where $ \beta_i^{'}(\bm{s}) R_e(0) \beta_j(\bm{s})=1$, if $i=j$, and $0$ otherwise.  



Based on (\ref{eqn:secM-ed}), we compute the coefficient functions by performing a  spectral decomposition of the operator $R^{-1}(0)R_e(0)$ or the operator $R^{-1/2}(0)R_e(0)R^{-1/2}(0)$, provided the orthogonal eigenvectors have norm one. However, the compact covariance operator $R(0)$ is not invertible for functional data. To overcome this limitation, we restrict the operator to a smaller domain as in \citet{,jiang2014inverse}. Note that under the assumption that $E(\Vert X(\bm{s}) \Vert_{L_2}^4) < \infty$, $R(0)$ is a self-adjoint, positive semidefinite, and Hilbert-Schmidt operator. Therefore, there exists an orthonormal basis $\{\pi_i(\cdot) \}_{i=1}^{\infty}$ in $L^2(\mathcal{I})$ such that $X(\bm{s};t)$ can be expressed as
\begin{align*}
  X(\bm{s};t) &= \mu(t) + \sum_{j=1}^{\infty} A_j(\bm{s}) \pi_j(t),
 \end{align*}
 where $\mu(t)=E(X(\bm{s};t))$ is the mean function, and $\{A_j(\bm{s})\}$ are assumed to be second-order stationary and uncorrelated random fields with $E(A_j(\bm{s}))=0$ and $E(A_j^2(\bm{s}))=\xi_j$. Here, $\xi_j's$ are the eigenvalues of 
\begin{align*}
  R(0; t_1, t_2) &= \sum_{i=1}^{\infty} \xi_i \pi_i(t_1) \pi_i(t_2),
\end{align*}
satisfying $\xi_1 \ge \xi_2 \ge \ldots \ge \xi_j \ge \cdots \ge 0$ with corresponding eigenfunctions $\pi_j(\cdot)$.
By employing the arguments analogous to \citet{,jiang2014inverse}, we  can show that when the following condition is satisfied,
\begin{align}
\sum_{i,j=1}^{\infty} \frac{E^2\{E(A_i(\bm{s})|Y(\bm{s}))E(A_j(\bm{s})|Y(\bm{s}))\}}{\xi_i^2\xi_j} < \infty,
\label{eqn:secM-sufc}
\end{align}
then $R^{-1/2}(0)R_e(0)R^{-1/2}(0)$ is well-defined on the range space of $R^{1/2}(0)$. This means for any $\mathfrak{f} \in \mathfrak{R}_{R^{1/2}(0)}$ with
\begin{align*}
  \mathfrak{R}_{R^{1/2}(0)} :=\left\{\mathfrak{f} \in L_2(\mathcal{I}): \sum_{i}\xi^{-1}|\langle \mathfrak{f},\pi_i \rangle|^2 < \infty, \mathfrak{f} \perp \text{ker}(R(0))
   \right\},
\end{align*}
it follows that
\begin{align*}
  R^{-1/2}(0) \mathfrak{f} = \sum_{i=1}^{\infty} \xi^{-1/2} \langle \mathfrak{f}, \pi_i \rangle \pi_i.
\end{align*}
While this solves the theoretical difficulty, often the estimation of $ R^{-1/2}(0)$ requires some regularization \citep{ferre2003functional, jiang2014inverse}. Let $\widetilde{R}^{-1/2}(0)= R^{-1/2}(0)|_{\mathfrak{R}_{ R^{1/2}(0)}}$ be the inverse operator defined on the range space of $ R^{1/2}(0)$. The directions we obtain from $\widetilde{R}^{-1/2}(0)R_e(0) \widetilde{R}^{-1/2}(0)$ are still in the e.d.r. space because  $\mathfrak{R}_{R^{1/2}(0)}$ is a subspace of $L_2(\mathcal{I})$. 

Let $R^{-1/2}(0) R_e(0) R^{-1/2}(0)(t_1,t_2)=\sum_{i=1}^{\infty}\lambda_i \eta_i(t_1)\eta_i(t_2)$. The arguments analogous to \citet{jiang2014inverse} can be used to show that this is a well-defined operator given (\ref{eqn:secM-sufc}). Furthermore, the e.d.r. directions $\beta_j=R^{-1/2}\eta_j$, for $j=1,\ldots,K$, are well-defined.

\section{Sampling plan and methodology} \label{sec:sample-method}
This section discusses the sampling plan for spatial sites and for functional data before outlining the procedure for estimating the coefficient functions in model (\ref{eqn:secM-i1}).

\subsection{Domain-expanding infill framework}

In general, there are two different paradigms for asymptotics (or the data collection) in spatial statistics: expanding domain asymptotics, where the domain can be expanded while maintaining at least roughly constant the distance between two neighboring observations (see, e.g., \citet{quenouille1949approximate,matern2013spatial,dalenius1961plane}), and infill or fixed domain asymptotics, where the domain is fixed but the distance between neighboring observations may tend to zero (see, e.g., \citet{traub1998complexity,novak2006deterministic}). The estimators may achieve consistency under expanding domain asymptotics, but they may not do so under fixed domain asymptotics, which is a noteworthy distinction between these two paradigms. Despite these differences, each of these frameworks for spatial statistics has a large body of research; for references, see \citet{lu2014nonparametric}. Similar sampling frameworks are also taken into account for functional data (see, for example, \citet{delicado2010statistics, chouaf2012functional,hormann2013consistency, laksaci2013spatial}).

In this study, the domain-expanding infill (DEI) framework by \citet{,lu2014nonparametric} is taken into consideration. The traditional infill and domain-expanding asymptotics have disadvantages. While the former cannot guarantee the estimators to be consistent as showed by \citet{,lahiri1996inconsistency} and \cite{,zhang2004inconsistent} the latter has a less attractive assumption, at least in some applications, that the distance between neighboring observations does not tend to zero. The DEI asymptotics framework overcomes these drawbacks and is defined as
\begin{align}
  \delta_n =\underset{1 \le j\le n}{\max} \delta_{j,n} \rightarrow 0, \text{ with } 
  \delta_{j,n}= \min \{\Vert \bm{s}_i-\bm{s}_j \Vert: 1 \le i \le n, i \neq j \},
  \label{eqn:sc-1}
\end{align}
which means that the distance between neighboring observations tends to $0$, as $n \rightarrow \infty$, and
\begin{align}
  \Delta_n = \underset{1 \le j\le n}{\min} \Delta_{j,n} \rightarrow \infty, \text{ with }  \Delta_{j,n} = \max \{\Vert \bm{s}_i-\bm{s}_j \Vert: 1 \le i \le n, i \neq j \},
  \label{eqn:sc-2}
\end{align}
which implies that the domain at each location is expanding to $\infty$, as $n \rightarrow \infty$, where $\Vert \cdot \Vert$ is the usual Euclidean norm in $\mathbb{R}^2$. The existing frameworks may be described using the above terminology. The traditional infill asymptotics require (\ref{eqn:sc-1}) but $\max_{1 \le j \le n}\Delta_{j,n} \le C_0 < \infty$, while the traditional domain-expanding asymptotics require (\ref{eqn:sc-2}) but $\min_{1 \le j \le n} \delta_{j,n} \ge c_0 >0$ for all $n$. By (\ref{eqn:sc-1}), it is clear that the sampling locations $\bm{s}_i$ are non-stochastic. There are also some related works of stochastic sampling designs (\cite{zhang2022unified} and \citet{kurisu2022nonparametricb} and the references therein) combining both infill and increasing domain. 

In many situations, the DEI framework may be natural as a result of the data structure. For example, socio-economic data are often collected over individual cities and suburbs: on one hand, more cities or suburbs are taken into account which means expanding the domain; on the other hand, more observations are collected within each city or suburb, necessitating the use of an infill asymptotic framework.

\subsection{Methodology}

Assume we observe data at locations $\{\bm{s}_i \in \mathcal{D}_n, i=1, \ldots, n \}$ that are allowed to be irregularly positioned. Suppose we observe
\begin{align}
    Z(\bm{s}_i; T_{ij}) &= X(\bm{s}_i; T_{ij}) + U_i(T_{ij}) +  e_{ij}, \qquad  i=1,\ldots,n, j=1,\ldots,N_i, \label{eqn:x-meas-err}
\end{align}
where 
\begin{align*}
X(\bm{s}; t) &=  \mu(t) + V(\bm{s}, t),
\end{align*}
 with $V(\bm{s}, t)$ being the stochastic part of $X(\bm{s}; t)$ with $EV(\bm{s}, t)=0$ and $e_{ij}$'s are i.i.d. copies of a random error $e$ with mean zero and finite variance $\sigma^2$, and $ U_i(t)$ is a mean zero temporal process called as functional nugget effect. Assume that $V(\cdot,\cdot)$, $U_i(\cdot)$, and  $e$  are mutually independent.  As mentioned in \citet{zhang2022unified}, the functional nugget effects $U_i(\cdot)$ specify local variations that are not correlated with neighbor functions. We denote the covariance between nugget effects as $\Lambda(t_1,t_2)=\text{cov}\{ U_i(t_1), U_i(t_2)\}$. The observation time points $\{T_{ij}\}$ are i.i.d. copies of a random variable $T$ which is defined on the compact interval $\mathcal{I}$.  We consider the local linear smoothing \citep{fan2018local} approach for the estimation of e.d.r. directions. For convenience in notation, we denote $Z_{ij}:=Z(\bm{s}_i; T_{ij})$, $X_{ij}:=X(\bm{s}_i; T_{ij})$, $Y_i=Y(\bm{s}_i)$, $V_{ij}= V(\bm{s}_i; T_{ij})$, and $U_{ij} := U_i(T_{ij})$.

In the following, we discuss the estimation of the required covariance matrices.
%

\vspace{0.5em}
\textbf{Estimation of $R(0; t_1,t_2)$: } 
Denote $K_h(u)=h^{-1}K_1(u)$ where $K_1(\cdot)$ is a univariate kernel function defined in Condition (III) below. Note that the covariance  $R(0; t_1,t_2)= \text{cov}\{X(\bm{s}; t_1),~ X(\bm{s}; t_2)\}$. We first estimate the mean function $\mu(t):= E(X(\bm{s}; t))$ using a one-dimensional local linear smoother applied to the pooled data
  $(Z_{ij},T_{ij})$, $1\leq i\leq n$, $1\leq j\leq N_i$. Therefore, the mean estimator is $\widehat{\mu}(t)=\widehat{\alpha}_0$ with
  \begin{align}
(\widehat{\alpha}_0, \widehat{\alpha}_1) &= \underset{(\alpha_0,\alpha_1)\in \mathbb{R}^2}{\text{argmin}} \sum_{i=1}^n w_i\sum_{j=1}^{N_i}\{Z_{ij}-\alpha_0-\alpha_1(T_{ij}-t) \}^2K_{h_{\mu}}\left(T_{ij}-t\right).
  \end{align}
Here $w_1,\ldots, w_n$ are weights satisfying $\sum_{i=1}^n w_iN_i=1$. The following two weighting schemes are commonly used in the literature \citep{yao2005functional, li2010uniform, zhang2016sparse}. 
\begin{enumerate}[label=(\alph*)]
    \item Equal-weight-per-observation (OBS): $w_i=1/N_{S,0}$ with $N_{S,0}=\sum_{i=1}^nN_i$;
    \item Equal-weight-per-subject (SUBJ): $w_i=1/nN_i$.
\end{enumerate}
As suggested by \cite{zhang2016sparse}, using a weighing scheme that is neither OBS nor SUBJ could achieve a better rate. Related works include a mixed scheme of OBS and SUBJ in \cite{zhang2016sparse} and a bandwidth-dependent scheme in \cite{zhang2018optimal}. For simplicity, we focus on the OBS and SUBJ schemes.


We assume $N_i \ge 2$ to estimate  $R(0;t_1,t_2)$. Let $L_b(u)=b^{-1}L(u/b)$ where $L(\cdot)$ is a kernel function defined similar to $K_1(\cdot)$ with bandwidth $b$. We apply a three-dimensional local linear smoother to the cross-products  $C_{ij,i'j'} :=[Z_{ij}-\widehat{\mu}(T_{ij})][Z_{i'j'}-\widehat{\mu}(T_{i'j'})]$, $1\leq i,i' \leq n$, $1\leq j \leq N_i$, $1 \le j' \le N_{i'}$,  and attach weight $v_{i,i'}$ to each $C_{ij,i'j'}$ for the $(i,i')$th subject pair such that $\sum_{i=1}^n \sum_{i' \neq i} N_iN_{i'} v_{i,i'}=1$. 
Specifically, we consider the estimator $\widehat{R}(\Vert \bm{s}_0 \Vert;t_1,t_2)=\widehat{\alpha}_0$ where 
\begin{eqnarray}
 &&(\widehat{\alpha}_0,\widehat{\alpha}_1,\widehat{\alpha}_2,\widehat{\alpha}_3) \nonumber \\
  &=& \underset{(\alpha_0,\alpha_1,\alpha_2,\alpha_3)\in
    \mathbb R^4}{\text{arg min}}\sum_{i=1}^{n} \sum_{i' \neq i} v_{i,i'} \sum_{j=1}^{N_i} \sum_{j'=1}^{N_{i'}}  \nonumber \\   && \qquad \bigg\{
    C_{ij, i'j'}-\alpha_0-\alpha_1(T_{ij}-t_1)-\alpha_2(T_{i'j'}-t_2) - \alpha_3(\Vert \bm{s}_{i}-\bm{s}_{i'} \Vert- \Vert \bm{s}_0 \Vert) \bigg\}^2 \nonumber
  \\ 
   && \qquad \qquad \times K_{h_c}\left(T_{ij}-t_1\right) K_{h_c}\left( T_{i'j'}-t_2\right) L_b(\Vert \bm{s}_i-\bm{s}_{i'} \Vert- \Vert \bm{s}_0 \Vert).
   \label{eq:obj-gs}
\end{eqnarray}

Similar to \citet{zhang2016sparse}, we consider weights $v_{i,i'}=1/ \sum_{i=1}^n \sum_{i' \neq i} N_iN_{i'}$ for the OBS scheme and weights $v_{i,i'}=1/[n(n-1)N_iN_{i'}]$ for the SUBJ scheme. 
 Our asymptotic theory regards $N_i$, $w_i$ and $v_{i,i'}$ as fixed quantities that are allowed to vary over $n$. For random $N_i$, the theory proceeds as if conditional on the values of $N_i$. The asymptotic properties of $\widehat{R}(\Vert \bm{s}_0 \Vert; t_1, t_2)$ are discussed in Theorem \ref{th:Rhat} and Lemma \ref{lm:L2-R} below.

\vspace{0.5em}

\textbf{ Estimation of $\Gamma(t_1,t_2)$ and $\Lambda(t_1,t_2)$:} Because of the independence between $X(\bm{s};t)$ and the functional nugget effect $U(t)$, we obtain 
\begin{align*}
  \Gamma(t_1,t_2) := \text{cov}\{X(\bm{s};t_1), X(\bm{s};t_2) \} = R(0; t_1,t_2) + \Lambda(t_1,t_2), 
\end{align*}
for $t_1 \neq t_2$. This motivates us to apply a two-dimensional local linear smoother to the cross-products  $C_{ijk} :=[Z_{ij}-\widehat{\mu}(T_{ij})][Z_{ik}-\widehat{\mu}(T_{ik})]$, $1\leq i\leq n$, $1\leq j \neq k\leq N_i$ to estimate  $\Gamma(t_1,t_2)$ and attach weight $v_i$ to each $C_{ijk}$ for the $i$th subject such that $\sum_{i=1}^nN_i(N_i-1)v_i=1$. 
Specifically, the estimator $\widehat{\Gamma}(t_1,t_2)=\widehat{\alpha}_0$ is considered where 
\begin{eqnarray}
  \nonumber (\widehat{\alpha}_0,\widehat{\alpha}_1,\widehat{\alpha}_2)
  = \underset{(\alpha_0,\alpha_1,\alpha_2)\in
    \mathbb R^3}{\text{arg min}}\sum_{i=1}^{n} v_i \sum_{1 \le k \neq j \le N_i}   \left\{
    C_{ijk}-\alpha_0-\alpha_1(T_{ij}-t_1)-\alpha_2(T_{ik}-t_2) \right\}^2
  \\ 
   \times K_{h_c}\left(T_{ij}-t_1\right) K_{h_c}\left( T_{ik}-t_2\right).
   \label{eq:obj-gamma}
\end{eqnarray}
As in \citet{zhang2016sparse}, weights $v_i=1/ \sum_{i=1}^n N_i(N_i-1)$ yield the OBS scheme and weights $v_i=1/[nN_i(N_i-1)]$ yield the SUBJ scheme. Here, a natural estimator for the functional nugget effect is 
\begin{align*}
  \widehat{\Lambda}(t_1,t_2) &= \widehat{\Gamma}(t_1,t_2) - \widehat{R}(0;t_1,t_2).
\end{align*}
The asymptotic properties of $\widehat{\Gamma}(t_1,t_2)$ are discussed in Theorem \ref{thm:rate-phi} below. Given $\widehat{\Gamma}(t_1,t_2)$, we can estimate $\sigma^2$ by proceeding similar to \citet{yao2005functional}. We do not discuss the properties of $\widehat{\sigma}^2$ here as it is not the main focus of our study. 


\vspace{1em}
\textbf{ Estimation of $R_e(0,t_1,t_2)$: }
Let us introduce $K_{H}(\bm{u})=|\bm{H}|^{-1/2}K_2(\bm{H}^{-1/2}\bm{u})$ where $K_2(\cdot)$ is a two-dimensional kernel defined in Condition (III) and $\bm{H}=\text{diag}(h_{t}^2,h_y^2)$ with $h_y$ and $h_t$ are the bandwidths for $Y_i$ and $T_{ij}$, respectively. Let $m(t,y) = E(X(\bm{s}; T_{ij}=t)|Y(\bm{s})=y)$.
The function $m(t,y)$ is assumed to be smooth
since we employ local linear smoothing for estimation.  For estimation of $m(t,y)$, we apply a two-dimensional smoothing method to
the pooled sample $\{Z_{ij}\}$ over $\{T_{ij},Y_i\}$ for $i=1,\ldots,n$, $j=1,\ldots, N_i$. Therefore, we consider the estimator $\widehat{m}(t,y)=\widehat{\alpha}_0$ where
\begin{align} \nonumber
  (\widehat{\alpha}_0,\widehat{\alpha}_1,\widehat{\alpha}_2) &= \underset{(\alpha_0,\alpha_1,\alpha_2)\in \mathbb R^3}{\text{argmin}}\sum_{i=1}^n w_i \sum_{j=1}^{N_i} \left\{ Z_{ij}-\alpha_0-\alpha_1(T_{ij}-t)-\alpha_2(Y_{i}-y) \right\}^2 \\ & \qquad \qquad \qquad \qquad \qquad \qquad \times K_H\left( (T_{ij}-t, Y_i-y)^T \right), \label{eqn:obj-ge}
\end{align}
where the weight, $w_i$ is attached to each observation for the $i$-th subject for $i=1,\ldots,n$, with $\sum_{i=1}^n N_iw_i=1$. Recall that $R_e(0)=\text{cov}(E(X(t)|Y))$. Let $\mathfrak{B}$ be a compact subset of $\mathbb{R}$. We estimate $R_e(0)$  empirically as follows:
\begin{align*}
\widehat{R}_e(0,t_1,t_2) &= \frac{1}{n}\sum_{i=1}^n \widehat{m}(t_1,Y_i) \widehat{m}(t_2,Y_i) \mathbbm{1}_{\{ Y_i \in \mathfrak{B}\}}-\frac{1}{n}\sum_{i=1}^n \widehat{m}(t_1,Y_i)\mathbbm{1}_{\{ Y_i \in \mathfrak{B}\}} \frac{1}{n}\sum_{j=1}^n \widehat{m}(t_2,Y_j)\mathbbm{1}_{\{ Y_j \in \mathfrak{B}\}},
\end{align*}
where $\mathbbm{1}_{\{ Y_j \in \mathfrak{B}\}}$ is an indicator function that takes value one if $Y_j \in \mathfrak{B}$ and zero otherwise. In practice, we may choose $\mathfrak{B} = [y_{\alpha}, y_{1-\alpha}]$ where $y_{\alpha}$ is the $\alpha$-th percentile of $Y_i$ values. In general, we can take $\alpha$ equal to  5\% or 10\%. The asymptotic properties of the estimator $\widehat{m}(t,y)$ and $\widehat{R}_e(0)$ are discussed in Theorem \ref{th:1} and Lemma \ref{lm:2} below.

\vspace{1em}
\textbf{ Estimation of the e.d.r. directions $\beta_j$ for $j=1,\ldots,K$ and the link function $f(\cdot)$: }
We calculate the e.d.r. directions using the eigen-analysis of  $R^{-1/2}(0) R_e(0) R^{-1/2}(0)$. Let $\widehat{\eta}_j$ be the estimator of the standardized e.d.r. direction $\eta$ and $\widehat{\beta}_j$ be the estimator of $\beta_j$. According to \citet{jiang2014inverse}, the following procedures are used to perform an eigen-analysis on a discrete grid in practice:

\begin{enumerate}[label=(\alph*)]
  \item on an equally-spaced grid $\{t_1,\ldots,t_p\}$, compute  the $p \times p$ matrices $\widehat{\bm{R}}_{e}$ and $\widehat{\bm{R}}$,
  \item obtain $\widehat{\bm{R}}^{-1/2}$ from the singular value decomposition of $\widehat{\bm{R}}$,
  \item obtain the first $K$  eigenfunctions $\widehat{\bm{\eta}}_1,\ldots,\widehat{\bm{\eta}}_K$  corresponding to the largest $k$ eigenvalues $\widehat{\lambda}_1,\ldots,\widehat{\lambda}_K$ from the eigen-analysis of $\widehat{\bm{R}}^{-1/2}\widehat{\bm{R}}_{e}\widehat{\bm{R}}^{-1/2}$,
  \item compute the e.d.r. directions $\widehat{\bm{\beta}}_{j,p} = \widehat{\bm{R}}^{-1/2}\widehat{\bm{\eta}}_j$, for $j=1,\ldots,K$.
 \end{enumerate}
 The link function $f(\cdot)$ is estimated, assuming $K$ is significantly smaller than $p$, by using standard nonparametric smoothing techniques to the indices $\langle \beta_{j}(t), Z_i(t) \rangle$ for $j=1,\ldots,K$,  where $Z_i(\cdot)$ is the observed function for the $i$th location, $i=1,\ldots,n$. The traditional nonparametric methods, however, suffer from the curse of dimensionality when $K$ is not small. In this situation, we may need to use methods like Multi-layer Neural Network \citep{ferre2005smoothed}, which are not affected by the curse of dimensionality.

\section{Asymptotic results} \label{sec:asymp}
In this section, we derive the convergence rates for the covariance operators and the e.d.r. directions. Standard smoothing assumptions are made on the second derivatives of $R$ and $R_e$ since their estimators $\widehat{R}$ and $\widehat{R}_e$ are obtained by employing local linear smoothing. 

We denote $\bm{T}_i=(T_{i1},\ldots, T_{iN_i})^T$ and $\bm{Z}_i=(Z_{i1},\ldots,Z_{iN_i})^T$ for $i=1,\ldots,n$. Let $Y_{i}=Y(\bm{s}_i)$ for $i=1,\ldots,n$. For convenience, we summarize the main assumptions here. Conditions (I) and (II) are related to the random field itself. 

\vspace{1em}
\textbf{Condition (I) (spatial process):}
\begin{enumerate}[label=(\roman*)]
    \item The random field $\{(Y_i, \bm{T}_i, \bm{Z}_i) : \bm{s}_i \in \mathcal{D}_n, i=1,\ldots,n\}$ is strictly stationary. 
    
    \item The time points $\{T_{ij}: i=1,\ldots,n, j=1,\ldots, N_i\}$, are i.i.d. copies of a random variable $T$ defined on $\mathcal{I}$, say $[0,1]$. The density $f_T(\cdot)$ of $T$ is bounded away from zero and infinity. The second derivative of $f_T(\cdot)$, denoted by $f^{(2)}_T(\cdot)$, is bounded.
    \begin{enumerate} [label=(\alph*)]
    \item $V$, $U_i$ and $e$ are mutually independent, and are independent of $T$. 
    \end{enumerate}
    
   \item  Assume that $\mu^{(2)}(t)$ exists and bounded on $\mathcal{I}$.

   \item 
   \begin{enumerate} [label=(\alph*)]
       \item For all $1 \le i \neq i' \le n$, the vectors $(T_{ij}, Y_{i})$ and $(T_{i'j'}, Y_{i'})$, $1 \le j \le N_i$, $1 \le j' \le N_{i'}$, admit a joint density $g_{i,i'}$ and a marginal density $g$; further, $\vert g_{i,i'}(t,y,t',y')-g(t,y) g(t',y') \vert \le C$ for $t,t',y,y' \in \mathbb{R}$, where $C>0$ is some constant. Let $g_3(t,t',y)$ be the joint density of $(T_{ij}=t,T_{ij'}=t',Y_i=y)$. We assume $g$, $g_3$ and $m(t,y)$ have continuous and bounded second derivatives.
       \item  The data points $Y_i$'s are independent of $U_{ij}$ and $e_{ij}$. Let $\Psi(t,y)$ be the variance of $X_{ij}$ given $T_{ij}=t$ and $Y_i=y$ and $\Psi(t,y,t')$ be the covariance of $X_{ij}$ and $X_{ij'}$ given $T_{ij}=t$, $T_{ij'}=t'$ and $Y_{i}=y$. We assume that  $\Psi$ has continuous and bounded second derivatives. Let $\Lambda(t)$ be the variance of $U_{ij}$ given $T_{ij}$ and $\Lambda(t_1,t_2)$ be the covariance of $U_{ij}$ and $U_{ik}$ given $T_{ij}$ and $T_{ik}$. We assume $\Lambda$ has continuous and bounded second derivatives.
       \item $\sup_{\bm{s}}E[\sup_t|V(\bm{s},t)|^{2+\varrho}] < \infty$, $E[\sup_t |U(t)|^{2+\varrho}]< \infty$, and $E|e|^{2+\varrho}< \infty$ for some $\varrho>0$. 
   \end{enumerate}

   \item 
   \begin{enumerate}[label=(\alph*)]
     \item  Let $\Xi(t_1,t_2,t_1',t_2')$ be the covariance of $[X_{ij}-\mu(T_{ij})]$ $[X_{ik}-\mu(T_{ik})]$ and $[X_{ij'}-\mu(T_{ij'})]$ $[X_{ik'}-\mu(T_{ik'})]$ and $\Lambda_1(t_1,t_2,t_1',t_2')$ be the covariance of $U_{ij}U_{ik}$ and $U_{ij'}U_{ik'}$ given $T_{ij}=t_1$, $T_{ik}=t_2$, $T_{ij'}=t_1'$ and $T_{ik'}=t_2'$. We assume  $\Xi$,  $\Lambda_1$ and $\Gamma(t_1,t_2)$  have continuous and bounded second derivatives. 
     \item $\sup_{\bm{s}}E[\sup_t|V(\bm{s},t)|^{4+\varrho_c}] < \infty$, $E[\sup_t|U(t)|^{4+\varrho_c}]< \infty$, and $E|e|^{4+\varrho_c}< \infty$ with $\varrho_c/2= \varrho > 0$.
   \end{enumerate}
   \item \begin{enumerate}
     \item  Let $\Xi_R(t_1,t_2,t_1',t_2'; \Vert \bm{s}_i-\bm{s}_{i'} \Vert)$ be the covariance of $[Z_{ij}-\mu(T_{ij})]$ $[Z_{i'j'}-\mu(T_{i'j'})]$  and  $[Z_{ik}-\mu(T_{ik})]$ $[Z_{i'k'}-\mu(T_{i'k'})]$ given $T_{ij}=t_1$, $T_{i'j'}=t_2$, $T_{ik}=t_1'$ and $T_{i'k'}=t_2'$ for $1 \le i \neq i' \neq n$ and $1 \le j,k \le N_i$, $1 \le j',k' \le N_{i'}$. We assume that $\Xi_R$ and $R(s;t_1,t_2)$ have continuous and bounded second derivatives.    
   \end{enumerate}
  
\end{enumerate}
Conditions (I)[(i), (ii)(a), (iii)(a)] are standard for the spatial data generating process in the context of the problem under study. For example, similar conditions are used by \citet{masry1986recursive} in time series, by \citet{tran1990kernel} and \citet{hallin2004local} in the gridded spatial data context, and by \citet{,lu2014nonparametric} for irregular spatial measurements. These conditions are required to ensure uniform consistency for joint probability densities across different spatial sites.   Conditions (I)[(ii)(b), (iii)(b)] are mild assumptions on the covariance of the function. They are similar to Assumption (A.6) in \citet{jiang2014inverse} and needed for local linear smoothing.  Conditions (I)[(ii)(c), (iii)(c)] are typical in the spatial regression setup. For example, they are similar to  (A2) of \citet{hallin2004local}, Assumption (3.2) in \citet{gao2006estimation} and Conditions (C5)-(C7) in \citet{li2010uniform}. These conditions are required for Lemma S.1.1 in the supplement \citet{suppl}. Condition (iv) is similar to Assumption (I)(iii) in \citet{lu2014nonparametric} which is required for the spatio-temporal covariance.
Condition (II) provides the assumptions on spatial mixing as well as spatial sites, which include irregularly positioned observations.
 Let $\mathcal{B}(\mathcal{S})$ be the Borel $\sigma$-field generated by $\{(Y_i, \bm{T}_i, \bm{Z}_i ) : \bm{s}_i \in \mathcal{S}, i=1,\ldots,n\}$ for any collection of sites $\mathcal{S} \subset \mathbb{R}^2$. For each couple $\mathcal{S}'$, $\mathcal{S}''$, let $d(\mathcal{S}', \mathcal{S}''):= \text{min}\{\Vert \bm{s}'-\bm{s}'' \Vert : \bm{s}' \in \mathcal{S}', \bm{s}'' \in \mathcal{S}''\}$ be the distance between $\mathcal{S}'$ and $\mathcal{S}''$. Finally, let Card$(\mathcal{S})$ be the cardinality of $\mathcal{S}$.

\vspace{1em}
\textbf{Condition (II) (mixing and sampling sites):}

\begin{enumerate}[label=(\roman*)]
\item   There exist a function $\phi$ such that $\phi(t) \rightarrow 0$ as $ t\rightarrow \infty$, and a symmetric function $\psi: \mathbb{N}^2 \rightarrow \mathbb{R}^+$ decreasing in each of its two arguments such that
  \begin{align}
    \alpha\left(\mathcal{B}(\mathcal{S'}),\mathcal{B}(\mathcal{S''})\right) &:= \sup\left\{|P(AB)-P(A)P(B)|, A\in \mathcal{B}(\mathcal{S'}), B\in \mathcal{B}(\mathcal{S''}) \right\} \nonumber\\ &\le \psi(\text{Card}(\mathcal{S'}),\text{Card}(\mathcal{S''})) \phi(d(\mathcal{S'},\mathcal{S''})),
    \label{eqn:mixing-coef}
  \end{align}
  for any $\mathcal{S'}, \mathcal{S''} \subset \mathbb{R}^2$. Moreover, the function $\phi$ is such that
  \begin{align}
    \underset{m \rightarrow \infty}{\lim}m^{\kappa} \sum_{j=m}^{\infty} j^{2} \{\phi(j)\}^{\varrho/(2+\varrho)} &=0
    \label{eqn:mixing-lim}
  \end{align} 
  for some constant $\kappa>$ and  for the constant $\varrho$ defined in Condition (I)[(iv), (v)]. In this study, we consider $\phi(j)\lesssim j^{-\rho}$ for some constant $\rho>3(2+\varrho)/\varrho$ and $ \psi(i,j) \lesssim \min(i,j)$. 
  \end{enumerate}
  
The observations are positioned as $\{\bm{s}_i, i=1,\ldots,n \}$ for which (\ref{eqn:sc-1}) and (\ref{eqn:sc-2}) hold true with $\underset{1 \le j\le n}{\min} \delta_{j,n}/\delta_n \ge c_1 >0$ and $\underset{1 \le j \le n}{\max} \Delta_{j,n}/\Delta_n \le C_1 < \infty$ for all $n$. Suppose  $g_{\bm{S}}$ is the  spatial density function (or sampling intensity function) defined on $\mathbb{R}^2$ such that:
\begin{enumerate}[label=(\roman*)]
\setcounter{enumi}{1}
\item  $  \sum_{i=1}^n \mathbb{W}_i \bm{I}(\bm{s}_i \in A) \rightarrow \int_A g_{\bm{S}}(\bm{s})d\bm{s}$ as $n \rightarrow \infty$, for any measurable set $A \subset \mathbb{R}^2$ and for some weights $\mathbb{W}_i>0$ and $\sum_{i=1}^n \mathbb{W}_i=1$.
\item  The spatial density $g_{\bm{S}}(\bm{s})$ is bounded and has second derivatives that are continuous on $\mathbb{R}^2$.
\item The following two integrals exist:
\begin{align*}
  A_0^*(\Vert \bm{s}_0 \Vert) &= \Vert \bm{s}_0 \Vert \int_{-\infty}^{\infty} \int_{-\infty}^{\infty} \bigg[  \int_{0}^{2\pi} g_{\bm{S}}(u+ \Vert \bm{s}_0 \Vert \cos(\vartheta), v + \Vert \bm{s}_0 \Vert \sin(\vartheta)) d\vartheta  \bigg] \\ & \qquad \qquad \times g_{\bm{S}}(u,v)du dv > 0, 
\end{align*}
and 
\begin{align*}
  A_1^*(\Vert \bm{s}_0 \Vert) &= \Vert \bm{s}_0 \Vert \int_{-\infty}^{\infty} \int_{-\infty}^{\infty} \bigg[  \int_{0}^{2\pi} (\cos(\vartheta), \sin(\vartheta)) \\& \qquad \times \bigg[\frac{ \partial g_{\bm{S}}(u+ \Vert \bm{s}_0 \Vert \cos(\vartheta), v + \Vert \bm{s}_0 \Vert \sin(\vartheta))}{\partial \bm{s}} d\vartheta  \bigg] g_{\bm{S}}(u,v)du dv > 0. 
\end{align*}
\end{enumerate}
The first part of Condition (II)(i), the $\alpha$-mixing condition, is standard for spatial data. For example, it is similar to (A4) in \citet{hallin2004local} and Assumption I(i) in  \citet{lu2014nonparametric}. This condition characterizes spatial dependence. If the condition (\ref{eqn:mixing-coef}) holds with $\psi=1$, then the random field $\{Y_i, \bm{T}_i, \bm{Z}_i \}$ is called strongly mixing. Many time series and stochastic processes are shown to be strongly mixing, at least, in the serial case $d=1$. The last part of Condition (II)(i) is  similar to Assumption 2 in \citet{hansen2008uniform} and Assumption E2 in \citet{kurisu2022nonparametric} and needed for uniform convergence. Conditions (II)[(ii), (iii)] are from \citet{lu2014nonparametric}. These are mild conditions on the spatial sites and are needed  for the derivation of asymptotic bias and variances of the concerned estimators. Condition (II)(iv) is also from \citet{lu2014nonparametric} which assumes that the joint probability density is isotropic.

Condition (III) deals with the kernel functions that are used for estimating both covariance operators. These are similar  to Assumption (A.3) of \citet{jiang2014inverse}.

\vspace{1em}
\textbf{Condition (III) (kernel functions):}

\begin{enumerate}[label=(\roman*)]
  \item Let $K_2(\cdot,\cdot)$ be a bivariate kernel function that is compactly
      supported, symmetric and B-Lipschitz. Further, it is a kernel of
      order $(|b_1|,|b_2|)$, that is,
      \begin{eqnarray*}
        \sum_{\ell_1 +\ell_2=\ell} \int \int u^{\ell_1}v^{\ell_2}K_2(u,v) du dv
        = \Bigg\{ \begin{array}{ll} 0, & 0 \le \ell < |b_2|, \ell \neq |b_1|,\\
            (-1)^{|b_1|} |b_1|!, & \ell = |b_1|,\\
                    \neq 0, & \ell= |b_2|.
                             \end{array} 
      \end{eqnarray*}
      Similarly, the univariate kernel function $K_1$ of order $(c_1, c_2)$ is compactly
      supported, symmetric and B-Lipschitz,  and
      defined as
 \begin{eqnarray*}
        \int u^{\ell}K_1(u) du
        = \Bigg\{ \begin{array}{ll} 0, & 0 \le \ell < c_2, \ell \neq c_1,\\
            (-1)^{c_1} c_1!, & \ell = c_1,\\
                    \neq 0, & \ell= c_2.
                             \end{array} 
      \end{eqnarray*}
We assume that  $\int |K_1(u)|^{2+\varrho} du < \infty$, $\int |u K_1(u)|^{2+\varrho} du < \infty$, $\int |K_2(u,v)|^{2+\varrho} dudv < \infty$, $\int |u^2 K_2(u,v)|^{2+\varrho} dudv < \infty$, and $\int |v^2 K_2(u,v)|^{2+\varrho} dudv < \infty$. In this study, we set $|b_1|=0$, $|b_2|=2$ for $K_2$, and $c_1=0, c_2=2$ for $K_1$.

      \item The kernel function $L(\cdot)$ satisfies that $\int L(s)ds=1$, $\int sL(s) ds=0$, and $\mu_{L,2}=\int s^2 L(s) ds < \infty$ and $\nu_{L,2}=\int L^2(s)ds < \infty$. Moreover, it is Lipschitz continuous. 
    \end{enumerate}
      %

 
Before we state the conditions on bandwidths, define
\begin{align}
       \gamma_{n1} &= n  h^{-2}  \text{max}^2\{w_i N_i\}^2, 
       \label{eq:def-gamn1}\\
       \gamma_{n2} &= n  h^{-2}  \text{max}^2\{v_i N_i(N_i-1)\}, \text{ and }
       \label{eq:def-gamn2}\\
       \gamma_{n3} &= n(n-1)^2 h^{-2}  \text{max}^2\{ v_{i,i'} N_i N_{i'} \}.
       \label{eq:def-gamn3}
\end{align}
 Let $\overline{N}_1=n^{-1}\sum_{i=1}^n N_i(N_i-1)$ and $\overline{N}_2 = n^{-1}(n-1)^{-1}\sum_{i=1}^n \sum_{i' \neq i}N_i N_{i'}$. For the SUBJ and OBS weighting schemes, we obtain $\gamma_{n1}= 1/nh^2$ and $\gamma_{n1}= (nh^2)^{-1}\text{max}^2\{N_i\}/\overline{N}^2$. Similarly, we obtain $\gamma_{n2}= 1/nh^2$ for SUBJ scheme and $\gamma_{n2}=(nh^2)^{-1}\text{max}^2\{N_i(N_i-1)\}/\overline{N}_1^2$ for OBS scheme. Moreover, for SUBJ scheme $\gamma_{n3}=1/nh^2$ and for OBS scheme $\gamma_{n3}= (nh^2)^{-1}\text{max}^2\{N_i N_{i'} \}/ \overline{N}_2^2$. We now provide conditions related to bandwidths.  

\vspace{1em}
    \textbf{Condition (IV) (bandwidth):}
    \begin{enumerate}[label=(\roman*)]
      \item  Assume that the bandwidths $h_t =O(h)$ and $h_y=O(h)$. The bandwidth $h$ satisfies $h \rightarrow  0$, $nh^2\delta_n^2 \rightarrow \infty$, $\sum_{i=1}^n N_iw_i^2/h^2 \rightarrow 0$, $\sum_{i=1}^n N_i(N_i-1)w_i^2/h \rightarrow 0$  and 
      \begin{align}
      \begin{split}
      \underset{n \rightarrow \infty}{\lim} \gamma_{n1}^{1+2/\kappa} \delta_n^{-2(1+2/\kappa)} h^{2\left(1-\frac{2\varrho}{\kappa(2+\varrho)}\right)} &= 0, \text{ for }  \kappa > \max\{1, 2\varrho/(2+\varrho) \}, \text{ and }\\
      \underset{n \rightarrow \infty}{\lim} \gamma_{n1}^{1+2/\kappa} \delta_n^{-2(1+2/\kappa)} h^{\frac{2\varrho}{4+\varrho}\left(1-\frac{2(4+\varrho)}{\kappa(2+\varrho)} \right)} &= 0, \text{ for }  \kappa > 2(4+\varrho)/(2+\varrho),
      \end{split}
      \label{eqn:bw-c1}
      \end{align}
      where $\kappa$ and $\varrho>0$ defined in (\ref{eqn:mixing-lim}), $\gamma_{n1}$ is defined in (\ref{eq:def-gamn1}), and $\delta_n$ is defined in (\ref{eqn:sc-1}). 
      
      \item The bandwidth $h_{\mu}$ satisfies $h_{\mu} \rightarrow  0$, $nh_{\mu}\delta_n^2 \rightarrow \infty$, $\sum_{i=1}^n N_iw_i^2/h_{\mu} \rightarrow 0$, $\sum_{i=1}^n N_i(N_i-1)w_i^2 \rightarrow 0$, and  (\ref{eqn:bw-c1}) with $h^2$ replaced by $h_{\mu}$. 
      
      \item The bandwidth $h_{c}$ satisfies $h_{c} \rightarrow  0$, $nh_c^2\delta_n^2 \rightarrow \infty$, $\sum_{i=1}^n N_i(N_i-1)v_i^2/h_{c}^2 \rightarrow 0$, $\sum_{i=1}^n N_i(N_i-1)(N_i-2)v_i^2/h_c \rightarrow 0$, $\sum_{i=1}^n N_i(N_i-1)(N_i-2)(N_i-3) v_i^2 \rightarrow 0$,  and  (\ref{eqn:bw-c1}) with $h$ replaced by $h_c$ and $\gamma_{n1}$ replaced by $\gamma_{n2}$. 
      
      \item  The bandwidths $h_{c}$ and $b$ satisfy $h_{c},b \rightarrow  0$, $nh_c^2\delta_n^2 \rightarrow \infty$,  $\sum_{i=1}^n \sum_{i' \neq i} v^2_{i,i'} N_i N_{i'}/h_c^2b \rightarrow 0$, $\sum_{i=1}^n \sum_{i' \neq i} v_{i,i'}^2 N_i N_{i'} (N_i-1) (N_{i'}-1) b^{-1}\rightarrow 0$, $\sum_{i=1}^n \sum_{i' \neq i} v_{i,i'}^2 N_i N_{i'} (N_i-1) h_c^{-1} b^{-1} \rightarrow 0$, $\sum_{i=1}^n \sum_{i' \neq i} v_{i,i'}^2 N_i N_{i'} (N_{i'}-1) h_c^{-1} b^{-1} \rightarrow 0$,  and (\ref{eqn:bw-c1}) with $h$ replaced by $h_c$ and $\gamma_{n1}$ replaced by $\gamma_{n3}$. 
    
    \end{enumerate}

 The first part of each of the Condition (IV) is standard in the nonparametric smoothing literature and is needed to obtain the consistency of estimators. For example, they are similar to (C1a) and (D1a) of \citet{zhang2016sparse}. The conditions  involving $\delta_n$ and $\kappa$ in (\ref{eqn:bw-c1}) are mild and similar to the condition in Assumption (IV) (a) in \citet{lu2014nonparametric} and (A4) in \citet{hallin2004local}. 

\vspace{1em}

\textbf{Condition (V) (uniform convergence)
}
\begin{enumerate}[label=(\roman*)]

 \item Assume $\sup_{i}n w_iN_i < \infty$,
    \begin{align}
    n Q_n^{\frac{2+\varrho}{1+\varrho}}h_{\mu}^{\frac{2+\varrho}{1+\varrho}}  \delta_n^2&\to \infty,\quad  \frac{n^{\rho+1} Q_n^{\frac{(2+\varrho)\rho+4+3\varrho}{1+\varrho}} h_{\mu}^{\frac{(2+\varrho)\rho+6+5\varrho}{1+\varrho}}\delta_n^2\log(n)}{\Delta_n^6}\to \infty,    \label{eq:condition:V:eq:1}
    \end{align}
with
\begin{align*}
    Q_n &= \left\{\sum_{i=1}^n  w_i^2N_i h_{\mu}^{-1}+ \sum_{i=1}^n  w_i^2N_i (N_{i}-1) \right\}^{1/2}.
\end{align*}

\item Assume $\sup_{i}n w_iN_i=O(1)$,  $nh^2\delta_n^2 \to \infty$, $h/\delta_n \rightarrow 0$, and 
\begin{eqnarray}
\frac{n^{1+\rho}h^{14+2\rho}\delta_n^{5+2\rho}}{\Delta_n^6 \log^4(n)}&\to&\infty. \label{eq:condition:V:eq:2}
\end{eqnarray}

\item  Assume $ v_{i,i'}=v_{i',i}$, $\sup_{i,i'} n^2 v_{i,i'}N_i N_{i'} \leq C$ for some constant $C$, and $\frac{nh_c^2b}{\log(n)} \to  \infty$.
\end{enumerate}

The above conditions are mild and sufficient for uniform convergence.  Condition (V) illustrates the roles of $\rho, \delta_n, \Delta_n$, namely, the spatial correlation and DEI parameters. A larger $\rho$ means weaker spatial dependence, and the choice of the bandwidths $h_\mu, h$ becomes more flexible. In the case without spatial correlation, namely $\rho=\infty$, the conditions in (\ref{eq:condition:V:eq:1}) and (\ref{eq:condition:V:eq:2}) are almost equivalent (up to logarithm factors) to
\begin{eqnarray*}
    nQ_nh_\mu \delta_n^2\to \infty,\quad nQ_n^{\frac{2+\varrho}{1+\varrho}}h_\mu^{\frac{2+\varrho}{1+\varrho}}\to \infty, \quad nh^2\delta_n^2\to \infty.
\end{eqnarray*}
The above conditions are similar to Conditions (C2c) and (C3c) from \citet{zhang2016sparse}, except the DEI parameter $\delta_n$ is involved.
 
 We now provide the asymptotic results for the covariance estimators proposed in the study. These results require the uniform convergence of $\widehat{\mu}(t)$ which is provided in Lemma S.1.3 in the supplement \citet{suppl}.
 
 \subsection{Results for $\widehat{R}( 0 ; t_1,t_2)$}
 
 In the following theorem, we provide the asymptotic bias and variance of the spatial covariance estimator $\widehat{R}(\Vert \bm{s}_0 \Vert; t_1,t_2)$.
\begin{theorem} \label{th:Rhat}
  Under conditions (I)[(i), (ii) (iii), (v)(b), (vi)], (II), (III), (IV)[(ii), (iv)] and (V)(i), for a fixed interior point $(\Vert \bm{s}_0 \Vert, t_1, t_2)$, we have that
  \begin{align*}
     E\{ \widehat{R}&( \Vert \bm{s}_0 \Vert; t_1,t_2) - R(\Vert \bm{s}_0 \Vert; t_1,t_2) ~|~ \mathcal{X}_T \} \\ &= \frac{1}{2}h_c^2 \left\{ \frac{\partial^2  R(\Vert \bm{s}_0 \Vert; t_1,t_2)}{\partial t_1^2} + \frac{\partial^2  R(\Vert \bm{s}_0 \Vert; t_1,t_2)}{ \partial t_2^2 } \right\} + \frac{1}{2} b^2 \mu_{L,2}  \frac{\partial^2  R(\Vert \bm{s}_0 \Vert; t_1,t_2)}{\partial \Vert \bm{s}_0 \Vert^2} + o_p(h_c^2+b^2)
       \end{align*}
       and
       \begin{align*}  
    \text{var}&\{ \widehat{R}(\Vert \bm{s}_0 \Vert; t_1,t_2) ~| \mathcal{X}_T\} \\
     &\qquad = \frac{\sum_{i=1}^n \sum_{i' \neq i} v^2_{i,i'} N_i N_{i'}}{h_c^{2}b A_0^* f_T(t_1)f_T(t_2)} \nu_{L,2}  \nu_{K_1,2}^2 ~ \Xi_R(t_1,t_2; \Vert \bm{s}_0 \Vert) \\ & \qquad \qquad + \frac{\sum_{i=1}^n \sum_{i' \neq i} v^2_{i,i'} N_i N_{i'} (N_{i'}-1)}{ h_c b A_0^* f_T(t_1)} \nu_{L,2} \nu_{K_1,2} ~ \Xi_R(t_1,t_2,  t_2; \Vert \bm{s}_0 \Vert)\\
    & \qquad \qquad + \frac{\sum_{i=1}^n \sum_{i' \neq i} v^2_{i,i'} N_i N_{i'} (N_{i}-1)}{h_c b A_0^* f_T(t_2)} \nu_{L,2} \nu_{K_1,2} ~  \Xi_R(t_1,t_2,  t_1; \Vert \bm{s}_0 \Vert)\\
     & \qquad \qquad + \frac{\sum_{i=1}^n \sum_{i' \neq i} v^2_{i,i'} N_i (N_i-1) N_{i'} (N_{i'}-1)}{b A_0^* }\nu_{L,2} \Xi_R(t_1,t_2, t_1, t_2; \Vert \bm{s}_0 \Vert)  +o_p(1),
  \end{align*}
  where $\mathcal{X}_T=\{T_{ij}, i=1,\ldots,n, j=1,\ldots, N_i\}$, $\nu_{L,2}=\int L^2(u)du $,  $\nu_{K_1,2}= \int K_1^2(u)du$ and $\Xi_R$ is defined in Condition (I)(vi), $A_0^*$ is defined in Condition (II)(iv), and $\mu_{L,2}$ is defined in Condition (III)(ii). 
\end{theorem}

The proof of Theorem \ref{th:Rhat} is available in the supplement \citet{suppl}. 
\begin{remark}
The asymptotic bias of the estimator $\widehat{R}( \Vert \bm{s}_0 \Vert; t_1,t_2)$ consists of the biases from two kernel smoothing functions: over the time, $T_{ij}$, and over the values of spatial locations,  $\bm{s}_i$.  It is different from the rate of standard two-dimensional smoothing, which includes only the bias from kernel smoothing over time. As discussed in \citet{lu2014nonparametric}, under the condition $b/h_c \rightarrow 0$ as $n \rightarrow \infty$, the asymptotic bias will simplify to the standard two-dimensional rate $O(h_c^2)$. Similarly, the asymptotic variance is different from the standard rate due to the extra bandwidth $b$. In Remark 4, we give sufficient conditions under which the asymptotic variance will simplify to the standard rate.   
\end{remark}
\begin{remark}
It is a common practice in the literature to provide asymptotic normality of estimator instead of pointwise bias and variance as in Theorem \ref{th:Rhat} when the observations are dependent \citep{fan2008nonlinear}. However, the asymptotic normality of $\widehat{R}( \Vert \bm{s}_0 \Vert; t_1,t_2)$ is very challenging especially under the broad framework such as the one adopted in this manuscript and requires stronger conditions. We leave this for future research.
\end{remark}

We define the Hilbert-Schmidt norm by $\Vert \Phi \Vert_{HS}=[\int \int \Phi(s,t)^2ds dt]^{1/2}$. The following lemma provides the convergence rate of $\widehat{R}(0)$ in Hilbert-Schmidt norm.
\begin{lemma} \label{lm:L2-R}
       Suppose conditions (I)[(i), (ii) (iii), (v)(b), (vi)], (II), (III), (IV)[(ii), (iv)], and (V)[(i), (iii)] hold. Then, we have that
      \begin{align}
        \Vert \widehat{R}(0) &- R(0) \Vert_{HS} = O_p\bigg( h_c^2 + b^2 + \bigg[\frac{\sum_{i=1}^n \sum_{i' \neq i} v^2_{i,i'} N_i N_{i'}}{h_c^{2}b} + \frac{\sum_{i=1}^n \sum_{i' \neq i} v^2_{i,i'} N_i N_{i'} (N_{i'}-1)}{ h_c b} \nonumber\\ & + \frac{\sum_{i=1}^n \sum_{i' \neq i} v^2_{i,i'} N_i N_{i'} (N_{i}-1)}{h_c b} + \frac{\sum_{i=1}^n \sum_{i' \neq i} v^2_{i,i'} N_i (N_i-1) N_{i'} (N_{i'}-1)}{b}\bigg]^{1/2} \bigg).
      \end{align}
\end{lemma}
The proof of the above lemma is available in the supplement \citet{suppl}. In the following corollary, we provide the corresponding rates for the OBS and SUBJ weighting schemes. Prior to that, we introduce some notation. Let $N_S=\sum_{i=1}^nN_i$, $\overline{N}= N_S/n$, $\overline{N}_H=[n^{-1}\sum_{i=1}^n N_i^{-1}]^{-1}$, and
 \begin{align*}
  \overline{N}_{2,12} &= n^{-1}(n-1)^{-1}\sum_{i=1}^n\sum_{i' \neq i}N_i^2 N_{i'}^2, \qquad 
  \overline{N}_{H,2} = \big(n^{-1}(n-1)^{-1}\sum_{i=1}^n \sum_{i' \neq i}N_{i}^{-1}N_{i'}^{-1}\big)^{-1},\\
  \overline{N}_{2,2} &= n^{-1}(n-1)^{-1}\sum_{i=1}^n\sum_{i' \neq i}N_i N_{i'}^2, \qquad
  \overline{N}_{2,1} = n^{-1}(n-1)^{-1}\sum_{i=1}^n\sum_{i' \neq i}N_i^2 N_{i'}. 
 \end{align*}
Let $\widehat{R}_{obs}$ and $\widehat{R}_{subj}$ be the specific estimators of $R(\cdot)$ for OBS and SUBJ weighting schemes, respectively. The following corollary provides the convergence rate of both these estimators.

\begin{corollary} \label{cor:re}
Suppose that conditions (I)[(i), (ii) (iii), (v)(b), (vi)], (II), (III), (IV)[(ii), (iv)], and (V)[(i), (iii)] hold. Then, we have the following:
\begin{enumerate}
    \item[(1)] OBS:
    \begin{align*}
        \Vert \widehat{R}_{obs}(0) &- R(0) \Vert_{HS} = O_p\bigg( h_c^2 + b^2 + \sqrt{\left( \frac{1}{\overline{N}_2h_c^2b} + \frac{\overline{N}_{2,1}}{\overline{N}_2^2h_c b}+ \frac{\overline{N}_{2,12}}{\overline{N}_2^2 b} \right)\frac{1}{n(n-1)}} \bigg).
    \end{align*}
    \item[(2)] SUBJ:
    \begin{align*}
        \Vert \widehat{R}_{subj}(0) &- R(0) \Vert_{HS} = O_p\bigg( h_c^2 + b^2 + \sqrt{\left( \frac{1}{\overline{N}_{H,2}h_c^2b} + \frac{1}{\overline{N}_{H}h_c b}+ \frac{1}{b} \right)\frac{1}{n(n-1)}} \bigg).
    \end{align*}
\end{enumerate}
\end{corollary}
\begin{remark}
We note that $\overline{N} \ge \overline{N}_H$ and $\overline{N}_2 \ge \overline{N}_{H2}$. Define $\overline{N}_{S2}=n^{-1}\sum_{i=1}^nN_i^2$. Analogously to the arguments used in \citet{zhang2016sparse}, the rate for $\widehat{R}_{obs}$ is as good as $\widehat{R}_{subj}$ if $\lim \sup_n \overline{N}_{S2}/(\overline{N})^2 < \infty$, $\lim \sup_n (\overline{N} \cdot \overline{N}_{2,1})/(\overline{N}_2)^2 < \infty$, and $\lim \sup_n (\overline{N}_{2,12})/(\overline{N}_2^2) < \infty$. 
\end{remark}
\begin{remark}
When $\overline{N}_2 < \infty$ and $\overline{N}_{H,2}<\infty$, the rate  of convergence of $\widehat{R}(0)$ without spatial smoothing is $O_p(1/ \sqrt{N(\Delta)}h_c^2)$ \citep{liu2017functional} where $N(\Delta)$ is the collection of location pairs that share the same spatial correlation structure. Comparing this with the asymptotic order in Corollary \ref{cor:re}, it appears reasonable to take  $b=O(1/ \sqrt{n-1})$. Then under the conditions in Remark 3 the convergence rates in Corollary \ref{cor:re} will be of similar order as that of two-dimensional smoothing as in  Corollary 4.3 of \citet{zhang2016sparse}. Moreover, $b=O(1/\sqrt{n-1})$ satisfy $b/h_c \rightarrow 0$ under the condition that $nh_c^2 \rightarrow \infty$.
\end{remark}


\begin{corollary} \label{cor:R-partition}
Suppose that conditions (I)[(i), (ii) (iii), (v)(b), (vi)], (II), (III), (IV)[(ii), (iv)], and (V)[(i), (iii)] hold. Then, we have the following:
\begin{enumerate}
    \item[(1)] OBS: Suppose the bandwidth satisfies $b=O(1/\sqrt{n-1})$ and $nh_c^2 \rightarrow \infty$.  Moreover, assume that $\lim \sup_n \overline{N}_{S2}/(\overline{N})^2 < \infty$, $\lim \sup_n (\overline{N} \cdot \overline{N}_{2,1})/(\overline{N}_2)^2 < \infty$, and $\lim \sup_n (\overline{N}_{2,12})/(\overline{N}_2^2) < \infty$. 
    \begin{enumerate}
        \item If $\overline{N}_2/n^{1/2} \rightarrow 0$ and $h_c \asymp (n\overline{N}_2)^{-1/6}$, then
        \begin{align*}
            \Vert \widehat{R}_{obs}(0) - R(0) \Vert_{HS} &= O_p(h_c^2) + o_p\left( \frac{1}{\sqrt{n\overline{N}_2h_c^2}} \right).
        \end{align*}
        \item If $\overline{N}_2/n^{1/2} \rightarrow C$ for some $C < \infty$ and $h_c \asymp n^{-1/4}$, then
        \begin{align*}
            \Vert \widehat{R}_{obs}(0) - R(0) \Vert_{HS} &= O_p\left( \frac{1}{\sqrt{n}} \right)+ o_p\left(\frac{1}{\sqrt{n}} \right).
        \end{align*}
        \item If $\overline{N}_2/n^{1/2} \rightarrow \infty$,  $h_c = o(n^{-1/4})$ and $h_c \overline{N}_2 \rightarrow \infty$, then
        \begin{align*}
            \Vert \widehat{R}_{obs}(0) - R(0) \Vert_{HS} &=O_p\left( \frac{1}{\sqrt{n}} \right)+ o_p\left(\frac{1}{\sqrt{n}}\right).
        \end{align*}
    \end{enumerate}
    \item[(2)] SUBJ: Replacing $\overline{N}$ and $\overline{N}_{2}$ in (1) with $\overline{N}_H$ and $\overline{N}_{H,2}$, respectively, leads to the results for $\widehat{R}_{subj}$.
\end{enumerate}
\end{corollary}
\begin{remark}
As in \citet{zhang2016sparse}, under either the OBS or SUBJ scheme,  Corollary \ref{cor:R-partition} yields a systematic partition of functional data into three categories: non-dense, dense and ultra-dense. Cases (b) and (c) correspond to dense data as root-$n$ convergence can be achieved, and case (a) corresponds to non-dense data where root-$n$ convergence can never be achieved.
\end{remark}

\begin{remark}
As suggested in \citet{zhang2016sparse}, one may use a mixture of the OBS and SUBJ schemes which attain at least the better rate between $\widehat{R}_{obs}(0)$ and $\widehat{R}_{subj}(0)$. One could use $v_{i,i'}=\theta/[\sum_{i=1}^n\sum_{i' \neq i} N_i N_{i'}] + (1-\theta)/[n(n-1) N_iN_{i'}]$ for some $0 \le \theta \le 1$. Then the rate of the corresponding estimator, denoted by $\widehat{R}_{\theta}(0)$, is
\begin{align*}
    \Vert \widehat{R}_{\theta}(0) - R(0) \Vert_{HS} &= O_p\left( h_c^2 + \sqrt{\theta^2 d_{n1} + (1-\theta)^2 d_{n2}} \right),
\end{align*}
where 
\begin{align*}
    d_{n1} = \left( \frac{1}{\overline{N}_2h_c^2} + \frac{\overline{N}_{2,1}}{\overline{N}_2^2h_c }+ \frac{\overline{N}_{2,12}}{\overline{N}_2^2} \right)\frac{1}{n}\quad\text{   and }\quad d_{n2} = \left( \frac{1}{\overline{N}_{H,2}h_c^2} + \frac{1}{\overline{N}_{H}h_c }+ 1 \right)\frac{1}{n}.
\end{align*}
It is easy to note that the rate is minimized for $\theta^*= d_{n2}/(d_{n1}+ d_{n2})$.
Recently, \citet{zhang2018optimal} proposed an optimal weighting scheme for longitudinal and functional data. This scheme is designed to minimize the convergence rate given a bandwidth. One may adopt a similar approach in our setting.
\end{remark}

\subsection{Results for $\widehat{R}_e(0, t_1, t_2)$}
The following theorem provides the asymptotic bias and variance for the inverse regression function $m(t,y)$ in (\ref{eqn:obj-ge}). It extends Lemma 2.1 of \citet{jiang2014inverse} to spatially correlated functional data.

  \begin{theorem} \label{th:1}
    Suppose conditions (I)[(i), (ii), (iv)], (II)(i), (III)(i) and (IV)(i)  hold. For a fixed interior point $(t,y)$, if $h_t \sum_{i=1}^n N_i(N_i-1)w_i^2/ \sum_{i=1}^nN_iw_i^2 \rightarrow C_0$, for some constant $C_0$, then
    \begin{eqnarray} \label{eqn:ge-bias}
      E\{\widehat{m}(t,y)-m(t,y) ~|~ \mathcal{X}_y\} &=& \frac{c(K_2)}{2} tr\left(\bm{H}.\bm{\mathcal{H}}_m(t,y) \right)+o_p(tr(\bm{H}))
      \end{eqnarray}
      and
      \begin{eqnarray}
      \text{var}\{\widehat{m}(t,y) ~|~ \mathcal{X}_y\}  &=&
\frac{\sum_{i=1}^n N_iw_i^2}{ |\bm{H}|^{1/2} g^2(t,y)}(\Psi(t,y) + \Lambda(t) + \sigma^2)  \nu_{K_2,2} g(t, y) \nonumber \\ &&  +  \frac{\sum_{i=1}^n N_i (N_i-1)w_i^2}{h_yg^2(t,y)} \times  \{  (\Psi(t,y,t) + \Lambda(t,t)) g_3(t,y,t)  \}  + o_p(1),\label{eqn:ge-var}
\end{eqnarray}
    where $\mathcal{X}_y=\{(T_{ij},Y_i), i=1,\ldots,n, j=1,\ldots,N_i\}$ is the set of design points,  $\bm{H}=\text{diag}\{h_t^2,h_y^2\}$, $\int \bm{u}\bm{u}^TK_2(\bm{u})d\bm{u}=c(K_2)\bm{I}$ with an identity matrix $\bm{I}$,
    $\bm{\mathcal{H}}_m(t,y)$ is the second derivative of $m(t,y)$, and $\nu_{K_2,2}=\int
    K_2^2(\bm{u})d\bm{u}$.
\end{theorem}
The proof of Theorem \ref{th:1} is available in the supplement \citet{suppl}. By letting $w_i=1/N_{S,0}$ and $w_i=1/nN_i$, respectively, in Theorem \ref{th:1}, we obtain the corresponding rates for the OBS and SUBJ estimators. Let $\mathfrak{B} \subset \mathbb{R}$ be a compact set. We now estimate the conditional covariance $R_e(0, t_1,t_2)$  empirically as, 
\begin{align*}
\widehat{R}_e(0,t_1,t_2) &= \frac{1}{n}\sum_{i=1}^n \widehat{m}(t_1,Y_i) \widehat{m}(t_2,Y_i) \mathbbm{1}_{\{ Y_i \in \mathfrak{B}\}}-\frac{1}{n}\sum_{i=1}^n \widehat{m}(t_1,Y_i)\mathbbm{1}_{\{ Y_i \in \mathfrak{B}\}} \frac{1}{n}\sum_{j=1}^n \widehat{m}(t_2,Y_j)\mathbbm{1}_{\{ Y_j \in \mathfrak{B}\}},
\end{align*}
where $\mathbbm{1}_{\{ Y_j \in \mathfrak{B}\}}$ is an indicator function that takes value one if $Y_j \in \mathfrak{B}$ and zero otherwise. 
With the help of Theorem \ref{th:1},  in the  following lemma, we obtain the convergence rate  of covariance operator $\widehat{R}_e(0,t_1,t_2)$ in Hilbert-Schmidt norm.

    \begin{lemma} \label{lm:2}
      Assume the  conditions (I)[(i),(ii), (iv)], (II)(i), (III)(i), (IV)(i) and (V)(ii) hold. If $h_t \asymp h_y \asymp h \asymp n^{-1/d}$ where $d \in (2,14)$, then
      \begin{align}
        \Vert \widehat{R}_e(0)- R_e(0) \Vert_{HS} &= O_p\left(h^2 + \sqrt{\frac{\sum_{i=1}^n N_i w_i^2}{h^2} + \frac{\sum_{i=1}^nN_i(N_i-1)w_i^2}{h} }\right).
      \end{align}   
    \end{lemma}
\begin{remark}
The proof of the above lemma is available in the supplement \citet{suppl}. It requires uniform convergence of the cross-product terms $u_{l,k}$ in (S.36) in the supplement \citet{suppl}. For this reason, we restrict both uniform convergence and Hilbert-Schmidt norm computations to the compact set $(t,y) \in \mathcal{I} \times \mathfrak{B}$. In practice, we can trim both ends of the support of $y$ and consider the middle 90\% or 95\% of $y$ values to compute $\widehat{R}_e(0)$. This in turn improves the quality of the estimator by alleviating boundary effects associated with smoothing.  
\end{remark}
Let $\widehat{R}_{e,obs}(0)$ and $\widehat{R}_{e,subj}(0)$ be the specific estimators of $R_e(\cdot)$ for the OBS and SUBJ weighting schemes, respectively. The following corollary provides the convergence rates of these estimators.
\begin{corollary}\label{cor:re-rate}
Suppose conditions (I)[(i), (ii), (iv)], (II)(i),(III)(i), (IV)(i) and (V)(ii) hold. If the bandwidths satisfy $h_t \asymp h_y \asymp h \asymp n^{-1/d}$ where $d \in (2,14)$, then 
\begin{align*}
&\textrm{(1) OBS:} \quad \Vert \widehat{R}_{e,obs}(0)- R_e(0) \Vert_{HS} = O_p\left(h^2 + \sqrt{\left(\frac{1}{\overline{N}h^2} + \frac{\overline{N}_{S2}}{(\overline{N})^2h} \right) \frac{1}{n}}\right);\\
&\textrm{(2) SUBJ:}\quad \Vert \widehat{R}_{e,subj}(0)- R_e(0) \Vert_{HS} = O_p\left(h^2 + \sqrt{\left(\frac{1}{\overline{N}_Hh^2} + \frac{1}{h} \right) \frac{1}{n}}\right).
\end{align*}
\end{corollary}
The above rates are consistent with the rates of two-dimensional smoothing. The following corollary provides the convergence rates for sparse, dense and ultra-dense functional data.

\begin{corollary}\label{cor:re-spa-den}
Suppose conditions (I)[(i), (ii), (iv)], (II)(i), (III)(i), (IV)(i) and (V)(ii) hold. Assume $h_t \asymp h_y \asymp h \asymp n^{-1/d}$ where $d \in (2,14)$. Then, we have the following:
\begin{enumerate}
    \item[(1)] OBS: Assume $\lim \sup_n \overline{N}_{S2}/(\overline{N})^2 < \infty$.
    \begin{enumerate}
        \item If $\overline{N}/n^{1/2} \rightarrow 0$ and $h \asymp (n\overline{N})^{-1/6}$, then $ \Vert \widehat{R}_{e,obs}(0) - R_e(0) \Vert_{HS}=O_P(h^2+1/\sqrt{n\overline{N}h^2})$.
        \item If $\overline{N}/n^{1/2} \rightarrow C$ for some $C < \infty$ and $h \asymp n^{-1/4}$, then $  \Vert \widehat{R}_{e,obs}(0) - R_e(0) \Vert_{HS} =O_P(1/\sqrt{n})$.
        \item If $\overline{N}/n^{1/2} \rightarrow \infty$, $h = o(n^{-1/4})$, and $h \overline{N} \rightarrow \infty$, then $ \Vert \widehat{R}_{e, obs}(0) - R_e(0) \Vert_{HS} =O_P(1/\sqrt{n})$.
    \end{enumerate}
    \item[(2)] SUBJ: Replacing $\overline{N}$  in (1) with $\overline{N}_H$  leads to the results for $\widehat{R}_{e, subj}$.
\end{enumerate}
\end{corollary}

\subsection{Results for $\widehat{\Gamma}(t_1,t_2)$}
We now derive the asymptotic bias and variance expression for the covariance operator $\widehat{\Gamma}(t_1,t_2)$ in (\ref{eq:obj-gs}) using general weight $v_i$ for the $i$th observation. This result requires the uniform convergence of $\widehat{\mu}(t)$ which is provided in Lemma S.1.3 in the supplement \citet{suppl}. Let $\Xi_{1,2}=\Xi(t_1,t_2)$, $\Xi_{1,2,2}=\Xi(t_1,t_2,t_2)$, $\Xi_{1,2,1}=\Xi(t_1,t_2,t_1)$ and $\Xi_{1,2,1,2}=\Xi(t_1,t_2,t_1,t_2)$ for $\Xi$ defined in Condition (I)(v). Similarly, let $\Lambda_{1,2}=\Lambda(t_1,t_2)$, $\Lambda_{1,2,2}=\Lambda(t_1,t_2,t_2)$, $\Lambda_{1,2,1}=\Lambda(t_1,t_2,t_1)$ and $\Lambda_{1,2,1,2}=\Lambda(t_1,t_2,t_1,t_2)$ for $\Lambda$ defined in condition (I)(v). The following theorem extends Theorem 3.2 of \citet{zhang2016sparse} to spatially correlated functional data except the asymptotic normality result, which is not pursued here given the complicated setting.

\begin{theorem} \label{thm:rate-phi}
  Under the conditions (I)[(i), (ii), (iii), (v)], (II)(i), (III)(i), (IV)(iii) and (V)(i), for a fixed interior point $(t_1,t_2)$, we have that
  \begin{align*}
            E\{\widehat{\Gamma}(t_1,t_2)- \Gamma(t_1,t_2) ~|~ \mathcal{X}_T \} &= \frac{c(K_1)}{2} tr\left(
            \bm{H}_c. \bm{\mathcal{H}}_{c}(t_1,t_2) \right) 
            +o_p(tr(\bm{H}_c))
  \end{align*}
  and 
  \begin{align*}
    \text{var}\{\widehat{\Gamma}(t_1,t_2) ~|~ \mathcal{X}_T\} &= \frac{\sum_{i=1}^n N_i(N_i-1) v_i^2}{h_c^2}
     \bigg(\frac{\Xi_{1,2}+\Lambda_{1,2}+\sigma^4}{f_T(t_1)f_T(t_2)} \left[\nu^2_{K_1,2}+ K_1^{*^2}\left(\frac{t_1-t_2}{h_c}\right) \right] \bigg) \\ &  + \frac{\sum_{i=1}^n N_i(N_i-1)(N_i-2)v_i^2}{ h_c}  \\&  \times \bigg( \frac{[\Xi_{1, 2,2}+ \Lambda_{1,2,2}] f_T(t_1)+ [\Xi_{1,2,1}+ \Lambda_{1,2,1}] f_T(t_2)}{f_T(t_1)f_T(t_2)}  \left[\nu_{K_1,2}+K_1^*\left(\frac{t_1-t_2}{h_c}\right) \right] \bigg) \\ & \qquad \qquad \qquad + o_p(1),
        \end{align*}
        where $\mathcal{X}_T$ is defined in Theorem \ref{th:Rhat}, $\nu_{K_1,2}=\int K_1^2(u)du$, $c(K_1)=\int u^2K_1(u)du$, $\bm{H}_c=\text{diag}\{h_{c}^2, h_{c}^2 \}$, $\bm{\mathcal{H}}_{c}(t_1,t_2)$ is the second derivatives of $\Gamma(t_1,t_2)$, and $K_1^*(\cdot)$ is a convolution of kernel $K_1(\cdot)$ with itself.
\end{theorem}
\begin{remark}
We omit the proof, as it proceeds similarly to Theorem \ref{th:Rhat}.  The asymptotic bias and variance expressions are similar to those from Theorem 3.2 of \citet{zhang2016sparse}, which assumes that the functions are independent.  As mentioned in \citet{zhang2016sparse}, the covariance estimator exhibits the ``discontinuity of the asymptotic variance'' because the asymptotic variance in Theorem \ref{thm:rate-phi} is different for $t_1=t_2$ and $t_1 \neq t_2$. We refer to Corollaries 3.4 and 3.5 in \citet{zhang2016sparse} for a comprehensive discussion. 
\end{remark}

\subsection{Results for e.d.r. directions}
    To provide convergence rates for e.d.r. directions, we consider the  following  condition on the decay rate of $\{\xi_i\}$ and an additional assumption on the coefficients $a_{k,i}=\langle \pi_i,\eta_k \rangle$.  Recall that $R(0,t,t')=\sum_{i=1}^{\infty} \xi_i \pi_i(t)\pi_i(t')$, where $\{\xi_i\}_1^{\infty}$ are the eigenvalues and $\{\pi_i\}_1^{\infty}$  are the corresponding eigenfunctions.

\vspace{1em}
    \textbf{Condition (VI) (eigenvalues):}
    \begin{enumerate}[label=(\alph*)]
    \item  Assume that $\xi_i-\xi_{i+1}> \overline{M}_0 i^{-\overline{\alpha}_1-1}$ for some constant $\overline{M}_0>0$, $i \ge 1$ and $\overline{\alpha}_1>1$. This implies $\xi_i \ge \overline{M}_1 i^{-\overline{\alpha}_1}$ for some constant $\overline{M}_1>0$. 
    \item The coefficient $|a_{k,j}| < \overline{M}_2 j^{-\overline{\alpha}_2}$ for some positive constants $\overline{M}_2$ and $\overline{\alpha}_2$.
    \end{enumerate}

The above condition (a) ensures that the gaps between the eigenvalues will not be too small. It also provides a lower bound on the rate at which $\xi_i$ decreases.  The condition (b) assumes that the coefficients $a_{k,j}$ do not decrease too quickly. Similar conditions exist in the literature for functional data, for example, we refer to \citet{hall2007methodology} and \citet{jiang2015correction}.
    Because  $\eta_k= \sum_{i=1}^{\infty} a_{k,i} \pi_i(t)$ and $\eta_k \in \mathfrak{R}_{R^{1/2}(0)}$, we obtain that $\sum_{i=1}^{\infty}a_{k,i}^2 \xi_i^{-1} < \infty$. Under Condition (VI), this implies that $2\overline{\alpha}_2-\overline{\alpha}_1 >1$. A sufficient condition for $2\overline{\alpha}_2-\overline{\alpha}_1 >1$ is $\overline{\alpha}_2 \ge \overline{\alpha}_1 $ and it  means that the space spanned by $\{\eta_k\}_{k=1}^K$ is at least as smooth as operator $R(0)$.

    Let $R_L^{-1/2}(0) =\sum_{i=1}^L \xi_i^{-1/2} \pi_i(s) \pi_i(t)$ be the truncated version of $R^{-1/2}(0)$. Let 
    \begin{align*}
    r_{n1} &=h^4 + \frac{\sum_{i=1}^n N_i w_i^2}{h^2} + \frac{\sum_{i=1}^nN_i(N_i-1)w_i^2}{h} ~ \text{  and} \\ r_{n2} &= h_c^4 + b^4 + \frac{\sum_{i=1}^n \sum_{i' \neq i} v^2_{i,i'} N_i N_{i'}}{h_c^{2}b} + \frac{\sum_{i=1}^n \sum_{i' \neq i} v^2_{i,i'} N_i N_{i'} (N_{i'}-1)}{ h_c b} \nonumber\\ & + \frac{\sum_{i=1}^n \sum_{i' \neq i} v^2_{i,i'} N_i N_{i'} (N_{i}-1)}{h_c b} + \frac{\sum_{i=1}^n \sum_{i' \neq i} v^2_{i,i'} N_i (N_i-1) N_{i'} (N_{i'}-1)}{b}.
    \end{align*}

    
    
    In the following theorem, we provide the convergence rate for the e.d.r. direction $\widehat{\beta}_j$.
    \begin{theorem}\label{th:main}
      Under the conditions of Lemma \ref{lm:L2-R} and Lemma \ref{lm:2} and condition (VI), we have that
      \begin{align*}
        \Vert \widehat{\beta}_j-\beta_j \Vert_{L_2}^2 &= O_p\left( L^{(-2\overline{\alpha}_2+ \overline{\alpha}_1+1)}+L^{(3 \overline{\alpha}_1+2)}r_{n1} + L^{(3\overline{\alpha}_1-2\overline{\alpha}_2+4)}r_{n2}\right) \text{ for } j=1,\ldots,K. 
      \end{align*}
      \end{theorem}

\begin{remark}
    The proof uses the results from \citet{hall2007methodology} and proceeds similarly to Theorem 2.1 in \citet{jiang2015correction}. Therefore, it is omitted here for brevity. The convergence rate of e.d.r. direction $\widehat{\beta}_j$ depends on the rates of involved covariance operators $\widehat{R}_e(0)$ and $\widehat{R}(0)$. The second and third terms in the above rate are equal for the choice of $L=(r_{n2}/r_{n1})^{1/2(\overline{\alpha}_2-1)}$. Corollary \ref{cor:edr-spa-den} provides conditions under which $r_{n2} \asymp r_{n1}$.
\end{remark}
\begin{remark}
    The rates obtained in Theorem \ref{th:main} may not be optimal for the e.d.r. directions $\widehat{\beta}_j$ which, as discussed in \citet{jiang2014inverse}, can be estimated at the optimal rate for a one-dimensional smoother. In the following Corollary \ref{cor:edr-spa-den}, under some regularity conditions, we show that  $\widehat{\beta}_j$ achieve the optimal convergence rate.  
\end{remark}

The following corollary provides the convergence rate of $\widehat{\beta}_j$  for sparse, dense and ultra-dense functional data similar to Corollaries \ref{cor:R-partition} and \ref{cor:re-spa-den}. Let $\widehat{\beta}_{obs}$ and $\widehat{\beta}_{subj}$ be the specific estimators of $\beta$ for the OBS and SUBJ weighting schemes, respectively. 

\begin{corollary}\label{cor:edr-spa-den}
Suppose all the conditions in corollaries \ref{cor:R-partition} and \ref{cor:re-spa-den} including the ones for OBS and SUBJ schemes hold. Additionally, assume $h_c \asymp h $ and $\overline{\alpha}_2>> \overline{\alpha}_1 >1$. Then, we have the following:
\begin{enumerate}
    \item[(1)] OBS: 
    \begin{enumerate}
        \item If $\overline{N}/n^{1/2} \rightarrow 0$ and $h \asymp (n\overline{N})^{-1/6}$, then $ \Vert \widehat{\beta}_{obs} - \beta \Vert_{L_2} =O_P(h^2+1/\sqrt{n\overline{N}h^2})$.
        \item If $\overline{N}/n^{1/2} \rightarrow C$ for some $C < \infty$ and $h \asymp n^{-1/4}$, then $\Vert \widehat{\beta}_{obs} - \beta \Vert_{L_2} = O_p( 1/\sqrt{n}).$
        \item If $\overline{N}/n^{1/2} \rightarrow \infty$,  $h = o(n^{-1/4})$, and $h \overline{N} \rightarrow \infty$, then $            \Vert \widehat{\beta}_{obs} - \beta \Vert_{L_2} =O_p( 1/\sqrt{n}).$
    \end{enumerate}
    \item[(2)] SUBJ: Replacing $\overline{N}$ and $\overline{N}_2$  in (1) with $\overline{N}_H$ and $\overline{N}_{H,2}$, respectively, leads to the results for $\widehat{\beta}_{subj}$.
\end{enumerate}
\end{corollary}

\section{Conclusion}
In this study, we propose a new inverse regression framework for spatially correlated functional data. Our framework adopts the general weighting scheme for functional data proposed in \citet{zhang2016sparse}. Similarly, it extends the mixed asymptotic framework for the spatial sampling proposed by \citet{lu2014nonparametric} to the functional data setting. Due to both of these frameworks, our methodology provides a unified framework for spatially correlated functional data which includes the existing results as special cases. In terms of theoretical results, we provide the point-wise convergence results, $L_2$ convergence results for the covariance operators and uniform convergence result for the mean. These results are new to the literature and hence of independent interest. Finally, we show that the estimated e.d.r. directions achieve the $\sqrt{n}$-convergence rate for the dense and ultra-dense functional data and achieve the standard smoothing rate for the sparse functional data. The simulations in the supplement \citet{suppl} validate our theoretical findings. The proposed method  shows a good performance. We find that it improves upon the existing methods when the functional data include a nugget effect. At the same time, it also does not lose much efficiency when the data do not include a nugget effect.

The present study considers the estimation of the e.d.r. directions and does not discuss the issues related to inference. An interesting direction for future research is to provide an inference framework for the estimated e.d.r. directions.

\section*{Acknowledgements}
The authors would like to thank an anonymous referee, an Associate Editor and the Editor for their
constructive comments which led to significant improvements in the paper. The authors would also like to thank Dr. Ci-Ren Jiang for sharing the code for Sliced Inverse Regression. 
\begin{supplement}
\stitle{Supplement to ``Inverse regression for spatially distributed functional data''}
\sdescription{The supplementary material contains the technical proofs of the main results, a simulation study, and a  real data analysis.}
\end{supplement}





\bibliographystyle{imsart-nameyear} 
\bibliography{dfsir}       

\setcounter{subsection}{0}
\setcounter{section}{0}
\renewcommand{\thesection}{S.\arabic{section}}
\renewcommand{\thesubsection}{S.\arabic{section}.\arabic{subsection}}
\setcounter{subsubsection}{0}
\renewcommand{\thesubsubsection}{\textbf{S.\arabic{subsection}.\arabic{subsubsection}}}
\setcounter{equation}{0}
\renewcommand{\theequation}{S.\arabic{equation}}
\newtheorem{thm}{Theorem}
\renewcommand{\thelemma}{S.\arabic{lemma}}

\renewcommand{\thefigure}{S.\arabic{figure}}
\renewcommand{\thetable}{S.\arabic{table}}


\newpage


\title{Supplement to ``Inverse regression for spatially distributed functional data''}

\vspace{3em}

\section{Proofs}

In this supplement, we present the proofs of our main results in Section \ref{sec:asymp} of the manuscript and provide the results from a simulation study, and the weather data analysis.  The asymptotic results for the covariance operators $\widehat{R}$ and $\widehat{R}_e$ require uniform convergence result for the mean estimator. Therefore, in Section \ref{subsec:mu} we provide the uniform convergence result for the mean estimator $\widehat{\mu}$. Section \ref{subsec:rmat} presents the pointwise and $L_2$ convergence results for the spatial covariance matrix $\widehat{R}$. Similarly, Section \ref{subsec:remat} provides the pointwise and $L_2$ convergence results for the conditional mean function $E(X(t)|Y)$ and the conditional covariance matrix $\widehat{R}_e$, respectively. Finally, we present the results from the simulation study and the data analysis in Sections \ref{sec:numer} and \ref{sec:application}, respectively.

The following two lemmas are required for our theoretical arguments. For completeness, we present them here.
\begin{lemma}\label{lem:mixing}
  Let $\mathcal{L}_{r}(\mathcal{F})$ denote the class of $\mathcal{F}$- measurable random variables $W$ satisfying  $\Vert W \Vert_{r}:=(E|W|^{r})^{1/{r}} <\infty$. Let $\bm{U} \in \mathcal{L}_{a}(\mathcal{B}(\mathcal{S}))$ and $ V \in \mathcal{L}_{b}(\mathcal{B}(\mathcal{S'}))$ where $\mathcal{B}(\mathcal{S})$ and $\mathcal{B}(\mathcal{S'})$ denote the $\sigma$- fields. Then for any $1 \le a, b, c < \infty$ such that $a^{-1}+ b^{-1}+ c^{-1}=1$,
  \begin{align}
    |E[\bm{U}V]- E[\bm{U}]E[V]| \le C_{\alpha}\Vert \bm{U} \Vert_{a} \Vert V \Vert_{b} [\alpha(\mathcal{S},\mathcal{S'})]^{1/{c}},
    \label{eqn:lem5-1}
  \end{align}
  where $\Vert \bm{U} \Vert_a := \Vert(\bm{U}^T\bm{U})^{1/2}\Vert_a$ and $\alpha(\mathcal{S},\mathcal{S'})=\sup\{|P(AB)-P(A)P(B)| : A \in \mathcal{B}(S), B \in \mathcal{B}(S') \}$ and $C_{\alpha}$ is some constant.
  
\end{lemma}
\begin{proof}
    It is taken from Theorem 17.2.3 in \citet{,bingham1973independent} and Lemma 1 in \citet{,deo1973note} shared identical results.
\end{proof}


\begin{lemma}\label{lemma:mgf:bound}
Suppose that $|X|\leq c$ for a random variable $X$, then for any $\lambda>0$ we have
\begin{eqnarray*}
\ev(e^{\lambda(X-\ev(X))})\leq \exp\left\{\lambda^2\sigma^2\left(\frac{e^{\lambda c}-1-\lambda c}{(\lambda c)^2}\right) \right\}
\end{eqnarray*}
where $\sigma^2=var(X)$.
\end{lemma}
\begin{proof}
This is the Bennett’s inequality; please see Page 51 of \cite{wainwright2019high}.
\end{proof}

\subsection{Results for $\widehat{\mu}(t)$} \label{subsec:mu}
We now provide the uniform convergence result for the mean $\widehat{\mu}(t)$, $t\in [0,1]$. 
\begin{lemma}\label{lem:uniform-mean}
    Under conditions (I)[(i),(ii), (iii), (iv)(c)], (III)(i), (IV)(ii) and (V)(i),
    \begin{align*}
        \underset{t \in [0,1]}{\sup} |\widehat{\mu}(t)- \mu(t)| &= O_p\left( h_{\mu}^2 + \log(n)\sqrt{ \frac{\sum_{i=1}^n N_iw_i^2}{h_{\mu}} + \sum_{i=1}^n N_i(N_i-1) w_i^2}\right).
    \end{align*}
\end{lemma}
\begin{proof}
    The proof requires the following two results. The arguments analogous to Theorem 5.1 in \citet{zhang2016sparse} together with Lemma \ref{lemma:exp:inequality:Qn} and Proposition \ref{proposition:uniform:rate} yield the desired result.  
\end{proof}

\begin{proposition} \label{proposition:uniform:rate}
 Suppose conditions (I)[(i),(ii), (iii), (iv)(c)], (III)(i), (IV)(ii) and (V)(i) hold.  Define 
 \begin{align*}
 S_n(t) &= \sum_{i=1}^n w_i\sum_{j=1}^{N_i} K_{h_{\mu}}(T_{ij}-t)\mathcal{U}_{ij},
 \end{align*}
 where $\mathcal{U}_{ij}=V_{ij}+U_{ij}$.
 Then it follows that 
 \begin{align}
 \sup_{t}|S_n(t)| &= O_P\left(\log(n)\sqrt{\frac{\sum_{i=1}^n N_iw_i^2}{h_{\mu}} + \sum_{i=1}^n N_i(N_i-1) w_i^2}\right).\label{eq:proposition:uniform:rate:statement:1}
 \end{align}
Moreover, it holds that
\begin{align}
    \sup_{t}\left|\sum_{i=1}^n w_i\sum_{j=1}^{N_i} K_{h_{\mu}}(T_{ij}-t)e_{ij}\right|=O_P\left(\log(n)\sqrt{\frac{\sum_{i=1}^n N_iw_i^2}{h_{\mu}} + \sum_{i=1}^n N_i(N_i-1) w_i^2}\right).\label{eq:proposition:uniform:rate:statement:2}
\end{align}
\begin{proof}
Statement (\ref{eq:proposition:uniform:rate:statement:2}) follows from Lemma 5 in \cite{zhang2016sparse}. It suffices to prove Statement (\ref{eq:proposition:uniform:rate:statement:1}), and the proof is divided into three steps. For simplicity in notation, we denote $h_{\mu}=h$ in the below proof.

\noindent \textbf{Step 1:} Consider the following decomposition: $S_n(t)=S_n^{(1)}(t)+S_n^{(2)}(t)$, where
\begin{eqnarray*}
S_n^{(1)}(t)&=&\sum_{i=1}^n w_i \sum_{j=1}^{N_i} \{K_h(T_{ij}-t)\mathcal{U}_{ij}I(|\mathcal{U}_{ij}|\leq B_n)-\ev[K_h(T_{ij}-t)\mathcal{U}_{ij}I(|\mathcal{U}_{ij}|\leq B_n)]\},\\
S_n^{(2)}(t)&=&\sum_{i=1}^n w_i \sum_{j=1}^{N_i} \{K_h(T_{ij}-t)\mathcal{U}_{ij}I(|\mathcal{U}_{ij}|> B_n)-\ev[K_h(T_{ij}-t)\mathcal{U}_{ij}I(|\mathcal{U}_{ij}|> B_n)]\}.
\end{eqnarray*}
Let $D_i(\bfs_i)=\sup_{t}|V(\bfs_i,t)|+\sup_{t}|U_i(t)|$ and $\sup_{\bfs}\ev[|D_i(\bfs)|^{2+\varrho}]<\infty$ by Condition (I)[(iv)(c)]. Hence, we see that
\begin{eqnarray*}
\sup_{t}|S_n^{(2)}(t)|\leq 2 h^{-1} \sum_{i=1}^n w_i \sum_{j=1}^{N_i} B_n^{-(1+\varrho)}|\mathcal{U}_{ij}|^{2+\varrho}\leq 2h^{-1} \sum_{i=1}^n w_i \sum_{j=1}^{N_i} B_n^{-(1+\varrho)} |D_i(\bfs_i)|^{2+\varrho}.
\end{eqnarray*}
Taking expectation, it follows that
\begin{eqnarray}
\ev[\sup_{t}|S_n^{(2)}(t)|]\leq 2Ch^{-1}B_n^{-(1+\varrho)}.\label{eq:proposition:uniform:rate:step1:eq1}
\end{eqnarray}

\noindent \textbf{Step 2:}
Let $t_1, t_2,\ldots, t_r$ be an equidistance partition of $[0, 1]$, and it follows that
\begin{eqnarray*}
\sup_{t\in [0, 1]}|S_n^{(1)}(t)|\leq \sup_{1\leq i \leq r}|S_n^{(1)}(t_i)|+\sup_{|t-\widetilde{t}|\leq 1/r}|S_n^{(1)}(t)-S_n^{(1)}(\widetilde{t})|.
\end{eqnarray*}
The Lipschitz conditions of kernel implies that
\begin{eqnarray*}
\sup_{|t-\widetilde{t}|\leq 1/r}|S_n^{(1)}(t)-S_n^{(1)}(\widetilde{t})|\leq \sup_{|t-\widetilde{t}|\leq 1/r}\sum_{i=1}^n w_i \sum_{j=1}^{N_i}\frac{C}{h^2}|t-\widetilde{t}| |\mathcal{U}_{ij}|\leq Ch^{-2}r^{-1} \sum_{i=1}^n w_i \sum_{j=1}^{N_i}|\mathcal{U}_{ij}|.
\end{eqnarray*}
Taking expectation, we show that
\begin{eqnarray}
\sup_{|t-\widetilde{t}|\leq 1/r}|S^{(1)}_n(t)-S_n^{(1)}(\widetilde{t})|=O_P(h^{-2}r^{-1}\sum_{i=1}^n w_i N_i)=O_P(h^{-2}r^{-1}).\label{eq:proposition:uniform:rate:step1:eq2}
\end{eqnarray}

Using Lemma \ref{lemma:exp:inequality:Qn}, it follows that
\begin{eqnarray}
&&\pr(\sup_{1\leq i \leq r}|S_n^{(1)}(t_i)|>8\log(n) Q_n)\nonumber\\
&\leq& 2r\exp\left\{-\frac{\log^2(n) Q_n^2}{2Q_n^2+\frac{2}{3}CB_nn^{-1} h^{-1}p^2\delta_n^{-2}\log(n) Q_n}\right\}+\frac{Cr\Delta_n^2}{p^2} \psi(Cp^2\delta_n^{-2} ,n)\phi(p))\sqrt{\frac{B_nn^{-1}h^{-1}\Delta_n^2\delta_n^{-2}}{\log(n) Q_n}}.\nonumber\\\label{eq:proposition:uniform:rate:step1:eq3}
\end{eqnarray}

Combining (\ref{eq:proposition:uniform:rate:step1:eq1})-(\ref{eq:proposition:uniform:rate:step1:eq3}), we show that
\begin{eqnarray}
\sup_{t}|S_n(t)|=O_P\left(\log(n)Q_n+h^{-2}r^{-1}+h^{-1}B_n^{-(1+\varrho)}\right),\label{eq:proposition:uniform:rate:step1:eq4}
\end{eqnarray}
provided the RHS of (\ref{eq:proposition:uniform:rate:step1:eq3}) goes to zero.

\noindent \textbf{Step 3:}
The convergence of RHS of (\ref{eq:proposition:uniform:rate:step1:eq3}) is implied by the following sufficient conditions:
\begin{eqnarray*}
\log(r)-\log^2(n)&\to& -\infty,\\
\log(r)-\frac{\log(n)nQ_nh\delta_n^2}{B_np^2}&\to & -\infty,\\
\frac{B_n\Delta_n^6\delta_n^{-2} r^2 \psi^2(Cp^2\delta_n^{-2} ,n)\phi^2(p)}{\log(n)nQ_n h p^4}&\to &0.
\end{eqnarray*}
Taking $r=h^{-2}Q_n^{-1}$, $B_n=(Q_nh)^{-\frac{1}{1+\varrho}}$, (\ref{eq:proposition:uniform:rate:step1:eq4}) becomes
\begin{eqnarray}
\sup_{t}|S_n(t)|=O_P\left(\log(n)Q_n\right),\nonumber
\end{eqnarray}
and the sufficient conditions become
\begin{eqnarray*}
\log(r)-\log^2(n)&\to& -\infty,\\
\log(r)-\frac{\log(n)n Q_n^{\frac{2+\varrho}{1+\varrho}}h^{\frac{2+\varrho}{1+\varrho}}  \delta_n^2}{p^2}&\to & -\infty,\\
\frac{\Delta_n^6\delta_n^{-2}  \psi^2(Cp^2\delta_n^{-2} ,n)\phi^2(p)}{\log(n)nQ_n^{\frac{4+3\varrho}{1+\varrho}} h^{\frac{6+5\varrho}{1+\varrho}} p^4}&\to &0.
\end{eqnarray*}
Finally, using the fact that $\psi(i,j)\leq \min(i,j)$ and $\alpha(p)\leq p^{-\rho}$, the above conditions can be implied by
\begin{eqnarray*}
\frac{n Q_n^{\frac{2+\varrho}{1+\varrho}}h^{\frac{2+\varrho}{1+\varrho}}  \delta_n^2}{p^2}\to \infty, \frac{\Delta_n^6\delta_n^{-2} }{\log(n)nQ_n^{\frac{4+3\varrho}{1+\varrho}} h^{\frac{6+5\varrho}{1+\varrho}} p^{2\rho}}\to 0.
\end{eqnarray*}
To guarantee existence such $p$, we need
\begin{eqnarray*}
n Q_n^{\frac{2+\varrho}{1+\varrho}}h^{\frac{2+\varrho}{1+\varrho}}  \delta_n^2\to \infty,  [n Q_n^{\frac{2+\varrho}{1+\varrho}}h^{\frac{2+\varrho}{1+\varrho}} ]^{\rho}>>\frac{\Delta_n^6\delta_n^{-2} }{\log(n)nQ_n^{\frac{4+3\varrho}{1+\varrho}} h^{\frac{6+5\varrho}{1+\varrho}} },
\end{eqnarray*}
which are identical to
\begin{eqnarray*}
n Q_n^{\frac{2+\varrho}{1+\varrho}}h^{\frac{2+\varrho}{1+\varrho}}  \delta_n^2\to \infty, \frac{n^{\rho+1} Q_n^{\frac{(2+\varrho)\rho+4+3\varrho}{1+\varrho}} h^{\frac{(2+\varrho)\rho+6+5\varrho}{1+\varrho}}\delta_n^2\log(n)}{\Delta_n^6}\to \infty.
\end{eqnarray*}

For example, when $\varrho\to \infty$, the conditions become
\begin{eqnarray*}
nQ_nh\delta_n^2\to \infty, \frac{n^{\rho+1}Q_n^{\rho+3}h^{\rho+5}\delta_n^2\log(n)}{\Delta_n^6}\to \infty.
\end{eqnarray*}
\end{proof}
\end{proposition}

\begin{lemma}\label{lemma:exp:inequality:Qn}
 Suppose conditions (I)[(i),(ii), (iii), (iv)(c)], (III)(i), (IV)(ii) and (V)[(i),(ii)] hold. 
Define $S_n(t)=\sum_{i=1}^n w_i\sum_{j=1}^{N_i} K_h(T_{ij}-t)\mathcal{U}_{ij}$, where $\mathcal{U}_{ij} :=\mathcal{U}(\bm{s}_i, T_{ij})= V_{ij}+U_{ij}$. 
Then the following holds for any $\epsilon, p>0$ and some constant $C>0$:
\begin{eqnarray}
\pr(|S_n(t)|>8\epsilon)\leq 2\exp\left\{-\frac{\epsilon^2}{2Q_n^2+\frac{2}{3}CB_nh^{-1}p^2\delta_n^{-2}\epsilon}\right\}+\frac{C\Delta_n^2}{p^2} \psi(Cp^2\delta_n^{-2} ,n)\phi(p)\sqrt{\frac{B_nn^{-1} h^{-1}\Delta_n^2\delta_n^2}{\epsilon}}.\nonumber
\end{eqnarray}
where $Q_n^2=\sum_{i=1}^n  w_i^2N_i h^{-1}+\sum_{i=1}^n  w_i^2N_i (N_{i}-1)$.
\end{lemma}
\begin{proof}
The proof is divided into three steps.

\noindent \textbf{Step 1:}
By definition of $\Delta_n$ and without loss of generality, we can assume $\bfs_1,\ldots, \bfs_n\in [0, C\Delta_n]^2:=\Omega_n$ for some $C>0$. If not, we can shift the points. For each pair $i, k$, we define rectangles
\begin{eqnarray*}
A_{ku}^{(1)}&=&\{\bfs=(s_1, s_2): 2(k-1)p\leq s_{1}<(2k-1)p, 2(u-1)p\leq s_{2}<(2u-1)p\},\\
A_{ku}^{(2)}&=&\{\bfs=(s_1, s_2): 2(k-1)p\leq s_{1}<(2k-1)p, (2u-1)p\leq s_{2}<2up\},\\
A_{ku}^{(3)}&=&\{\bfs=(s_1, s_2): (2k-1)p\leq s_{1}<2kp, 2(u-1)p\leq s_{2}<(2u-1)p\},\\
A_{ku}^{(4)}&=&\{\bfs=(s_1, s_2): (2k-1)p\leq s_{1}<2kp, (2u-1)p\leq s_{2}<2up\}.
\end{eqnarray*}
Correspondingly, let us define
\begin{eqnarray*}
V_{ku}^{(1)}&=&\sum_{i=1}^n I(\bfs_i \in A_{ku}^{(1)})w_i\sum_{j=1}^{N_i} K_h(T_{ij}-t)\mathcal{U}_{ij},\\
V_{ku}^{(2)}&=&\sum_{i=1}^n I(\bfs_i \in A_{ku}^{(2)})w_i\sum_{j=1}^{N_i} K_h(T_{ij}-t)\mathcal{U}_{ij},\\
V_{ku}^{(3)}&=&\sum_{i=1}^n I(\bfs_i \in A_{ku}^{(3)})w_i\sum_{j=1}^{N_i} K_h(T_{ij}-t)\mathcal{U}_{ij},\\
V_{ku}^{(4)}&=&\sum_{i=1}^n I(\bfs_i \in A_{ku}^{(4)})w_i\sum_{j=1}^{N_i} K_h(T_{ij}-t)\mathcal{U}_{ij}.
\end{eqnarray*}
It is not difficult to verify that
\begin{eqnarray}
S_n(t)=\sum_{k=1}^M \sum_{u=1}^M(V_{ku}^{(1)}+V_{ku}^{(2)}+V_{ku}^{(3)}+V_{ku}^{(4)}),\label{eq:lemma:exp:inequality:Qn:step1:eq0}
\end{eqnarray}
where $M$ is an integer such that $M\leq C\Delta_n/p$.  Noting that 
$$|{V}_{ku}^{(1)}|\leq 2B_nh^{-1}\sum_{i=1}^n I(\bfs_i\in A_{ku}^{(1)})w_iN_i\leq B_n h^{-1} Cp^{2}\delta_{n}^{-2}\max_{i}w_iN_i\leq CB_nn^{-1}h^{-1}p^2\delta_n^{-2},$$ Lemma 1.2 of \citet{bosq2012nonparametric} implies that there is a sequence of independent variables $\{W_{ku}\}$ such that ${V}_{ku}^{(1)}$ and $W_{ku}$ share the same distribution, and further
\begin{eqnarray}
\pr(|{V}_{ku}^{(1)}-W_{ku}|>x)&\leq&C \sqrt{\frac{B_n n^{-1}h^{-1}p^2\delta_n^{-2}}{x}} \alpha(A_{ku}^{(1)},\Omega_n/A_{ku}^{(1)})\nonumber\\
&\leq& C \sqrt{\frac{B_n n^{-1}h^{-1}p^2\delta_n^{-2}}{x}} \psi(Cp^2\delta_n^{-2} ,n)\phi(p).\label{eq:lemma:exp:inequality:Qn:step1:eq1}
\end{eqnarray}

\noindent \textbf{Step 2:}
By direct examination, we show that
\begin{eqnarray*}
&&\ev(|W_{ku}|^2)= \ev(|{V}_{ku}^{(1)}|^2)\\
&\leq& \ev\left\{\left(\sum_{i=1}^n I(\bfs_i \in A_{ku}^{(1)})w_i\sum_{j=1}^{N_i} K_h(T_{ij}-t)\mathcal{U}_{ij}\right)^2\right\}\\
&\leq&\sum_{i=1}^n  I(\bfs_i \in A_{ku}^{(1)})w_i^2Var\left(\sum_{j=1}^{N_i} K_h(T_{ij}-t)\mathcal{U}_{ij}\right)\\
&&+ \sum_{0<d(\bfs_i, \bfs_{i'})\leq c_n}Cov\left(I(\bfs_i \in A_{ku}^{(1)})w_i\sum_{j=1}^{N_i} K_h(T_{ij}-t)\mathcal{U}_{ij},I(\bfs_{i'} \in A_{ku}^{(1)})w_{i'}\sum_{j=1}^{N_i} K_h(T_{{i'}j}-t)\mathcal{U}_{{i'}j}\right)\\
&&+\sum_{d(\bfs_i, \bfs_{i'})> c_n}Cov\left(I(\bfs_i \in A_{ku}^{(1)})w_i\sum_{j=1}^{N_i} K_h(T_{ij}-t)\mathcal{U}_{ij},I(\bfs_{i'} \in A_{ku}^{(1)})w_{i'}\sum_{j=1}^{N_i} K_h(T_{{i'}j}-t)\mathcal{U}_{{i'}j}\right)\\
&\leq&\sum_{i=1}^n  I(\bfs_i \in A_{ku}^{(1)})w_i^2\sum_{j=1}^{N_i}\sum_{j'=1}^{N_i'} \ev\left\{|K_h(T_{ij}-t)\mathcal{U}_{ij}K_h(T_{{i}j'}-t)\mathcal{U}_{{i}j'}|\right\}\\
&&+  \sum_{0<d(\bfs_i, \bfs_{i'})\leq c_n}I(\bfs_i \in A_{ku}^{(1)})I(\bfs_{i'} \in A_{ku}^{(1)})w_iw_{i'}\sum_{j=1}^{N_i}\sum_{j'=1}^{N_i'} \ev\left\{|K_h(T_{ij}-t)\mathcal{U}_{ij}K_h(T_{{i'}j'}-t)\mathcal{U}_{{i'}j'}|\right\}\\
&&+ \sum_{d(\bfs_i, \bfs_{i'})> c_n}I(\bfs_i \in A_{ku}^{(1)})I(\bfs_{i'} \in A_{ku}^{(1)})w_iw_{i'}\sum_{j=1}^{N_i}\sum_{j'=1}^{N_i'} \ev\left\{|K_h(T_{ij}-t)\mathcal{U}_{ij}K_h(T_{{i'}j'}-t)\mathcal{U}_{{i'}j'}|\right\}.
\end{eqnarray*}
Taking summation, we show that
\begin{eqnarray}
\sum_{k=1}^M \sum_{u=1}^M\ev(|W_{ku}|^2)&\leq&\sum_{i=1}^n  w_i^2\sum_{j=1}^{N_i}\sum_{j'=1}^{N_i} \ev\left\{|K_h(T_{ij}-t)\mathcal{U}_{ij}K_h(T_{{i}j'}-t)\mathcal{U}_{{i}j'}|\right\}\nonumber\\
&&+  \sum_{0<d(\bfs_i, \bfs_{i'})\leq c_n}w_iw_{i'}\sum_{j=1}^{N_i}\sum_{j'=1}^{N_i'} \ev\left\{|K_h(T_{ij}-t)\mathcal{U}_{ij}K_h(T_{{i'}j'}-t)\mathcal{U}_{{i'}j'}|\right\}\nonumber\\
&&+ \sum_{d(\bfs_i, \bfs_{i'})> c_n}w_iw_{i'}\sum_{j=1}^{N_i}\sum_{j'=1}^{N_i'} \ev\left\{|K_h(T_{ij}-t)\mathcal{U}_{ij}K_h(T_{{i'}j'}-t)\mathcal{U}_{{i'}j'}|\right\}\nonumber\\
&:=&V_1+V_2+V_3.\label{eq:lemma:exp:inequality:Qn:step2:eq1}
\end{eqnarray}
Notice that if $j\neq j'$, it follows that
\begin{eqnarray*}
&&\ev\left\{|K_h(T_{ij}-t)\mathcal{U}_{ij}K_h(T_{{i}j'}-t)\mathcal{U}_{{i}j'}|\right\}\\
&=&\ev\left\{\int_0^1\int_0^1\frac{1}{h^2}K\left(\frac{x_1-t}{h}\right)K\left(\frac{x_2-t}{h}\right)|\mathcal{U}(\bfs_i, x_1)\mathcal{U}(\bfs_i, x2)|f_T(x_1)f_T(x_2)dx_1dx_2\right\}\\
&=& \ev\left\{\int_{-C}^{C}\int_{-C}^{C}K\left(y_1\right)K\left(y_2\right)|\mathcal{U}(\bfs_i, t+hy_1)\mathcal{U}(\bfs_i, t+hy_2)|f_T(t+hy_1)f_T(t+hy_2)dy_1dy_2\right\}\\
&\leq& C\ev\left\{\int_{-C}^{C}\int_{-C}^{C}|\mathcal{U}(\bfs_i, t+hy_1)\mathcal{U}(\bfs_i, t+hy_2)|dy_1dy_2\right\}\\
&\leq& C\left\{\int_{-C}^{C}\int_{-C}^{C}\ev(|\mathcal{U}(\bfs_i, t+hy_1)\mathcal{U}(\bfs_i, t+hy_2)|)dy_1dy_2\right\}\\
&\leq&C\int_{-C}^{C}\sup_{\bfs, t}\ev\left\{|\mathcal{U}(\bfs, t)|^2\right\}dy\leq C.
\end{eqnarray*}
Similarly, if $j=j'$, we have
\begin{eqnarray*}
\ev\left\{|K_h(T_{ij}-t)\mathcal{U}_{ij}K_h(T_{{i}j'}-t)\mathcal{U}_{{i}j'}|\right\}
&=& \ev\left\{K_h^2(T_{ij}-t)\mathcal{U}_{ij}^2\right\}\\
&=&\ev\left\{\int_0^1\frac{1}{h^2}K^2\left(\frac{x-t}{h}\right)|\mathcal{U}^2(\bfs_i, x)|f_T(x)dx\right\}\\
&=&\ev\left\{\int_{-C}^{C}\frac{1}{h}K^2\left(y\right)\mathcal{U}^2(\bfs_i, t+hy)f_T(x)dx\right\}\\
&\leq& Ch^{-1}\ev\left\{\int_{-C}^{C}\mathcal{U}^2(\bfs_i, t+hy)dx\right\}\leq Ch^{-1}.
\end{eqnarray*}
By the same calculation, we show that
\begin{eqnarray*}
\ev\left\{|K_h(T_{ij}-t)\mathcal{U}_{ij}K_h(T_{{i'}j'}-t)\mathcal{U}_{{i'}j'}|\right\}&\leq& C, \textrm{if } i\neq i',\\
\ev\left\{|K_h^{2+\varrho}(T_{ij}-t)|\mathcal{U}_{ij}|^{2+\varrho}\right\}&\leq& Ch^{-(1+\varrho)}\ev\left\{\int_{-C}^{C} |\mathcal{U}(\bfs_i, t+hy)|^{2+\varrho}dy\right\}\\
&\leq&Ch^{-(1+\varrho)}.
\end{eqnarray*}
Hence, if $i\neq i'$, we have
\begin{eqnarray*}
&&\ev\left\{|K_h(T_{ij}-t)\mathcal{U}_{ij}K_h(T_{{i'}j'}-t) \mathcal{U}_{{i'}j'}|\right\}\\
&\leq& \ev^{\frac{1}{2+\varrho}}\left\{|K_h^{2+\varrho}(T_{ij}-t)|\mathcal{U}_{ij}|^{2+\varrho}\right\}\ev^{\frac{1}{2+\varrho}}\left\{|K_h^{2+\varrho}(T_{i'j'}-t)|\mathcal{U}_{i'j'}|^{2+\varrho}\right\}\phi^{\frac{\varrho}{2+\varrho}}(d(\bfs_i, \bfs_{i'}))\\
&\leq& Ch^{-\frac{2(1+\varrho)}{2+\varrho}}\phi^{\frac{\varrho}{2+\varrho}}(d(\bfs_i, \bfs_{i'})).
\end{eqnarray*}
Using the above bounds and the definitions of $V_1, V_2, V_3$ in (\ref{eq:lemma:exp:inequality:Qn:step2:eq1}), we show that
\begin{eqnarray*}
V_1&\leq&C\sum_{i=1}^n  w_i^2N_i h^{-1}+C\sum_{i=1}^n  w_i^2N_i (N_{i}-1),\\
V_2&\leq& C\sum_{0<d(\bfs_i, \bfs_{i'})\leq c_n}w_iw_{i'}N_i N_{i'},\\
V_3&\leq&  \sum_{d(\bfs_i, \bfs_{i'})\geq c_n}w_iw_{i'}N_i N_{i'}h^{-\frac{2(1+\varrho)}{2+\varrho}}\phi^{\frac{\varrho}{2+\varrho}}(d(\bfs_i, \bfs_{i'}))\\
&\leq& C\sum_{m=c_n}^\infty n^{-2}h^{-\frac{2(1+\varrho)}{2+\varrho}}\phi^{\frac{\varrho}{2+\varrho}}(m) nm\delta_n^{-2}=Ch^{-\frac{2(1+\varrho)}{2+\varrho}}n^{-1}\delta_n^{-2} \sum_{m=c_n}^\infty m\phi^{\frac{\varrho}{2+\varrho}}(m).
\end{eqnarray*}
Hence, (\ref{eq:lemma:exp:inequality:Qn:step2:eq1}) and the rate conditions imply that
\begin{eqnarray}
\sum_{k=1}^M \sum_{u=1}^M\ev(|W_{ku}|^2)\leq V_1+V_2+V_3= Q_n^2 \label{eq:lemma:exp:inequality:Qn:step2:eq2}
\end{eqnarray}
and
\begin{eqnarray}
|W_{ku}|\leq 2B_nh^{-1}\sum_{i=1}^n I(\bfs_i\in A_{ku}^{(1)})w_iN_i\leq B_n h^{-1} Cp^{2}\delta_{n}^{-2}\max_{i}w_iN_i\leq CB_nn^{-1} h^{-1}p^2\delta_n^{-2}.\label{eq:lemma:exp:inequality:Qn:step2:eq3}
\end{eqnarray}

\noindent \textbf{Step 3:}
For any $\epsilon>0$, simple calculation leads to
\begin{eqnarray*}
&&\pr(|\sum_{k=1}^M \sum_{u=1}^M V_{ku}^{(1)}|>2\epsilon)\\
&\leq&  \pr(|\sum_{k=1}^M \sum_{u=1}^M {V}_{ku}^{(1)}|>2\epsilon, |{V}_{ku}^{(1)}-W_{ku}|\leq \epsilon/M^2, \textrm{ for all } k,u)+\sum_{k=1}^M \sum_{u=1}^M\pr( |{V}_{ku}^{(1)}-W_{ku}|> \epsilon/M^2)\\
&\leq&  \pr(|\sum_{k=1}^M \sum_{u=1}^M W_{ku}|>\epsilon)+\sum_{k=1}^M \sum_{u=1}^M\pr( |{V}_{ku}^{(1)}-W_{ku}|> \epsilon/M^2)\\
&:=&R_1+R_2.
\end{eqnarray*}
Using (\ref{eq:lemma:exp:inequality:Qn:step2:eq2}) and (\ref{eq:lemma:exp:inequality:Qn:step2:eq3}), Bernstein inequality implies that
\begin{eqnarray*}
R_1\leq 2\exp\left\{-\frac{\epsilon^2}{2Q_n^2+\frac{2}{3}CB_nh^{-1}p^2\delta_n^{-2}\epsilon}\right\}.
\end{eqnarray*}
The definition of $M$ and (\ref{eq:lemma:exp:inequality:Qn:step1:eq1}) together show that
\begin{eqnarray*}
R_2\leq CM^2 \sqrt{\frac{B_nn^{-1} h^{-1}M^2p^2\delta_n^{-2}}{\epsilon}}  \psi(Cp^2\delta_n^{-2} ,n)\phi(p)\leq \frac{C\Delta_n^2}{p^2} \psi(Cp^2\delta_n^{-2} ,n)\phi(p))\sqrt{\frac{B_nn^{-1} h^{-1}\Delta_n^2\delta_n^{-2}}{\epsilon}}.
\end{eqnarray*}
Combining the above three inequalities, we conclude that
\begin{eqnarray}
\pr(|\sum_{k=1}^M \sum_{u=1}^M V_{ku}^{(1)}|>2\epsilon)&\leq& 2\exp\left\{-\frac{\epsilon^2}{2Q_n^2+\frac{2}{3}CB_nh^{-1}p^2\delta_n^{-2}\epsilon}\right\}\nonumber\\
&&+\frac{C\Delta_n^2}{p^2} \psi(Cp^2\delta_n^{-2} ,n)\phi(p)\sqrt{\frac{B_nn^{-1} h^{-1}\Delta_n^2\delta_n^{-2}}{\epsilon}}.\label{eq:lemma:exp:inequality:Qn:step3:eq1}
\end{eqnarray}

Similar inequalities as (\ref{eq:lemma:exp:inequality:Qn:step3:eq1}) can be proved for $V_{ku}^{(2)}, V_{ku}^{(3)}, V_{ku}^{(4)}$ by the same argument. Hence, we complete the proof using union bound and  (\ref{eq:lemma:exp:inequality:Qn:step1:eq0}).

\end{proof}

\subsection{Results for $\widehat{R}(s; t_1, t_2)$} \label{subsec:rmat}
In this section, we provide  point-wise and $L_2$ convergence rates for the spatial covariance estimator $\widehat{R}(s; t_1,t_2)$. The following theorem provides the point-wise bias and variance.

\begin{proof}[Proof of Theorem \ref{th:Rhat}]
  For convenience in notation, we denote $K_{ij}(t_1)=K_{h_c}(T_{ij}-t_1)$, $L_{ii'}=L_b(\Vert \bm{s}_i-\bm{s}_{i'} \Vert- \Vert \bm{s}_0 \Vert)$, $\mathcal{A}_{ij,i'j'}=(T_{ij}, T_{i'j'}, \Vert \bm{s}_i-\bm{s}_{i'}\Vert)^T$, $\bm{a}=(t_1, t_2, \Vert \bm{s}_0 \Vert)^T$, and $\bm{\alpha}=(\alpha_1, \alpha_2, \alpha_3)^T$. 
  Let $C_{ij,i'j'} :=[Z_{ij}-\widehat{\mu}(T_{ij})][Z_{i'j'}-\widehat{\mu}(T_{i'j'})]$, $1\leq i,i' \leq n$, $1\leq j \leq N_i$, $1 \le j' \le N_{i'}$.  Because of Lemma \ref{lem:uniform-mean}, for simplicity, in the following proof we directly work with $\mu(t)$ instead of $\widehat{\mu}(t)$.
  
  Note that  $\widehat{R}(\Vert \bm{s}_0 \Vert; t_1,t_2)=\widehat{\alpha}_0$, where 
  \begin{eqnarray}
   (\widehat{\alpha}_0,\widehat{\bm{\alpha}}^T)^T 
    &=& \underset{(\alpha_0,\bm{\alpha})\in
      \mathbb R^4}{\text{arg min}}\sum_{i=1}^{n} \sum_{i' \neq i} v_{i,i'} \sum_{j=1}^{N_i} \sum_{j'=1}^{N_{i'}}   \bigg\{
      C_{ij, i'j'}-\alpha_0-\bm{\alpha}^T(\mathcal{A}_{ij,i'j'}-\bm{a}) \bigg\}^2 \nonumber
    \\ 
     && \qquad  \qquad \times K_{ij}(t_1) K_{i'j'}(t_2) L_{ii'}.
  \end{eqnarray}
Define 
\begin{align}
  \mathcal{P} = \begin{bmatrix}
    P_{00} & \bm{P}_{10}^T \\
    \bm{P}_{10} & \bm{P}_{11}
  \end{bmatrix}
  \label{eqn:r0-pmat}
\end{align}
where 
\begin{align}
  P_{00} &= \sum_{i=1}^{n} \sum_{i' \neq i} v_{i,i'} \sum_{j=1}^{N_i} \sum_{j'=1}^{N_{i'}} K_{ij}(t_1) K_{i'j'}(t_2) L_{ii'}, 
\end{align}
and 
\begin{align}
  \bm{P}_{lk} &= \sum_{i=1}^{n} \sum_{i' \neq i} v_{i,i'} \sum_{j=1}^{N_i} \sum_{j'=1}^{N_{i'}} (\mathcal{A}_{ij,i'j'}-\bm{a})^l((\mathcal{A}_{ij,i'j'}-\bm{a})^T)^k  K_{ij}(t_1) K_{i'j'}(t_2) L_{ii'},
\end{align}
 for $l,k \in \{(1,0) ,(1,1)\}$. Then, we can write 
 \begin{align*}
   \widehat{\bm{\theta}} := \begin{pmatrix}
     \widehat{\alpha}_0 - \alpha_0 \\
     \widehat{\bm{\alpha}}- \bm{\alpha}
   \end{pmatrix}
   =\mathcal{P}^{-1} \mathcal{Q}, \qquad \mathcal{Q}=\begin{pmatrix}
     q_{0} \\
     \bm{q}
   \end{pmatrix}
 \end{align*}
where 
\begin{align*}
  q_0 &= \sum_{i=1}^{n} \sum_{i' \neq i} v_{i,i'} \sum_{j=1}^{N_i} \sum_{j'=1}^{N_{i'}} \eta_{ij,i'j'} K_{ij}(t_1) K_{i'j'}(t_2) L_{ii'},  
\end{align*}
and 
\begin{align*}
  \bm{q} &= \sum_{i=1}^{n} \sum_{i' \neq i} v_{i,i'} \sum_{j=1}^{N_i} \sum_{j'=1}^{N_{i'}} \eta_{ij,i'j'} (\mathcal{A}_{ij,i'j'}-\bm{a}) K_{ij}(t_1) K_{i'j'}(t_2) L_{ii'},
\end{align*}
for $\eta_{ij,i'j'}=C_{ij,i'j'}-\alpha_0(\bm{a})-(\mathcal{A}_{ij,i'j'}-\bm{a})^T\bm{\alpha}$.
First, we consider the elements of (\ref{eqn:r0-pmat}). It follows from Conditions (I)[(i), (ii)] and (III) that
\begin{align*}
  E\{K_{ij}(t_1) K_{i'j'}(t_2)\} &= \int K(u)K(v) f_T(t_1+uh_c) f_T(t_2+vh_c) du dv \\
  &= f_T(t_1) f_T(t_2) (1+o(1)).
\end{align*}
 Therefore, using Conditions (II)[(ii), (iii), (iv)], and (III), and proceeding similar to \citet{lu2014nonparametric} we obtain that
\begin{align*}
  E\{ P_{00}\} &= \sum_{i=1}^{n} \sum_{i' \neq i} v_{i,i'} \sum_{j=1}^{N_i} \sum_{j'=1}^{N_{i'}} L_{ii'} E\{K_{ij}(t_1) K_{i'j'}(t_2)\} \\
  &= \int \int f_T(t_1) f_T(t_2) L((\Vert \bm{s}_1-\bm{s}_{2}\Vert-\Vert \bm{s}_0 \Vert)/b) b^{-1} g_{\bm{S}}(\bm{s}_1) g_{\bm{S}}(\bm{s}_2) d\bm{s}_1 d\bm{s}_2 (1+o(1)).
\end{align*} 
Now consider a transformation $(\Vert \bm{s}_1-\bm{s}_{2}\Vert- \Vert \bm{s}_0 \Vert)/b=s$, with $\bm{s}_j=(u_j,v_j)$ for $j=1,2$, where  $u_1=u_2+ (\Vert \bm{s}_0 + bs \Vert)\cos(\vartheta)$ and $v_1=v_2+ (\Vert \bm{s}_0 + bs \Vert)\sin(\vartheta)$ for $0< \vartheta \le 2\pi$. Let $u_2^*=u_2 + \Vert \bm{s}_0 \Vert \cos(\vartheta)$, $v_2^*=v_2+ \Vert \bm{s}_0 \Vert \sin(\vartheta)$. Note that $d\bm{s}_1=du_1 dv_1=b(\Vert \bm{s}_0 \Vert+ bs) ds d\vartheta$. Consequently,
\begin{align*}
  f_T(t_1) f_T(t_2) \int b^{-1} & L(s)  g_{\bm{S}}(u_2^*+bs\cos(\vartheta), v_2^*+bs\sin(\vartheta)) g_{\bm{S}}(u_2,v_2) \\ & \qquad \qquad \qquad \times b(\Vert \bm{s}_0 \Vert+ bs) ds d\vartheta du_2dv_2 (1+o_p(1))\\
  &=f_T(t_1) f_T(t_2) A_0^*(\Vert \bm{s}_0 \Vert) (1+o_p(1)).
\end{align*}
Next we show that $\text{var}\{P_{00}\} \rightarrow 0$. 
It follows from Conditions (I)(ii), (II),  (III), and (IV)(iv)  that
\begin{align*}
    \text{var}\{P_{00}\} &= \text{var}\left\{ \sum_{i=1}^{n} \sum_{i' \neq i} v_{i,i'} \sum_{j=1}^{N_i} \sum_{j'=1}^{N_{i'}} L_{ii'} K_{ij}(t_1) K_{i'j'}(t_2)  \right\} \\
    &= O_p\left(h_c^{-2}b^{-1} \sum_{i=1}^{n} \sum_{i' \neq i} v^2_{i,i'} N_i N_{i'} \right),
\end{align*}
which goes to zero based on the Condition (IV)(iv). Consequently,
\begin{align*}
  P_{00} &= f_T(t_1) f_T(t_2) A_0^*(\Vert \bm{s}_0 \Vert) (1+o_p(1)).
\end{align*}
For convenience in notation, in the following expression, we denote $f_{j}=f_T(t_j)$, $f'_j= \partial f_T(t_j)/ \partial t_j$, for $j=1,2$, $A_0^*=A_0^*(\Vert \bm{s}_0 \Vert)$,  $A_1^*=A_1^*(\Vert \bm{s}_0 \Vert)$, and $A_{10}^*=A_1^*+ \frac{A_0^*}{\Vert \bm{s}_0 \Vert}$. The calculations analogous to above yield the following, ignoring higher order terms, 
\begin{align}
  \mathcal{P} &= \begin{pmatrix}
    f_1f_2 A_0^* & h_c^2 f'_1f_2 A_0^* & h_c^2 f_1 f'_2 A_0^* & \mu_{L,2} b^2A_{10}^*\\
    h_c^2 f'_1f_2 A_0^*  & h_c^2 f_1f_2 A_0^* & h_c^4 f'_1f'_2 A_0^* & h_c^2 f'_1f_2 b^2 \mu_{L,2} A_{10}^* \\
    h_c^2 f_1f'_2 A_0^* & h_c^4 f'_1f'_2 A_0^* & h_c^2 f_1f_2 A_0^* & h_c^2 f_1f'_2 b^2 \mu_{L,2} A_{10}^* \\
    b^2 \mu_{L,2}A_{10}^* & h_c^2 f'_1f_2 b^2 \mu_{L,2} A_{10}^* & h_c^2 f_1f'_2 b^2 \mu_{L,2} A_{10}^* & b^2 \mu_{L,2} f_1f_2 A_0^*
  \end{pmatrix}. \label{eqn:pmat}
\end{align}

\vspace{2em}

\textbf{Bias Calculation:}
Observe that, by Taylor expansion, the leading term in $E\left\{\eta_{ij,i'j'}| T_{ij}, T_{i'j'} \right\}$ is $ 1/2[(T_{ij}-t_1)^2 \partial^2 R(\Vert \bm{s}_0 \Vert; t_1,t_2)/ \partial t_1^2 + (T_{i'j'}-t_2)^2 \partial^2 R(\Vert \bm{s}_0 \Vert; t_1,t_2)/ \partial t_2^2 + \\ (\Vert \bm{s}_i-\bm{s}_{i'} \Vert- \Vert \bm{s}_0 \Vert)^2 \partial^2 R(\Vert \bm{s}_0 \Vert; t_1,t_2)/\partial \Vert \bm{s}_0 \Vert]$. We write
\begin{align*}
  E[q_0] &= \sum_{i=1}^{n} \sum_{i' \neq i} v_{i,i'} \sum_{j=1}^{N_i} \sum_{j'=1}^{N_{i'}} E[\eta_{ij,i'j'} K_{ij}(t_1) K_{i'j'}(t_2) L_{ii'}]\\
  &= \frac{1}{2}\sum_{i=1}^{n} \sum_{i' \neq i}  L_{ii'} v_{i,i'} N_i N_{i'} \int \big\{ (T_{ij}-t_1)^2 \partial^2 R/ \partial t_1^2 + (T_{i'j'}-t_2)^2 \partial^2 R/ \partial t_2^2 \\ & + (\Vert \bm{s}_i-\bm{s}_{i'} \Vert- \Vert \bm{s}_0 \Vert)^2 \partial^2 R/\partial \Vert \bm{s}_0 \Vert^2 \big\} K_{ij}(t_1) K_{i'j'}(t_2) f_T(T_{ij}) f_T(T_{i'j'}) dT_{ij} dT_{ij'} + o(1) \\
  &= \frac{1}{2}\bigg\{ h_c^2  \big\{\partial^2 R/ \partial t_1^2+ \partial^2 R/ \partial t_2^2  \big\} + b^2 \mu_{L,2}  \partial^2 R/\partial \Vert \bm{s}_0 \Vert^2 \bigg\} f_T(t_1) f_T(t_2)A_0^*  + o_p(h_c^2+b^2),
\end{align*}
where  $\mu_{L,2}=\int u^2 L(u)du$. 
Consequently,
\begin{align*}
  E\{ \widehat{R}(\Vert \bm{s}_0 \Vert; t_1,t_2)- R(\Vert \bm{s}_0 \Vert; t_1,t_2) ~|~ \mathcal{X}_T \} &= \frac{1}{2}h_c^2 \big\{\partial^2 R/ \partial t_1^2+ \partial^2 R/ \partial t_2^2  \big\} \\ & \qquad \qquad  + \frac{1}{2} b^2 \mu_{L,2}  \partial^2 R/\partial \Vert \bm{s}_0 \Vert^2 + o_p(h_c^2+b^2),
\end{align*}
where $\mathcal{X}_T= \{T_{ij}, i=1,\ldots,n, j=1,\ldots,N_i\}$.

\vspace{2em}

\textbf{Variance Calculation:} We now compute the conditional variance for the estimator $\widehat{R}(\Vert \bm{s}_0 \Vert; t_1,t_2)$. By definition
\begin{eqnarray}
  \text{var}(\widehat{R}(\Vert \bm{s}_0 \Vert; t_1,t_2) \vert \mathcal{X}_T) &=&  \text{var}(\bm{e}_1^T \bm{\theta}\vert \mathcal{X}_T) 
  =  \text{var}\left\{P^{00}q_{0}+ \bm{P}^{{10}^T} \bm{q}_{1} \vert \mathcal{X}_T \right\},
  \label{eq:varR-split}
\end{eqnarray}
where $P^{00}$ and $\bm{P}^{10}$ are the elements of the first row in matrix $\mathcal{P}^{-1}$ which is the inverse of (\ref{eqn:pmat}).
 The leading term in the variance of  (\ref{eq:varR-split}) is
  \begin{align}
    \text{var}(P^{00} q_{0} \vert \mathcal{X}_T) &= (P^{00})^2 \bigg[ \sum_{i=1}^n\sum_{i' \neq i} v_{i,i'}^2 \text{var}\bigg(\sum_{j=1}^{N_i}\sum_{j'=1}^{N_{i'}} \eta_{ij, i'j'}K_{ij}(t_1)K_{i'j'}(t_2) L_{ii'} ~|~ \mathcal{X}_T\bigg) \nonumber \\ \qquad & +  \sum_{i=1}^{n}\sum_{i'\neq i} \underset{(l,l') \neq (i,i')}{\sum_{l}\sum_{l' \neq l}} v_{i,i'} v_{l,l'} \nonumber\\ & \times \text{cov}\bigg(\sum_{j}^{N_i} \sum_{j'}^{N_{i'}} \eta_{ij,i'j'} K_{ij}(t_1) K_{i'j'}(t_2)L_{ii'}, \sum_{k}^{N_l} \sum_{k'}^{N_{l'}} \eta_{lk,l'k'}K_{lk}(t_1)K_{l'k'}(t_2) L_{ll'} ~|~ \mathcal{X}_T\bigg)\bigg] \nonumber \\
    &:= (P^{00})^2 \{ (I) + (II) \}. \label{eqn:rmat-var-split}
  \end{align}
We show that the term (II) goes to zero as $n \rightarrow \infty$. Therefore the leading term for the variance is obtained from (I).

\vspace{2em}
\underline{Term (I)}:
Observe that  
  \begin{align*}
    \text{var}&\bigg(\sum_{j=1}^{N_i}\sum_{j'=1}^{N_{i'}} \eta_{ij, i'j'}K_{ij}(t_1)K_{i'j'}(t_2) L_{ii'} ~|~ \mathcal{X}_T \bigg) \\ &= \sum_{j=1}^{N_i}\sum_{j'=1}^{N_{i'}} \text{var}\bigg( \eta_{ij, i'j'}K_{ij}(t_1)K_{i'j'}(t_2) L_{ii'} ~|~ \mathcal{X}_T \bigg) \\ &  \qquad + \sum_{j=1}^{N_i}\sum_{j'=1}^{N_{i'}} \underset{(k,k') \neq (j,j')}{\sum_{k}\sum_{k'}} \text{cov}\big\{ \eta_{ij,i'j'} K_{ij}(t_1) K_{i'j'}(t_2)L_{ii'}, \eta_{ik,i'k'} K_{ik}(t_1) K_{i'k'}(t_2) L_{ii'} ~|~ \mathcal{X}_T \big\}.
  \end{align*}
Let $R(K_1)=\int K_1^2(u)du$ and $R(L)=\int L^2(u)u$.  Denote
\begin{align*}
  \Xi_R(t_1,t_2,t_1,t_2; \Vert \bm{s}_i-\bm{s}_{i'} \Vert) =\text{cov}\big\{ \eta_{ij,i'j'}, \eta_{ik,i'k'} | T_{ij}=t_1, T_{i'j'}=t_2, T_{ik}=t_1, T_{i'k'}=t_2 \}. 
\end{align*}
Under  Conditions (I)[(i), (ii), (vi)], (II) and (III), and using the standard arguments, we write the expected value of part (I) in (\ref{eqn:rmat-var-split}) as 
\begin{align*}
   &E[(I)] = \sum_{i=1}^n \sum_{i' \neq i} v^2_{i,i'} N_i N_{i'} h_c^{-2}b^{-1} R(L) R^2(K_1) f_T(t_1) f_T(t_2) ~ \Xi_R(t_1,t_2; \Vert \bm{s}_0 \Vert) A_0^* \\ 
    & \qquad + \sum_{i=1}^n \sum_{i' \neq i} v^2_{i,i'} N_i N_{i'} (N_i-1)(N_{i'}-1) b^{-1} R(L)  f_T^2(t_1) f_T^2(t_2) ~ \Xi_R(t_1,t_2, t_1, t_2; \Vert \bm{s}_0 \Vert) A_0^* \\
    &\qquad + \sum_{i=1}^n \sum_{i' \neq i} v^2_{i,i'} N_i N_{i'} (N_{i'}-1) h_c^{-1}b^{-1} R(L) R(K_1) f_T(t_1) f_T^2(t_2) ~ \Xi_R(t_1,t_2,t_2; \Vert \bm{s}_0 \Vert) A_0^* \\
    & \qquad +\sum_{i=1}^n \sum_{i' \neq i} v^2_{i,i'} N_i N_{i'} (N_i-1) h_c^{-1}b^{-1} R(L) R(K_1) f_T^2(t_1) f_T(t_2) ~ \Xi_R(t_1,t_2, t_1; \Vert \bm{s}_0 \Vert) A_0^* + o_p(1).
    \end{align*}
    Consequently, the leading term in the conditional variance is as follows
    \begin{align*}
    (P^{00})^2  E[(I)] &= \frac{\sum_{i=1}^n \sum_{i' \neq i} v^2_{i,i'} N_i N_{i'}}{h_c^{2}b A_0^* f_T(t_1)f_T(t_2)} R(L)  R^2(K_1) ~ \Xi_R(t_1,t_2; \Vert \bm{s}_0 \Vert) \\ & \qquad + \frac{\sum_{i=1}^n \sum_{i' \neq i} v^2_{i,i'} N_i N_{i'} (N_{i'}-1)}{ h_c b A_0^* f_T(t_1)} R(L) R(K_1)  \Xi_R(t_1,t_2,  t_2; \Vert \bm{s}_0 \Vert)\\
    & \qquad + \frac{\sum_{i=1}^n \sum_{i' \neq i} v^2_{i,i'} N_i N_{i'} (N_{i}-1)}{h_c b A_0^* f_T(t_2)} R(L) R(K_1)  \Xi_R(t_1,t_2,  t_1; \Vert \bm{s}_0 \Vert)\\
     & \qquad + \frac{\sum_{i=1}^n \sum_{i' \neq i} v^2_{i,i'} N_i (N_i-1) N_{i'} (N_{i'}-1)}{b A_0^* } R(L) \Xi_R(t_1,t_2, t_1, t_2; \Vert \bm{s}_0 \Vert)\bigg\} \\ &\qquad \qquad \qquad \times  (1+o_p(1)).
\end{align*}

\underline{Term (II)}: The idea is to relate the conditional covariance term with unconditional covariance terms using the following fact
\begin{align}
     E[\text{cov}(A,B|C)]  &= \text{cov}(A,B)- \text{cov}(E[A|C],E[B|C]), \label{eqn:c-cov}
\end{align}
for random variables $A$, $B$, and $C$. For simplicity in notation, let us denote $$\sum_{(l,l') \neq (i,i')} =\sum_{i=1}^{n}\sum_{i'\neq i} \underset{(l,l') \neq (i,i')}{\sum_{l}\sum_{l' \neq l}}.$$ We write the term (II) in (\ref{eqn:rmat-var-split}) as 
\begin{align}
    E \bigg\{ &  \sum_{(l,l') \neq (i,i')}  v_{i,i'}v_{l,l'}  \text{cov}\bigg(\sum_{j}^{N_i} \sum_{j'}^{N_{i'}} \eta_{ij,i'j'} K_{ij}(t_1) K_{i'j'}(t_2)L_{ii'}, \nonumber \\ & \qquad \qquad \qquad \qquad \qquad \qquad \sum_{k}^{N_l} \sum_{k'}^{N_{l'}} \eta_{lk,l'k'}K_{lk}(t_1)K_{l'k'}(t_2) L_{ll'} ~|~ \mathcal{X}_T\bigg) \bigg\} \nonumber \\
    &=   \sum_{(l,l') \neq (i,i')}  v_{i,i'}v_{l,l'}  \text{cov}\bigg(\sum_{j}^{N_i} \sum_{j'}^{N_{i'}} \eta_{ij,i'j'} K_{ij}(t_1) K_{i'j'}(t_2)L_{ii'}, \sum_{k}^{N_l} \sum_{k'}^{N_{l'}} \eta_{lk,l'k'}K_{lk}(t_1)K_{l'k'}(t_2) L_{ll'} \bigg) \nonumber \\
    &- \sum_{(l,l') \neq (i,i')}  v_{i,i'}v_{l,l'} \text{cov}\bigg(\sum_{j}^{N_i} \sum_{j'}^{N_{i'}} E[\eta_{ij,i'j'} K_{ij}(t_1) K_{i'j'}(t_2)L_{ii'} ~|~ \mathcal{X}_T],  \nonumber \\ & \qquad \qquad \qquad \qquad \qquad \qquad \sum_{k}^{N_l} \sum_{k'}^{N_{l'}} E[\eta_{lk,l'k'}K_{lk}(t_1)K_{l'k'}(t_2) L_{ll'} ~|~ \mathcal{X}_T]\bigg). \label{eqn:ltc-R}
\end{align}
From Lemma \ref{lem:R} it directly follows that the first  term in the right hand side of the equality in (\ref{eqn:ltc-R})   goes to zero. Recall that
\begin{align*}
    & E\left\{\eta_{ij,i'j'}| \mathcal{X}_T \right\} = 1/2[(T_{ij}-t_1)^2 \partial^2 R/ \partial t_1^2 + (T_{i'j'}-t_2)^2 \partial^2 R/ \partial t_2^2 + (\Vert \bm{s}_i-\bm{s}_{i'} \Vert- \Vert \bm{s}_0 \Vert)^2 \partial^2 R/\partial \Vert \bm{s}_0 \Vert].
\end{align*}
Since $T_{ij}$'s are independent, direct calculations yield that the second term in the right side of equality in (\ref{eqn:ltc-R}) goes to zero under condition (IV)(iv). Hence, the result is proved.
  
\end{proof}

\begin{lemma} \label{lem:R}
    Under the conditions of Theorem \ref{th:Rhat}, it holds that 
        \begin{align}
             & \sum_{(l,l') \neq (i,i')}  v_{i,i'}v_{l,l'}  \text{cov}\bigg\{ \sum_{j=1}^{N_i} \sum_{j'=1}^{N_{i'}} \eta_{ij,i'j'} L_{ii'} K_{ij}(t_1) K_{i'j'}(t_2),   \sum_{k=1}^{N_l} \sum_{k'=1}^{N_{l'}} \eta_{lk,l'k'} L_{ll'} K_{lk}(t_1) K_{l'k'}(t_2) \bigg\}
        \end{align}
        goes to zero as $n \rightarrow \infty$.
\end{lemma}
\begin{proof}

Since the random variable $\eta_{ij,ij'}$ is not bounded, we use the truncation argument in the following proof. 
Let $L_n= h^{-4/(4+\varrho)} $ for some $\varrho>0$ defined in Condition (II)(i). Define $\eta_{ij,i'j'}=\eta_{ij,i'j',1}+ \eta_{ij,i'j',2}$ where $\eta_{ij,i'j',1}:= \eta_{ij,i'j'}I_{\{|\eta_{ij,i'j'}| \le L_n\}}$ and $\eta_{ij,i'j',2} := \eta_{ij,i'j'}I_{\{|\eta_{ij,i'j'}| > L_n\}}$. Let
\begin{align*}
    \chi_{ij,i'j',l} := \eta_{ij,i'j',l}K_{ij}(t_1) K_{i'j'}(t_2) \text{  and   } \varpi_{ij,i'j',l} := \chi_{ij,i'j',l} - E \chi_{ij,i'j',l}, \qquad l=1,2.  
\end{align*}
Observe that $\varpi_{ij,i'j'}=\varpi_{ij,i'j',1}+\varpi_{ij,i'j',2}$. Therefore,
\begin{align}
  E\varpi_{ij,i'j'}\varpi_{lk,l'k'} &= E\varpi_{ij,i'j',1}\varpi_{lk,l'k',1} + E\varpi_{ij,i'j',1}\varpi_{lk,l'k',2} \nonumber\\ & \qquad \qquad +E\varpi_{ij,i'j',2}\varpi_{lk, l'k',1}+ E\varpi_{ij,i'j',2}\varpi_{lk, l'k',2}. \label{eqn:var-split-trun-r}
\end{align}
We now provide bounds for each term in (\ref{eqn:var-split-trun-r}). By Cauchy-Schwartz inequality, note that,
\begin{align*}
    |E\varpi_{ij,i'j',1}&\varpi_{lk,l'k',2}| \\ & \le  \{E \chi_{ij,i'j',1}^2 \}^{1/2} \{E \chi_{lk,l'k',2}^2 \}^{1/2} \\ & \le M_{71} h_c^{-1}  \{ E(E(|\eta_{lk,l'k'}|^{2} I_{\{|\eta_{lk,l'k'}| > L_n\}} | T_{lk}, T_{l'k'}) K^2_{lk}(t_1) K^2_{l'k'}(t_2) \}^{1/2} \\
    & \le M_{71} h_c^{-1} \{ h_c^{-2} L_n^{-\varrho} E(|\eta_{lk,l'k'}|^{2+\varrho} | T_{lk}, T_{l'k'}) \}^{1/2} \\
    & \le M_{72} h_c^{-2} L_n^{-\varrho/2}= M_{72} h_c^{-8/(4+\varrho)},
\end{align*}
where the last step follows from Condition (I)(v)(b) and $M_{71}$ and $M_{72}$ are some constants. Analogous arguments yield
\begin{align*}
    E\varpi_{ij,i'j',2} \varpi_{lk,l'k',1} &\le M_{72} h_c^{-2} L_n^{-\varrho/2}= M_{72} h_c^{-8/(4+\varrho)} \qquad \text{ and } \\
    E\varpi_{ij,i'j',2}\varpi_{lk,l'k',2} &\le M_{72} h_c^{-2} L_n^{-\varrho} = M_{72} h_c^{2(\varrho-4)/(4+\varrho)}.
\end{align*}

Consider
\begin{align*}
   E\varpi_{ij,i'j',1}\varpi_{lk, l'k',1} &= \text{cov}\{\eta_{ij,i'j',1} K_{ij}(t_1) K_{i'j'}(t_2),  \eta_{lk,l'k',1} K_{lk}(t_1) K_{l'k'}(t_2)\}\\
      & = E\{\eta_{ij,i'j',1} K_{ij}(t_1) K_{i'j'}(t_2)  \eta_{lk,l'k',1} K_{lk}(t_1) K_{l'k'}(t_2)\}\\&\qquad \qquad-E
  \{ \eta_{ij,i'j',1} K_{ij}(t_1) K_{i'j'}(t_2)) E\{\eta_{lk, l'k',1} K_{lk}(t_1) K_{l'k'}(t_2)\}\\
  &= \int  K_1(u)K_1(u')K_1(v)K_1(v') \{\Xi_{1,i,i',l,l'}(u,v,u',v')  \\ & \qquad - \Xi_{1,i,i'}(u,v) \Xi_{1,l,l'}(u',v') \}  f_T(u) f_T(v) f_T(u') f_T(v') du dv du' dv',
\end{align*}
where 
\begin{align*}
\Xi_{1,i,i',l,l'}(u,v,u',v') &:=E(\eta_{ij,i'j',1}\eta_{lk,l'k',1} | T_{ij}=u, T_{i'j'}=v, T_{lk}=u', T_{l'k'}= v'), \\ \Xi_{1,i,i'}(u,v) &:=E(\eta_{ij,i'j',1}| T_{ij}=u, T_{i'j'}=v).
\end{align*}
Using the definition of $|\eta_{ij,i'j',1}| \le L_n$, we obtain 
\begin{align*}
|\Xi_{1,i,i',l,l'}(u,v,u',v')| &\le L_n^2, \text{  and  } \\
   |\Xi_{1,i,i'}(u,v) \times \Xi_{1,l,l'}(u',v')| &\le L_n^2.
\end{align*}
Thus, using Condition (I)[(i),(ii)] and (III) we obtain
\begin{align*}
  |E\varpi_{ij,i'j',1}\varpi_{lk,l'k',1}| \le M_{73}L_n^2.  
\end{align*}
Consequently, we bound (\ref{eqn:var-split-trun-r}) as
\begin{align}
  |E\varpi_{ij,i'j'}\varpi_{lk,l'k'}| \le M_{72} h_c^{-2} L_n^{\varrho/2} + M_{73}L_n^2 = M_{74} h_c^{-8/(4+\varrho)}.
\end{align}
It follows from Conditions (I)[(i),(ii)] and (III) that
\begin{align}
\bigg\vert \text{cov} \big(\sum_{j=1}^{N_i}\sum_{j'=1}^{N_{i'}} \eta_{ij,i'j'} K_{ij}(t_1)K_{i'j'}(t_2), &\sum_{k=1}^{N_l} \sum_{k'=1}^{N_{l'}} \eta_{lk,l'k'} K_{lk}(t_1) K_{l'k'}(t_2) \big) \bigg\vert  \nonumber \\ & \qquad \qquad \le M_{74} N_i N_{i'} N_{l} N_{l'} h_c^{-8/(4+\varrho)}. 
\label{eqn:R-cv1}
\end{align}
Further, by applying Lemma \ref{lem:mixing} we obtain, for $\varrho >0$,
\begin{align*}
\vert \text{cov}&(\eta_{ij,i'j'} K_{ij}(t_1) K_{i'j'}(t_2),  \eta_{lk,l'k'} K_{lk}(t_1) K_{l'k'}(t_2)) \vert\\
&\le C_{\alpha} \Vert \eta_{ij,i'j'} K_{ij}(t_1)K_{i'j'}(t_2)   \Vert_{2+\varrho} \Vert \eta_{lk, l'k'} K_{lk}(t_1) K_{l'k'}(t_2) \Vert_{2+\varrho} \left[ \alpha(d(\bm{s}_i,\bm{s}_{i'}), d(\bm{s}_l,\bm{s}_{l'})\right]^{\frac{\varrho}{2+\varrho}}.
\end{align*}
Calculations due to change of variable yield
\begin{align*}
\Vert \eta_{ij,i'j'}& K_{ij}(t_1) K_{i'j'}(t_2) \Vert_{2+\varrho} \\&=\left(E|\eta_{ij,i'j'} K_{ij}(t_1) K_{i'j'}(t_2) |^{2+\varrho} \right)^{\frac{1}{2+\varrho}}\\
&= \left(\int E(|\eta_{ij,i'j'}|^{2+\varrho}| T_{ij}, T_{i'j'})  K_{ij}(t_1)^{2+\varrho} K_{i'j'}(t_2)^{2+\varrho} f_T(T_{ij}) f_T(T_{i'j'}) dT_{ij}dT_{i'j'}  \right)^{\frac{1}{2+\varrho}} \\
&\le M_{75} h_{c}^{\frac{-2(1+\varrho)}{2+\varrho}},
\end{align*}
which follows from  Conditions (I)[(v)(b), (ii)]  and (III) where $M_{75}$ is some positive constant. This implies
\begin{align}
\vert \text{cov}\big(\eta_{ij,i'j'} K_{ij}(t_1) K_{i'j'}(t_2),  \eta_{lk,l'k'}K_{lk}(t_1) K_{l'k'}(t_2)\big) \vert &\le M_{76} h_{c}^{\frac{-4(1+\varrho)}{2+\varrho}} \left[ \alpha(d(\bm{s}_i,\bm{s}_{i'}),d(\bm{s}_l,\bm{s}_{l'})) \right]^{\frac{\varrho}{2+\varrho}},
\label{eqn:R-cv2}
\end{align}
for some positive constant $M_{76}$.


Suppose 
\begin{align*}
J_1 &= \left\{ (i,i',l,l') : 1 \le i' \neq i, l' \neq l \le n, (l,l') \neq (i,i'), d(\{\bm{s}_i,\bm{s}_{i'}\}, \{\bm{s}_l,\bm{s}_{l'}\}) \le c_n  \right\}
\end{align*}
and 
\begin{align*}
J_2 &= \left\{ (i,i',l,l') : 1 \le i' \neq i, l' \neq l \le n, (l,l') \neq (i,i'), d(\{\bm{s}_i,\bm{s}_{i'}\}, \{\bm{s}_l,\bm{s}_{l'}\}) > c_n  \right\}
\end{align*}
partition the summation over $\{(i,i',l,l') : 1 \le i' \neq i, l' \neq l \le n, (l,l') \neq (i,i')\}$. For a tuple $(i,i',l,l') \in J_1$, without loss of generality we assume that $d(\bm{s}_i, \bm{s}_l)=d(\{\bm{s}_i, \bm{s}_{i'}\}, \{\bm{s}_l, \bm{s}_{l'}\})$.  We write the cardinality 
\begin{align*}
\#\{l : l \neq i, d(\bm{s}_i, \bm{s}_l) \le c_n\} \le C (c_n/\delta_n)^2,
\end{align*}
where $C$ is some constant independent of $i$ and $n$.  Therefore, on the index $J_1$, 
\begin{align}
&  \sum_{(i,i',l,l') \in J_1}  v_{i,i'}v_{l,l'}  \text{cov}\bigg\{ \sum_{j=1}^{N_i} \sum_{j'=1}^{N_{i'}} \eta_{ij,i'j'} L_{ii'} K_{ij}(t_1) K_{i'j'}(t_2), \sum_{k=1}^{N_l} \sum_{k'=1}^{N_{l'}} \eta_{lk,l'k'} L_{ll'} K_{lk}(t_1) K_{l'k'}(t_2) \bigg\} \nonumber\\
& \le O(1)  \sum_{i,i'}^n\underset{d(\bm{s}_i, \bm{s}_l) \le c_n}{\sum_{l,l'}}  v_{i,i'} v_{l,l'} L_{ii'} L_{ll'} N_{i}N_{i'} N_l N_{l'} h_c^{-8/(4+\varrho)} \nonumber \\
&\le n(n-1)^2h_c^{-8/(4+\varrho)}  \max\{ v_{l,l'} N_l N_{l'} \}^2  C(c_n/\delta_n)^2 O(1). \label{eqn:lecn}
\end{align}
We now discuss the covariance on the index set $J_2$. The combination of Lemma \ref{lem:mixing} and (\ref{eqn:R-cv2}) yields that 
\begin{align*}
& \sum_{(i,i',l,l') \in J_2}   v_{i,i'}v_{l,l'}  \text{cov}\bigg\{ \sum_{j=1}^{N_i} \sum_{j'=1}^{N_{i'}} \eta_{ij,i'j'}L_{ii'} K_{ij}(t_1) K_{i'j'}(t_2), \sum_{k=1}^{N_l} \sum_{k'=1}^{N_{l'}} \eta_{lk,l'k'} L_{ll'} K_{lk}(t_1) K_{l'k'}(t_2) \bigg\}\\
& \quad \le O(h_c^{-4(1+\varrho)/(2+\varrho)}) \sum_{(i,i',l,l') \in J_2} v_{i,i'}v_{l,l'} N_iN_{i'} N_l N_{l'} L_{ii'}L_{ll'} [\alpha(d(\bm{s}_i, \bm{s}_{l}))]^{\varrho/2+\varrho}\\
& \quad \le O(h_c^{-4(1+\varrho)/(2+\varrho)}) \sum_{i,i'}^n \underset{d(\bm{s}_i, \bm{s}_{l}) > c_n}{\sum_{l,l'}}  [\alpha(d(\bm{s}_i, \bm{s}_{l}))]^{\varrho/2+\varrho} v_{i,i'}v_{ll'} N_l N_{l'}N_i N_{i'}L_{ii'}L_{ll'}\\
&\quad \le O(h_c^{-4(1+\varrho)/(2+\varrho)}) \sum_{i,i'} \underset{d(\bm{s}_i, \bm{s}_{l}) > c_n}{\sum_{l,l'}}  [\alpha(d(\bm{s}_i, \bm{s}_{l}))]^{\varrho/2+\varrho}   \max\{v_{i,i'}N_iN_{i'}\}^2 \\
&\quad  \le O(h_c^{-4(1+\varrho)/(2+\varrho)})  \max\{v_{i,i'}N_iN_{i'}\}^2 \sum_{t=c_n}^{\infty} [\alpha(t)]^{\varrho/(2+\varrho)} \sum_{i,i'} \underset{t \le d(\bm{s}_i, \bm{s}_{l})  \le t+1}{\sum_{l,l'}} \\
&\quad  \le O(h_c^{-4(1+\varrho)/(2+\varrho)})  \max\{v_{i,i'}N_iN_{i'}\}^2 n(n-1)^2 \sum_{t=c_n}^{\infty} [\alpha(t)]^{\varrho/(2+\varrho)} \left(\frac{t+1}{\delta_n} \right)^2 \\ 
&\quad \le O(h_c^{-4(1+\varrho)/(2+\varrho)})  \max\{v_{i,i'}N_iN_{i'}\}^2 n(n-1)^2 \delta_n^{-2} \sum_{t=c_n}^{\infty} t^2 [\alpha(t)]^{\varrho/(2+\varrho)}.
\end{align*}
Consequently,
\begin{align}
&  \sum_{(i,i',l,l') \in J_2}   v_{i,i'}v_{l,l'}  \text{cov}\bigg\{ \sum_{j=1}^{N_i} \sum_{j'=1}^{N_{i'}} \eta_{ij,i'j'} L_{ii'} K_{ij}(t_1) K_{i'j'}(t_2), \sum_{k=1}^{N_l} \sum_{k'=1}^{N_{l'}} \eta_{lk,l'k'} L_{ll'} K_{lk}(t_1) K_{l'k'}(t_2) \bigg\} \nonumber\\
& \le \gamma_{n3}  h_c^{-2\varrho/(2+\varrho)}  \delta_n^{-2} \sum_{t=c_n}^{\infty} t^2 [\alpha(t)]^{\varrho/(2+\varrho)}, \label{eqn:gecn}
\end{align}
where $\gamma_{n3}= n(n-1)^2 h^{-2} \max\{v_{i,i'}N_iN_{i'}\}^2$. Let $c_n= \left( \gamma_{n3}^{-1} h_c^{2 \varrho/(2+\varrho)} \delta_n^2  \right)^{-1/\kappa}$ for $\kappa>0$. Then it follows from condition (II)(i) that (\ref{eqn:gecn}) goes to zero. From (\ref{eqn:lecn}) we have 
\begin{align*}
\gamma_{n3} h_c^{2\varrho/(4+\varrho)}  (c_n/\delta_n)^2 &= \gamma_{n3} h_c^{2\varrho/(4+\varrho)} \delta_n^{-2} \gamma_{n3}^{2/\kappa} b^{2/\kappa} \delta_n^{-4/\varrho} h_c^{-4\varrho/(\kappa(2+\varrho))}\\
&=  \gamma_{n3}^{1+2/\kappa} \delta_n^{-2(1+2/\kappa)} h_c^{\frac{2\varrho}{4+\varrho}(1-\frac{2(4+\varrho)}{\kappa(2+\varrho)})},
\end{align*}
which goes to zero as $n \rightarrow \infty$ based on Condition (IV)(iv).
\end{proof}

We now provide the $L_2$ convergence rate for the covariance estimator $\widehat{R}(0; t_1,t_2)$.
\begin{proof}[Proof of Lemma \ref{lm:L2-R}]
    The proof is similar to Theorem 4.2 in \citet{zhang2016sparse}. The result follows by doing calculations analogous to variance part in Theorem \ref{th:Rhat} and Proposition \ref{proposition:p00}.  
\end{proof}
\begin{proposition}(Convergence of $P_{00}$-type terms)\label{proposition:p00}
Suppose that Conditions (I)[(i),(ii)], (II)[(ii), (iii)],  (III) and (V)(iii) hold. Let $g:\mathbb{R}^3 \to [-B, B]$ be a $B$-Lipschitz function. Let us define
\begin{eqnarray*}
S_n(\bft)=\sum_{i=1}^n\sum_{i'\neq i} v_{i,i'}\sum_{j=1}^{N_i}\sum_{j'=1}^{N_{i'}}\mathscr{X}_{ii'jj'}(\bft), \qquad \text{for } \bft = (t_1,t_2),
\end{eqnarray*}
where
\begin{eqnarray*}
\mathscr{X}_{ii'jj'}(\bft)=g\left(T_{ij}-t_1,T_{i'j'}-t_2,\|\bfs_i-\bfs_{i'}\|-\|\bfs\|\right)K_{h_c}(T_{ij}-t_1)K_{h_c}(T_{i'j'}-t_2)L_b(\|\bfs_i-\bfs_{i'}\|-\|\bfs\|).
\end{eqnarray*}
Then it follows that
\begin{eqnarray*}
\sup_{\bft \in [0, 1]^2}|S_n(\bft)-\ev[S_n(\bft)]|\cip 0,
\end{eqnarray*}
where $\cip 0$ denotes convergence in probability.
\end{proposition}

\begin{proof}
The proof is divided into four steps. For simplicity in notation, instead of $h_c$ we use $h$ throughout.

\noindent \textbf{Step 1:}
For each $\bft \in [0,1]^2$, let us define  
\begin{eqnarray*}
Y_{ii'}(\bft)=v_{i,i'} \sum_{j=1}^{N_i}\sum_{j'=1}^{N_{i'}}\mathscr{X}_{ii'jj'}(\bft)+v_{i'i} \sum_{j'=1}^{N_{i'}}\sum_{j=1}^{N_i} \mathscr{X}_{i'ij'j}(\bft)=v_{i,i'} \sum_{j=1}^{N_i}\sum_{j'=1}^{N_{i'}}[\mathscr{X}_{ii'jj'}(\bft)+\mathscr{X}_{i'ij'j}(\bft)],
\end{eqnarray*}
here we use the fact that $v_{i,i'}=v_{i'i}$. For a permutation $\sigma_n=(i_1,i_2,\ldots, i_n)$ over $\{1,2,\ldots, n\}$, we define
\begin{eqnarray*}
\psi_\bft(\sigma_n)=\frac{Y_{i_1i_2}(\bft)+Y_{i_3i_4}(\bft)+\ldots+Y_{i_{n-1}i_n}(\bft)}{n/2},
\end{eqnarray*}
where we assume $n$ is even for simplicity. After some computation, we can verify
\begin{eqnarray*}
Y_{ii'}(\bft)&=&Y_{i'i}(\bft),\\
S_n(\bft)&=&\sum_{1\leq i<i'\leq n}Y_{ii'}(\bft)=\frac{1}{2} \sum_{i=1}^n \sum_{i'\neq i}Y_{ii'}(\bft),\\
\frac{n}{2}\sum_{\sigma_n}\psi_\bft(\sigma_n)&=& \frac{n}{2} (n-2)! \sum_{i=1}^n \sum_{i'\neq i}Y_{ii'}(\bft)=n(n-2)! S_n(\bft).
\end{eqnarray*}
Here $\sum_{\sigma_n}$ is over all the permutations of $\{1,2,\ldots, n\}$.
Hence, we conclude that
\begin{eqnarray}
\sum_{\sigma_n}\psi_\bft(\sigma_n)=2(n-2)!S_n(\bft). \label{eq:proporsition:uniform:convergence:permutation}
\end{eqnarray}
 Simple calculation implies that
\begin{eqnarray*}
 \ev\left\{K_h^2(T_{ij}-t)\right\}=\int_0^1\frac{1}{h^2}K^2\left(\frac{x-t}{h}\right)f_T(x)dx=\int_{-C}^{C}\frac{1}{h}K^2\left(y\right)f_T(t+hy)dy\leq Ch^{-1}
\end{eqnarray*}
Hence, for any fixed $\bft=(t_1, t_2)^\top \in [0, 1]^2$, we have
\begin{eqnarray}
var(\mathscr{X}_{ii'jj'}(\bft))&\leq& B^2L_b^2(\|\bfs_i-\bfs_{i'}\|-\|\bfs\|)\ev[K_{h}^2(T_{ij}-t_1)]\ev[K_{h}^2(T_{ij}-t_1)]\nonumber\\
&\lesssim& B^2h^{-2} L_b^2(\|\bfs_i-\bfs_{i'}\|-\|\bfs\|). \label{eq:proporsition:uniform:convergence:var:X}
\end{eqnarray}

\noindent \textbf{Step 2:}
Using (\ref{eq:proporsition:uniform:convergence:permutation}), we see that
\begin{eqnarray}
\pr(S_n(\bft)-\ev[S_n(\bft)]>x)&\leq& e^{-\lambda x}\ev\left\{e^{\lambda (S_n(\bft)-\ev[S_n(\bft)])}\right\}\nonumber\\
&=& e^{-\lambda x}\ev\left\{e^{\frac{\lambda}{2(n-2)!}  \sum_{\sigma_n}(\psi_\bft(\sigma_n)-\ev[\psi_\bft(\sigma_n)])}\right\}\nonumber\\
&=& e^{-\lambda x}\ev\left\{e^{\frac{\lambda}{n(n-2)!}  \sum_{\sigma_n}\frac{n}{2} (\psi_\bft(\sigma_n)-\ev[\psi_\bft(\sigma_n)])}\right\}\nonumber\\
&=& e^{-\lambda x}\ev\left\{e^{\frac{1}{n!}  \sum_{\sigma_n}\frac{\lambda n(n-1)}{2} (\psi_\bft(\sigma_n)-\ev[\psi_\bft(\sigma_n)])}\right\}.\nonumber
\end{eqnarray}
Applying generalized holder's inequality, the above inequality becomes
\begin{eqnarray}
\pr(S_n(\bft)-\ev[S_n(\bft)]>x)\leq e^{-\lambda x}\prod_{\sigma_n}\ev^{1/n!}  \left\{e^{ \frac{\lambda n(n-1)}{2} (\psi_\bft(\sigma_n)-\ev[\psi_\bft(\sigma_n)])}\right\}.\label{eq:proporsition:uniform:convergence:step2:eq1}
\end{eqnarray}
Using  (\ref{eq:proporsition:uniform:convergence:var:X}), we have
\begin{eqnarray*}
\frac{ n(n-1)}{2} \psi_\bft(\sigma_n)&=& (n-1)[Y_{i_1i_2}(\bft)+Y_{i_3i_4}(\bft)+\ldots+Y_{i_{n-1}i_n}(\bft)]\\
&=&(n-1)Y_{i_1i_2}(\bft)+(n-1)Y_{i_3i_4}(\bft)+\ldots+(n-1)Y_{i_{n-1}i_n}(\bft),\\
Y_{ii'}(\bft)&=&v_{i,i'} \sum_{j=1}^{N_i}\sum_{j'=1}^{N_{i'}}[\mathscr{X}_{ii'jj'}(\bft)+\mathscr{X}_{i'ij'j}(\bft)],\\
|\mathscr{X}_{ii'jj'}(\bft)+\mathscr{X}_{i'ij'j}(\bft)|&\leq& Ch^{-2}b^{-1},\\
var(\mathscr{X}_{ii'jj'}(\bft)+\mathscr{X}_{i'ij'j}(\bft))&\leq& Ch^{-2}L_b^2(\|\bfs_i-\bfs_{i'}\|-\|\bfs\|).
\end{eqnarray*}
Hence, Lemma \ref{lemma:mgf:bound}  implies that
\begin{eqnarray*}
&&\ev\left\{e^{\lambda (n-1)[Y_{ii'}(\bft)-\ev(Y_{ii'}(\bft))] }\right\}\\
&=&\ev\left\{e^{\lambda (n-1)\{v_{i,i'} \sum_{j=1}^{N_i}\sum_{j'=1}^{N_{i'}}[\mathscr{X}_{ii'jj'}(\bft)+\mathscr{X}_{i'ij'j}(\bft)]-v_{i,i'} \sum_{j=1}^{N_i}\sum_{j'=1}^{N_{i'}}\ev[\mathscr{X}_{ii'jj'}(\bft)+\mathscr{X}_{i'ij'j}(\bft)]\} }\right\}\\
&=&\prod_{j=1}^{N_i}\prod_{j'=1}^{N_{i'}} \ev\left\{e^{\lambda (n-1)v_{i,i'}\{[\mathscr{X}_{ii'jj'}(\bft)+\mathscr{X}_{i'ij'j}(\bft)]-\ev[\mathscr{X}_{ii'jj'}(\bft)+\mathscr{X}_{i'ij'j}(\bft)]\}}\right\}\\
&\leq&  \prod_{j=1}^{N_i}\prod_{j'=1}^{N_{i'}}   \\ &\times&  \exp\left\{\lambda^2 (n-1)^2 v_{i,i'}^2 Ch^{-2}L_b^2(\|\bfs_i-\bfs_{i'}\|-\|\bfs\|)\left(\frac{e^{\lambda (n-1)v_{i,i'}2Ch^{-2}b^{-1}}-1-\lambda (n-1)v_{i,i'}2Ch^{-2}b^{-1} }{[\lambda (n-1)v_{i,i'}2Ch^{-2}b^{-1}]^2}\right) \right\}\\
&\leq &\prod_{j=1}^{N_i}\prod_{j'=1}^{N_{i'}} \exp\left\{\lambda^2 C(n-1)^2h^{-2} v_{i,i'}^2 L_b^2(\|\bfs_i-\bfs_{i'}\|-\|\bfs\|)\left(\frac{e^{\lambda \widetilde{C}_n }-1-\lambda \widetilde{C}_n }{[\lambda \widetilde{C}_n]^2}\right) \right\}\\
&=&  \exp\left\{\lambda^2 C(n-1)^2h^{-2} v_{i,i'}^2N_i N_{i'} L_b^2(\|\bfs_i-\bfs_{i'}\|-\|\bfs\|)\left(\frac{e^{\lambda \widetilde{C}_n }-1-\lambda \widetilde{C}_n }{[\lambda \widetilde{C}_n]^2}\right) \right\},
\end{eqnarray*}
where  $\widetilde{C}_n= 2C n h^{-2}b^{-1} \sup_{i,i'}v_{i,i'}$. Since $\psi_\bft(\sigma_n)$ is the average of $n/2$ independent random variables, we have
\begin{eqnarray*}
&&\ev\left\{e^{\frac{\lambda n(n-1)}{2} (\psi_\bft(\sigma_n)-\ev[\psi_\bft(\sigma_n)])}\right\}\\
&\leq&  \ev\left\{e^{\lambda (n-1)[Y_{i_1i_2}(\bft)-\ev(Y_{i_1i_2}(\bft))+Y_{i_3i_4}(\bft)-\ev(Y_{i_3i_4}(\bft))+\ldots+Y_{i_{n-1}i_n}(\bft)-\ev(Y_{i_{n-1}i_n}(\bft))] }\right\}\\
&=&\ev\left\{e^{\lambda (n-1)[Y_{i_1i_2}(\bft)-\ev(Y_{i_1i_2}(\bft))] }\right\}\times \ldots \times \ev\left\{e^{\lambda (n-1)[Y_{i_{n-1}i_n}(\bft)-\ev(Y_{i_{n-1}i_n}(\bft))] }\right\}\\
&\leq&  \exp\left\{\lambda^2C(n-1)^2h^{-2} \sum_{k=1,3,\ldots,n-1}  v_{i_k i_{k+1}}^2N_{i_k} N_{i_{k+1}} L_b^2(\|\bfs_{i_k}-\bfs_{i_{k+1}}\|-\|\bfs\|) \left(\frac{e^{\lambda \widetilde{C}_n }-1-\lambda \widetilde{C}_n }{[\lambda \widetilde{C}_n]^2}\right) \right\}\\
&=&\exp\left\{\lambda^2u_{\sigma_n}(\bfs)\left(\frac{e^{\lambda \widetilde{C}_n }-1-\lambda \widetilde{C}_n }{[\lambda \widetilde{C}_n]^2}\right)  \right\},
\end{eqnarray*}
where $u_{\sigma_n}(\bfs)=C(n-1)^2h^{-2} \sum_{k=1,3,\ldots,n-1}  v_{i_k i_{k+1}}^2N_{i_k} N_{i_{k+1}} L_b^2(\|\bfs_{i_k}-\bfs_{i_{k+1}}\|-\|\bfs\|)$ for the permutation $(i_1, i_2, \ldots, i_n)$. Using the above bound and (\ref{eq:proporsition:uniform:convergence:step2:eq1}), we show that
\begin{eqnarray*}
\pr(S_n(\bft)-\ev[S_n(\bft)]>x)&\leq& e^{-\lambda x}\prod_{\sigma_n}\ev^{1/n!}  \left\{e^{ \frac{\lambda n(n-1)}{2} (\psi_\bft(\sigma_n)-\ev[\psi_\bft(\sigma_n)])}\right\}\\
&\leq&e^{-\lambda x}\prod_{\sigma_n}\exp\left\{ \frac{1}{n!}\lambda^2u_{\sigma_n}(\bfs)\left(\frac{e^{\lambda \widetilde{C}_n }-1-\lambda \widetilde{C}_n }{[\lambda \widetilde{C}_n]^2}\right)  \right\}\\ 
&=&e^{-\lambda x}\exp\left\{\frac{1}{n!}
\sum_{\sigma_n}\lambda^2u_{\sigma_n}(\bfs)\left(\frac{e^{\lambda \widetilde{C}_n }-1-\lambda \widetilde{C}_n }{[\lambda \widetilde{C}_n]^2}\right)  \right\}.
\end{eqnarray*}
Simple combinatorics shows that
\begin{eqnarray*}
\sum_{\sigma_n}u_{\sigma_n}(\bfs)=\frac{n}{2}(n-2)! \sum_{i=1}^n \sum_{i'\neq i}C(n-1)^2h^{-2} v_{i,i'}^2N_i N_{i'} L_b^2(\|\bfs_i-\bfs_{i'}\|-\|\bfs\|),
\end{eqnarray*}
which further implies that
\begin{eqnarray*}
\frac{1}{n!}\sum_{\sigma_n}u_{\sigma_n}(\bfs)=\frac{1}{2} \sum_{i=1}^n \sum_{i'\neq i}C(n-1)h^{-2} v_{i,i'}^2N_i N_{i'} L_b^2(\|\bfs_i-\bfs_{i'}\|-\|\bfs\|):=v_n(\bfs).
\end{eqnarray*}
Hence, we conclude that
\begin{eqnarray*}
\pr(S_n(\bft)-\ev[S_n(\bft)]>x)\leq \exp\left\{\frac{v_n(\bfs)}{\widetilde{C}_n^2}(e^{\lambda \widetilde{C}_n }-1-\lambda \widetilde{C}_n)-\lambda x \right\}.
\end{eqnarray*}
In particular, taking $\lambda=\widetilde{C}_n^{-1} \log(1+\widetilde{C}_nx/v_n(\bfs))$ leads to
\begin{eqnarray*}
&&\frac{v_n(\bfs)}{\widetilde{C}_n^2}(e^{\lambda \widetilde{C}_n }-1-\lambda \widetilde{C}_n)-\lambda x\\
&=&\frac{v_n(\bfs)}{\widetilde{C}_n^2}(e^{\lambda \widetilde{C}_n }-1)- \left(\frac{v_n(\bfs)}{\widetilde{C}_n}+x\right)\lambda\\
&=& \frac{v_n(\bfs)}{\widetilde{C}_n^2}\left[1+\frac{\widetilde{C}_nx}{v_n(\bfs)}-1\right]-\left(\frac{v_n(\bfs)}{\widetilde{C}_n}+x\right)\frac{1}{\widetilde{C}_n}\log\left(1+\frac{\widetilde{C}_nx}{v_n(\bfs)}\right)\\
&=&\frac{x}{\widetilde{C}_n}-\left(\frac{v_n(\bfs)}{\widetilde{C}_n^2}+\frac{x}{\widetilde{C}_n}\right)\log\left(1+\frac{\widetilde{C}_nx}{v_n(\bfs)}\right)\\
&=&\frac{v_n(\bfs)}{\widetilde{C}_n^2} \frac{\widetilde{C}_n x}{v_n(\bfs)}-\frac{v_n(\bfs)}{\widetilde{C}_n^2} \left(1+\frac{\widetilde{C}_nx}{v_n(\bfs)}\right)\log\left(1+\frac{\widetilde{C}_nx}{v_n(\bfs)}\right)\\
&=&\frac{v_n(\bfs)}{\widetilde{C}_n^2} [u-(1+u)\log(1+u)],
\end{eqnarray*}
where $u=\widetilde{C}_nx/v_n(\bfs)$. It can be shown that  
\begin{eqnarray*}
u-(1+u)\log(1+u)\leq -\frac{u^2}{2+2u/3}.
\end{eqnarray*}
Therefore, we conclude that
\begin{eqnarray*}
\pr(S_n(\bft)-\ev[S_n(\bft)]>x)\leq   \exp\left\{   -\frac{v_n(\bfs)}{\widetilde{C}_n^2} \frac{u^2}{2+2u/3} \right\}\leq  \exp\left\{-\frac{x^2}{2v_n(\bfs)+2\widetilde{C}_nx/3}\right\}.
\end{eqnarray*}
Similarly, we can show that 
\begin{eqnarray*}
\pr(S_n(\bft)-\ev[S_n(\bft)]<-x)\leq   \exp\left\{-\frac{x^2}{2v_n(\bfs)+2\widetilde{C}_nx/3}\right\}.
\end{eqnarray*}
Hence, we prove that
\begin{eqnarray}
\pr(|S_n(\bft)-\ev[S_n(\bft)]|>x)\leq  2 \exp\left\{-\frac{x^2}{2v_n(\bfs)+2\widetilde{C}_nx/3}\right\}, \label{eq:proporsition:uniform:convergence:step2:conclusion}
\end{eqnarray}
for any $x>0$.

\noindent \textbf{Step 3:} Let $\bft_1,\ldots, \bft_M$ be a sequence of set in $[0, 1]^2$ such that for any $\bft\in [0, 1]^2$ there is a index $j$ such that $\|\bft-\bft_j\|\leq a_n$ for some $a_n\to 0$. Clearly, there is a such sequence satisfying $M\leq (\pi a_n^2)^{-1}$. Therefore, it follows that
\begin{eqnarray*}
\sup_{\bft \in [0, 1]}|S_n(\bft)-\ev[S_n(\bft)]|\leq \sup_{k=1,\ldots,M}|S_n(\bft_k)-\ev[S_n(\bft_k)]|+D_1+D_2,
\end{eqnarray*}
where
\begin{eqnarray*}
D_1=\sup_{\|\bft-\widetilde{\bft}\|\leq a_n}|S_n(\bft)-S_n(\widetilde{\bft})|,\;\;\; D_2=\sup_{\|\bft-\widetilde{\bft}\|\leq a_n}|\ev[S_n(\bft)]-\ev[S_n(\widetilde{\bft})]|.
\end{eqnarray*}
The Lipschitz properties imply
\begin{eqnarray*}
|\mathscr{X}_{ii'jj'}(\bft)-\mathscr{X}_{ii'jj'}(\widetilde{\bft})|\leq Ch^{-3}L_b(\|\bfs_i-\bfs_{i'}\|-\|\bfs\|)\|\bft-\widetilde{\bft}\|.
\end{eqnarray*}
Hence, we have
\begin{eqnarray*}
D_1\leq Ch^{-3}a_n\sum_{i=1}^n\sum_{i'\neq i} v_{i,i'}N_i N_{i'}L_b(\|\bfs_i-\bfs_{i'}\|-\|\bfs\|)
\end{eqnarray*}
Similarly, we can show that 
\begin{eqnarray*}
D_2\leq Ch^{-3}a_n\sum_{i=1}^n\sum_{i'\neq i} v_{i,i}N_i N_{i'}L_b(\|\bfs_i-\bfs_{i'}\|-\|\bfs\|).
\end{eqnarray*}

Using (\ref{eq:proporsition:uniform:convergence:step2:conclusion}), it follows that
\begin{eqnarray}
\pr(\max_{k=1,\ldots, M}|S_n(\bft_k)-\ev[S_n(\bft_k)]|>\epsilon)&\leq&M \max_{1\leq k\leq M}\pr(|S_n(\bft_k)-\ev[S_n(\bft_k)]|>\epsilon) \nonumber \\
&\leq& 2M  \exp\left\{-\frac{\epsilon^2}{2v_n(\bfs)+2\widetilde{C}_n\epsilon/3}\right\}\nonumber\\
&\leq& 2(\pi a_n^2)^{-1} \exp\left\{-\frac{\epsilon^2}{2v_n(\bfs)+2\widetilde{C}_n\epsilon/3}\right\}\nonumber\\
&\leq& 2\pi^{-1} \exp\left\{-2\log(a_n)-\frac{\epsilon^2}{2v_n(\bfs)+2\widetilde{C}_n\epsilon/3}\right\}
.\label{eq:bound:1:proporsition:uniform:convergence}
\end{eqnarray}

\noindent \textbf{Step 4:} We will compete the proof by showing there exists a sequence $a_n\to 0$ such that $D_1$, $D_2$ and (\ref{eq:bound:1:proporsition:uniform:convergence}) converge to zero. Specifically, we need the existence of $a_n$ such that as 
\begin{eqnarray*}
h^{-3}a_n\sum_{i=1}^n\sum_{i'\neq i} v_{i,i'}N_i N_{i'}L_b(\|\bfs_i-\bfs_{i'}\|-\|\bfs\|)&\to& 0,\\
2\log(a_n)+\frac{\epsilon^2}{2v_n(\bfs)+2\widetilde{C}_n\epsilon/3}&\to&\infty,
\end{eqnarray*}
where
\begin{eqnarray*}
v_n(\bfs)=\frac{1}{2} \sum_{i=1}^n \sum_{i'\neq i}C(n-1)h^{-2} v_{i,i'}^2N_i N_{i'} L_b^2(\|\bfs_i-\bfs_{i'}\|-\|\bfs\|), \widetilde{C}_n= 2C n h^{-2}b^{-1} \sup_{i,i'}v_{i,i'}.
\end{eqnarray*}
Using the rate condition  $\sup_{i,i'} n^2 v_{i,i'}N_i N_{i'}\leq C<\infty$, we show that
\begin{eqnarray*}
v_n(\bfs)\leq \frac{C}{nh^2b}, \widetilde{C}_n\leq \frac{C}{nh^2b}, \sum_{i=1}^n\sum_{i'\neq i} v_{i,i'}N_i N_{i'}L_b(\|\bfs_i-\bfs_{i'}\|-\|\bfs\|)\leq C.
\end{eqnarray*}
Hence, it suffices to guarantee that
\begin{eqnarray*}
h^{-3}a_n\to 0, \log(a_n)+nh^2b\to \infty.
\end{eqnarray*}
Because the condition $ \frac{nh^2b}{\log(n)}\to \infty$ implies $h\geq Cn^{-1}$. Hence, taking  $a_n=h^4$ will imply that
\begin{eqnarray*}
h^{-3}a_n=h&\to&0,\\
\log(a_n)+nh^2b&=& 4\log(h)+nh^2b\geq 4\log(C)-\log(n)+nh^2b\to \infty.
\end{eqnarray*}
Here the last convergence is due to  $ \frac{nh^2b}{\log(n)}\to \infty$.

%
\end{proof}

\subsection{Results for $\widehat{R}_e(0,t_1,t_2)$} \label{subsec:remat}
In this section, we provide the point-wise results for the inverse regression function $\widehat{m}(t,y)$ in Theorem \ref{th:1} and provide the $L_2$ convergence rate for the conditional covariance $\widehat{R}_e(0,t_1,t_2)$ in Lemma \ref{lm:2}.
\begin{proof}[Proof of Theorem \ref{th:1}] 
Consider the local least squares expression from  (\ref{eqn:obj-ge}),
\begin{align}
  (\widehat{m}(\bm{y}_t),\widehat{\bm{\alpha}}^T)^T &= \underset{\alpha_0,\bm{\alpha}}{\text{argmin}}\sum_{i=1}^n w_i\sum_{j=1}^{N_i} \{Z_{ij}-\alpha_0-\bm{\alpha}^T(\bm{\mathcal{Y}}_{ij}-\bm{y}_t)\}^2 K_H(\bm{\mathcal{Y}}_{ij}-\bm{y}_t),
                                                  \label{eqn:th1-of1}
\end{align}
where $\bm{y}_t=(t,y)^T$, $\bm{\mathcal{Y}}_{ij}=(T_{ij},Y_i)^T$, and $K_H(\bm{u})=\vert \bm{H}
\vert^{-1/2}K(\bm{H}^{-1/2}\bm{u})$, which means
$K_H(\bm{\mathcal{Y}}_{ij}-\bm{y}_t)=\frac{1}{h_t h_y}K_2(\frac
{T_{ij}-t}{h_t},\frac{Y_i-y}{h_y})$ with $\bm{H}=\text{diag}(h_t^2,h_y^2)$ and $K_2(\cdot)$ is defined in Condition (III).

Let
\begin{align*}
\bm{U}=\begin{bmatrix}
  u_{00} & \bm{u}_{10}^T\\ \bm{u}_{10} & \bm{u}_{11}
\end{bmatrix}, 
\end{align*}
where
\begin{align}
  u_{00} &= \sum_{i=1}^n w_i \sum_{j=1}^{N_i}  K_{H}(\bm{\mathcal{Y}}_{ij}-\bm{y}_t), \label{eqn:thm-mty-u00}
\end{align}
and
\begin{align*}
  \bm{u}_{lk} &= \sum_{i=1}^n w_i \sum_{j=1}^{N_i} (\bm{\mathcal{Y}}_{ij}-\bm{y}_t)^l ((\bm{\mathcal{Y}}_{ij}-\bm{y}_t)^T)^k K_{H}(\bm{\mathcal{Y}}_{ij}-\bm{y}_t),\qquad (l,k) \in \{ (1,0), (1,1) \}.
\end{align*}
Then, we write
\begin{align*}
  \bm{\theta} &:= \begin{pmatrix}
    \widehat{m} -m \\ \widehat{\bm{\alpha}}-\bm{\alpha}
  \end{pmatrix}
  = \bm{U}^{-1} \bm{Q}, \quad \bm{Q}=\begin{bmatrix}
    q_0 \\ \bm{q}_1
  \end{bmatrix},
\end{align*}
where
\begin{align*}
  q_0 &=  \sum_{i=1}^n w_i \sum_{j=1}^{N_i} \eta_{ij} K_{H}(\bm{\mathcal{Y}}_{ij}-\bm{y}_t),
\end{align*}
and
\begin{align*}
  \bm{q}_1 &= \sum_{i=1}^n w_i \sum_{j=1}^{N_i} \eta_{ij}(\bm{\mathcal{Y}}_{ij}-\bm{y}_t) K_{H}(\bm{\mathcal{Y}}_{ij}-\bm{y}_t),
\end{align*}
with $\eta_{ij}= Z_{ij}-m(\bm{y}_t)-\bm{\alpha}^T(\bm{\mathcal{Y}}_{ij}-\bm{y}_t)$. It follows from Conditions (I)[(i),(iv)(a)], (III) and (IV)(i)  that
\begin{align*}
  E\left( \sum_{i=1}^n w_i \sum_{j=1}^{N_i} K_{H}(\bm{\mathcal{Y}}_{ij}-\bm{y}_t)\right) &=  E(K_{H}(\bm{\mathcal{Y}}_{ij}-\bm{y}_t)) \\
  &=   \int K_2(\bm{u}) g(\bm{y}_t+ \bm{H}^{1/2}\bm{u}) d\bm{u} \\
  &=  g(\bm{y}_t)+o(1),
\end{align*}
where the last step follows from the Taylor expansion of $g(\bm{y}_t+ \bm{H}^{1/2}\bm{u})$ at $\bm{y}_t$. By Lemma \ref{lem:den1-var} we obtain that
\begin{align*}
\text{var}\left( \sum_{i=1}^n w_i \sum_{j=1}^{N_i} K_{H}(\bm{\mathcal{Y}}_{ij}-\bm{y}_t)\right) \rightarrow 0 \qquad \text{ as } n \rightarrow \infty. 
\end{align*}
This implies
\begin{align*}
  u_{00}= \sum_{i=1}^n w_i \sum_{j=1}^{N_i} K_{H}(\bm{\mathcal{Y}}_{ij}-\bm{y}_t) 
  &=g(\bm{y}_t)+o_p(1).
\end{align*}
The other entries $\bm{u}_{0,1}$ and $ \bm{u}_{1,1}$ can be handled similarly
\begin{align*}
  \bm{u}_{10}=  \sum_{i=1}^n w_i \sum_{j=1}^{N_i} K_{H}(\bm{\mathcal{Y}}_{ij}-\bm{y}_t) (\bm{\mathcal{Y}}_{ij}-\bm{y}_t)
  &= \bm{H} \bm{\nabla} g \int K_2(\bm{v}) \bm{v}\bm{v}^T d\bm{v} +o_p(\bm{H}\cdot\bm{1}),
\end{align*}
where $\bm{1}$ is  a column vector of ones with length 2, $\bm{\nabla} g$ is the gradient of $g$, and
\begin{align*}
  \bm{u}_{11}= \sum_{i=1}^n w_i \sum_{j=1}^{N_i} K_{H}(\bm{\mathcal{Y}}_{ij}-\bm{y}_t) (\bm{\mathcal{Y}}_{ij}-\bm{y}_t) (\bm{\mathcal{Y}}_{ij}-\bm{y}_t)^T
  &= \bm{H}  g(\bm{y}_t)  +o_p(\bm{H}).
\end{align*}

Therefore, the leading terms in the matrix $\bm{U}_n$ are
\begin{eqnarray*}
  \bm{U} &=& \begin{pmatrix} g
    (\bm{y}_t)+o_p(1) & (\bm{H} \bm{\nabla} g)^T+o_p((\bm{H}\bm{1})^T) \\
    \bm{H} \bm{\nabla} g+o_p(\bm{H}\bm{1}) & \bm{H}g(\bm{y}_t) +
    o_p(\bm{H}) \end{pmatrix}.
  \end{eqnarray*}
  
By applying the formula for matrix inverse in block form similar to \citet{,jiang2014inverse} we obtain that
\begin{eqnarray}
\bm{U}^{-1} &=& \begin{pmatrix} g^{-1}
  (\bm{y}_t)+o_p(1) & -\frac{(\bm{\nabla} g)^T}{g^2(\bm{y}_t)}+o_p(\bm{1}^T) \\
  -\frac{\bm{\nabla} g}{g^2(\bm{y}_t)}+o_p(\bm{1}) & (\bm{H}g(\bm{y}_t))^{-1} +
  o_p(\bm{H}^{-1}) \end{pmatrix}.
  \label{eq:th1-uninv}
\end{eqnarray}

\vspace{2em}
  \textbf{Bias Calculation:} Let $\bm{\mathcal{H}}_m(\bm{y}_t)$ be the second derivative of the function $m(\bm{y}_t)$. Under the Conditions (I)[(i), (iv)(a)] and (III), the  rate of convergence for the leading terms $q_0$ is obtained as
\begin{align*}
  Eq_0 &= \frac{1}{2} \sum_{i=1}^n w_i\sum_{j=1}^{N_i} E \left(K_H(\bm{\mathcal{Y}}_{ij}-\bm{y}_t) (\bm{\mathcal{Y}}_{ij}-\bm{y}_t)^T \bm{\mathcal{H}}_m(\bm{y}_t) (\bm{\mathcal{Y}}_{ij}-\bm{y}_t) \right)\\
  &= \frac{1}{2} \sum_{i=1}^n w_i \sum_{j=1}^{N_i} \int K_2(\bm{v}) \bm{v}^T \bm{H}^{1/2} \bm{\mathcal{H}}_m(\bm{y}_t) \bm{H}^{1/2} \bm{v} g(\bm{y}_t+ \bm{H}^{1/2}\bm{v}) d\bm{v} + o(\text{tr}(\bm{H}))\\
  &=\frac{1}{2} \text{tr}\left\{ \bm{H}^{1/2} \bm{\mathcal{H}}_m(\bm{y}_t) \bm{H}^{1/2} \int \bm{v}\bm{v}^T K_2(\bm{v}) d\bm{v}  \right\}+ o(\text{tr}(\bm{H}))\\
  &= \frac{c(K_2)}{2} g(\bm{y}_t) \text{tr}(\bm{H} \bm{\mathcal{H}}_m(\bm{y}_t)) + o(\text{tr}(\bm{H})),
\end{align*}
where $\int \bm{v}\bm{v}^T K_2(\bm{v}) d\bm{v}=c(K_2)\bm{I}$. Similarly we obtain  $ E \bm{q}_1 = O( \bm{H}^{3/2}\cdot \bm{1})$. Let $\mathcal{X}_y=\{(T_{ij}, Y_i), i=1,\ldots,n, j=1,\ldots,N_i\}$ be the set of design points. We write the bias,
\begin{align*}
  E(\widehat{m}(\bm{y}_t)-m(\bm{y}_t)~|~ \mathcal{X}_y) &= \frac{c(K_2)}{2} \text{tr}(\bm{H}\bm{\mathcal{H}}_m(\bm{y}_t)) + o_p(\text{tr}(\bm{H})).
\end{align*}

\textbf{Variance Calculation}: We now compute the conditional variance of $\widehat{m}(\bm{y}_t)$. By definition
\begin{eqnarray}
  \text{var}(\widehat{m}(\bm{y}_t) ~|~ \mathcal{X}_y) =  \text{var}(\bm{e}_1^T \bm{\theta}~|~ \mathcal{X}_y) 
  =  \text{var}\left\{u^{00} q_0+ \bm{u}^{{10}^T} \bm{q}_1 ~|~ \mathcal{X}_y \right\},
  \label{eq:var-split}
\end{eqnarray}
where $u^{00}$ and $\bm{u}^{10}$ are the elements of the $\bm{U}^{-1}$ matrix in (\ref{eq:th1-uninv}). 
 The  variance of the first term in the expansion of (\ref{eq:var-split}) is
  \begin{align}
    \text{var}(u^{00}q_0~|~ \mathcal{X}_y) &= (u^{00})^2 \bigg[ \sum_{i=1}^n w_i^2 \text{var}\bigg(\sum_{j=1}^{N_i} \eta_{ij}K_H(\bm{\mathcal{Y}}_{ij}-\bm{y}_t) ~|~ \mathcal{X}_y \bigg)\nonumber \\& \qquad + 2 \sum_{i=1}^{n-1}\sum_{i'=i+1}^n w_i w_{i'} \text{cov}\bigg(\sum_{j=1}^{N_i}\eta_{ij} K_{H}(\bm{\mathcal{Y}}_{ij}-\bm{y}_t), \sum_{j'=1}^{N_{i'}} \eta_{i'j'}K_H(\bm{\mathcal{Y}}_{i'j'}-\bm{y}_t)~|~ \mathcal{X}_y\bigg)\bigg] \nonumber \\
    &:= (u^{00})^2  (a+b), \label{eq:var-split-ab}
  \end{align}
  where
  \begin{align}
  \label{eqn:p2-a}
      a &= \sum_{i=1}^n w_i^2 \text{var}\bigg(\sum_{j=1}^{N_i} \eta_{ij}K_H(\bm{\mathcal{Y}}_{ij}-\bm{y}_t) ~|~ \mathcal{X}_y \bigg), \text{ and}\\
      b &= 2 \sum_{i=1}^{n-1}\sum_{i'=i+1}^n w_i w_{i'} \text{cov}\bigg(\sum_{j=1}^{N_i}\eta_{ij} K_{H}(\bm{\mathcal{Y}}_{ij}-\bm{y}_t), \sum_{j'=1}^{N_{i'}} \eta_{i'j'}K_H(\bm{\mathcal{Y}}_{i'j'}-\bm{y}_t) ~|~ \mathcal{X}_y\bigg). \label{eqn:p2-b}
  \end{align}
We show that the term (b) goes to zero as $n \rightarrow \infty$. Hence, the leading term of the variance come from the term (a).

\vspace{2em}
\underline{Term (a) in (\ref{eq:var-split-ab})}: We have the following decomposition
  \begin{align*}
    \text{var} \left(\sum_{j=1}^{N_i} \eta_{ij}K_H(\bm{\mathcal{Y}}_{ij}-\bm{y}_t)~|~ \mathcal{X}_y\right) &= \sum_{j=1}^{N_i} \text{var} \left(\eta_{ij}K_H(\bm{\mathcal{Y}}_{ij}-\bm{y}_t)~|~ \mathcal{X}_y\right) \\ & \qquad  +\sum_{1 \le j \neq j' \le N_i}  \text{cov}\left(\eta_{ij}K_H(\bm{\mathcal{Y}}_{ij}-\bm{y}_t), \eta_{ij'}  K_H(\bm{\mathcal{Y}}_{ij'}-\bm{y}_t) ~|~ \mathcal{X}_y \right).
  \end{align*}
  Standard arguments yield under condition (I)(iv)(b) that
  \begin{align*}
     E[\text{var}(\eta_{ij}K_H(\bm{\mathcal{Y}}_{ij}-\bm{y}_t)~|~ \mathcal{X}_y)] &= E( \text{var}(\eta_{ij}|\bm{\mathcal{Y}}_{ij}) K_H^2(\bm{\mathcal{Y}}_{ij}-\bm{y}_t)) \\
     &=|\bm{H}|^{-1/2}d(K_2) (\Psi(t,y) + \Lambda(t) + \sigma^2) g(\bm{y}_t) (1 + o(1)),
  \end{align*}
  where $d(K_2)=\int K_2^2(\bm{u})d\bm{u}$ and $\text{var}(\eta_{ij}|T_{ij}=t, Y_i=y)=\Psi(t,y) + \Lambda(t) + \sigma^2$.  Moreover,
  \begin{align*}
    \sum_{1\le j \neq j' \le N_i} &E\left[\text{cov}\left(\eta_{ij}K_H(\bm{\mathcal{Y}}_{ij}-\bm{y}_t), \eta_{ij'}  K_H(\bm{\mathcal{Y}}_{ij'}-\bm{y}_t)  ~|~ \mathcal{X}_y \right) \right]\\ &= N_i(N_i-1) \bigg( \frac{\Psi(t,y,t)+ \Lambda(t,t) }{h_y}  g_3(t,y,t) +o(h_y^{-1})\bigg),
  \end{align*}
  where $g_3(\cdot)$ is the joint distribution of $T_{ij}$, $T_{ik}$ and $Y_i$ as mentioned in Conditions (I)(iv)(a). 
 Hence, we obtain the following leading term of the variance
\begin{align*}
    (u^{00})^2  E[(a)] &= \frac{\sum_{i=1}^n N_iw_i^2}{ |\bm{H}|^{1/2} g^2(t,y)}  (\Psi(t,y) + \Lambda(t) + \sigma^2)  \nu_{K_2,2} g(t, y)\\ & \qquad +  \frac{\sum_{i=1}^n N_i (N_i-1)w_i^2}{h_y g^2(t,y)} \times  \{  (\Psi(t,y,t) + \Lambda(t,t)) g_3(t,y,t)  \} + o_p(1)\bigg\}.
\end{align*}

  \vspace{2em}
  \underline{Term (b) in (\ref{eq:var-split-ab})}: The idea is to relate the conditional covariance in part (b) to unconditional covariance which simplifies the details of the proof. Based on (\ref{eqn:c-cov}), we have
  \begin{align*}
    E&\bigg\{  \sum_{i=1}^{n-1}\sum_{i'=i+1}^n w_i w_{i'} \text{cov}\bigg(\sum_{j=1}^{N_i}\eta_{ij} K_{H}(\bm{\mathcal{Y}}_{ij}-\bm{y}_t), \sum_{j'=1}^{N_{i'}} \eta_{i'j'}K_H(\bm{\mathcal{Y}}_{i'j'}-\bm{y}_t) ~|~ \mathcal{X}_y\bigg) \bigg\} \\
       &=  \sum_{i=1}^{n-1}\sum_{i'=i+1}^n w_i w_{i'} \text{cov}\bigg(\sum_{j=1}^{N_i}\eta_{ij} K_{H}(\bm{\mathcal{Y}}_{ij}-\bm{y}_t), \sum_{j'=1}^{N_{i'}} \eta_{i'j'}K_H(\bm{\mathcal{Y}}_{i'j'}-\bm{y}_t) \bigg)\\  &\quad -  \sum_{i=1}^{n-1}\sum_{i'=i+1}^n w_i w_{i'} \text{cov}\bigg(\sum_{j=1}^{N_i} E[\eta_{ij} ~|~ \mathcal{X}_y] K_{H}(\bm{\mathcal{Y}}_{ij}-\bm{y}_t), \sum_{j'=1}^{N_{i'}} E[\eta_{i'j'} ~|~ \mathcal{X}_y] K_H(\bm{\mathcal{Y}}_{i'j'}-\bm{y}_t)\bigg).
\end{align*}
  From  Lemma \ref{lem:den1-var} it directly follows that the first term in the right side of the above equality goes to zero as $n \rightarrow \infty$. Similarly, calculations analogous to Lemma \ref{lem:den1-var} yields that the second term in the right side of the above equality goes to zero as $n \rightarrow \infty$. Hence, the result is proved.
\end{proof}

\begin{lemma} \label{lem:den1-var}
    Under the conditions of Theorem \ref{th:1}, it holds that
    \begin{align*}
      \text{var}\left( \sum_{i=1}^n w_i \sum_{j=1}^{N_i} K_{H}(\bm{\mathcal{Y}}_{ij}-\bm{y}_t)\right) \rightarrow 0 \qquad \text{ as } n \rightarrow \infty,
      \end{align*}
      and
      \begin{align*}
         \sum_{i=1}^{n-1}\sum_{i'=i+1}^n w_i w_{i'} \text{cov}\bigg(\sum_{j=1}^{N_i}\eta_{ij} K_{H}(\bm{\mathcal{Y}}_{ij}-\bm{y}_t), \sum_{j'=1}^{N_{i'}} \eta_{i'j'}K_H(\bm{\mathcal{Y}}_{i'j'}-\bm{y}_t)\bigg),
      \end{align*}
      goes to zero as $n \rightarrow \infty$.
  \end{lemma}

  \begin{proof}
  First, observe that
  \begin{align}
      \text{var}\bigg( & \sum_{i=1}^n w_i \sum_{j=1}^{N_i} K_{H}(\bm{\mathcal{Y}}_{ij}-\bm{y}_t)\bigg) \nonumber \\ &=  \bigg\{ \sum_{i=1}^n w_i^2 \text{var}\bigg[ \sum_{j=1}^{N_i} K_H(\bm{\mathcal{Y}}_{ij}-\bm{y}_t ) \bigg] \nonumber \\ & \qquad  + 2 \sum_{i=1}^{n-1} \sum_{i'=i+1}^n w_i w_{i'} \text{cov}\bigg[\sum_{j=1}^{N_i} K_H(\bm{\mathcal{Y}}_{ij}-\bm{y}_t ), \sum_{j'=1}^{N_{i'}} K_H(\bm{\mathcal{Y}}_{i'j'}-\bm{y}_t)  \bigg] \bigg\} \nonumber \\
      & := ( V_1 +  V_2), \label{eqn:p1-varsplit}
  \end{align}
  where
  \begin{align}
      V_{1} &=   \text{var}\bigg[ \sum_{j=1}^{N_i} K_H(\bm{\mathcal{Y}}_{ij}-\bm{y}_t ) \bigg], \label{eqn:p1-varsplit1}  
  \end{align}
  and
  \begin{align}
      V_{2} &=  2 \sum_{i=1}^{n-1} \sum_{i'=i+1}^n w_iw_{i'} \text{cov}\bigg[\sum_{j=1}^{N_i} K_H(\bm{\mathcal{Y}}_{ij}-\bm{y}_t), \sum_{j'=1}^{N_{i'}} K_H(\bm{\mathcal{Y}}_{i'j'}-\bm{y}_t)  \bigg]. \label{eqn:p1-varsplit2}
  \end{align}
  Under Conditions (I)[(i), (ii)(a)] and (III), elementary calculations yield
  \begin{align}
      \bigg\vert \text{cov} & \bigg[\sum_{j=1}^{N_i} K_H(\bm{\mathcal{Y}}_{ij}-\bm{y}_t ), \sum_{j'=1}^{N_{i'}} K_H(\bm{\mathcal{Y}}_{i'j'}-\bm{y}_t )  \bigg] \bigg\vert \nonumber \\ &= \bigg\vert \sum_{j=1}^{N_i} \sum_{j'=1}^{N_{i'}} \text{cov} \left[K_H(\bm{\mathcal{Y}}_{ij}-\bm{y}_t ), K_H(\bm{\mathcal{Y}}_{i'j'}-\bm{y}_t) \right] \bigg\vert\nonumber \\ &= \bigg\vert\sum_{j=1}^{N_i} \sum_{j'=1}^{N_{i'}} \left\{ E[K_H(\bm{\mathcal{Y}}_{ij}-\bm{y}_t ) K_H(\bm{\mathcal{Y}}_{i'j'}-\bm{y}_t )]- EK_H(\bm{\mathcal{Y}}_{ij}-\bm{y}_t ) EK_H(\bm{\mathcal{Y}}_{i'j'}-\bm{y}_t ) \right\}\bigg\vert \nonumber \\
      &\le  \sum_{j=1}^{N_i} \sum_{j'=1}^{N_{i'}} \bigg\{ \int K_2(\bm{u})K_2(\bm{v}) \left\vert g_{i,i'}(\bm{t}_1,\bm{t}_2)-g(\bm{t}_1)g(\bm{t}_2)  \right\vert d\bm{u}d\bm{v} \bigg\} \nonumber \\
      & \le N_i N_{i'} C, \label{eqn:p1-c1}
  \end{align}
  where $\bm{t}_1=\bm{y}_t+\bm{H}^{1/2}\bm{u}$, $\bm{t}_2=\bm{y}_t+\bm{H}^{1/2}\bm{v}$ and $C$ is defined in Condition (I)[(ii)(a)]. Further, an application of Lemma \ref{lem:mixing} yields  that, for $\varrho > 0$,
  \begin{align*}
    \vert \text{cov}&(K_{H}(\bm{\mathcal{Y}}_{ij}-\bm{y}_t),  K_H(\bm{\mathcal{Y}}_{i'j'}-\bm{y}_t)) \vert\\
  &\le C_{\alpha} \Vert K_{H}(\bm{\mathcal{Y}}_{ij}-\bm{y}_t)   \Vert_{2+\varrho} \Vert K_H(\bm{\mathcal{Y}}_{i'j'}-\bm{y}_t) \Vert_{2+\varrho} \left[ \alpha(d(\bm{s}_i,\bm{s}_{i'})) \right]^{\frac{\varrho}{2+\varrho}}.
  \end{align*}
  By straightforward computations, we obtain
  \begin{align*}
    \Vert K_{H}(\bm{\mathcal{Y}}_{ij}-\bm{y}_t)   \Vert_{2+\varrho} &=\left(E| K_{H}(\bm{\mathcal{Y}}_{ij}-\bm{y}_t)|^{2+\varrho} \right)^{\frac{1}{2+\varrho}}\\
    &= \left(\int  K_H(\bm{u}-\bm{y}_t)^{2+\varrho} g(\bm{u})d\bm{u}  \right)^{\frac{1}{2+\varrho}} \\
    &=|\bm{H}|^{\frac{-(1+\varrho)}{2(2+\varrho)}}\left(\int K(\bm{v})^{2+\varrho} g(\bm{y}_t+ \bm{H}^{1/2}\bm{v})d\bm{v}  \right)^{\frac{1}{2+\varrho}}\\
    &\le M_{11} |\bm{H}|^{\frac{-(1+\varrho)}{2(2+\varrho)}},
  \end{align*}
  where the last step follows from the Conditions (I)[(ii)(a)] and (III) and $M_{11}$ is some positive constant. This implies
  \begin{align}
    \vert \text{cov}(K_{H}(\bm{\mathcal{Y}}_{ij}-\bm{y}_t),  K_H(\bm{\mathcal{Y}}_{i'j'}-\bm{y}_t)) \vert &\le M_{12} |\bm{H}|^{\frac{-2(1+\varrho)}{2(2+\varrho)}}\left[ \alpha(d(\bm{s}_i,\bm{s}_{i'})) \right]^{\frac{\varrho}{2+\varrho}},
  \label{eqn:p1-c2}
  \end{align}
  for some positive constant $M_{12}$.
  Thus is follows from  (\ref{eqn:p1-varsplit2}), (\ref{eqn:p1-c1}) and (\ref{eqn:p1-c2}) that
  \begin{align*}
     |V_2| &\le 2M_{13} \sum_{i=1}^{n-1}\sum_{i'=i+1}^n w_i w_{i'} \sum_{j=1}^{N_i} \sum_{j'=1}^{N_{i'}}  \left\vert \text{cov}( K_{H}(\bm{\mathcal{Y}}_{ij}-\bm{y}_t), K_H(\bm{\mathcal{Y}}_{i'j'}-\bm{y}_t))\right\vert\\
    &\le M_{13} \sum_{0 <d(\bm{s}_i,\bm{s}_{i'})<c_n} w_i w_{i'} \sum_{j=1}^{N_i} \sum_{j'=1}^{N_{i'}} \left\vert \text{cov}(K_{H}(\bm{\mathcal{Y}}_{ij}-\bm{y}_t), K_H(\bm{\mathcal{Y}}_{i'j'}-\bm{y}_t))\right\vert\\
    & \qquad + M_{13} \sum_{d(\bm{s}_i,\bm{s}_{i'})>c_n} w_i w_{i'}  \sum_{j=1}^{N_i} \sum_{j'=1}^{N_{i'}} \left\vert \text{cov}(K_{H}(\bm{\mathcal{Y}}_{ij}-\bm{y}_t), K_H(\bm{\mathcal{Y}}_{i'j'}-\bm{y}_t))\right\vert\\
    &\le M_{13} \sum_{0 <d(\bm{s}_i,\bm{s}_{i'})<c_n} w_iN_i w_{i'}N_{i'} \\ & \qquad \qquad \qquad + M_{13} \sum_{d(\bm{s}_i,\bm{s}_{i'})>c_n} w_iN_i w_{i'}N_{i'} |\bm{H}|^{\frac{-(1+\varrho)}{2+\varrho}} \left[ \alpha(d(\bm{s}_i,\bm{s}_{i'})) \right]^{\frac{\varrho}{2+\varrho}},
  \end{align*}
  where $\sum_{0 <d(\bm{s}_i,\bm{s}_{i'})<c_n}$ stands for the summation over $\{(i,i'):1 \le i,i'\le n, 0 <d(\bm{s}_i,\bm{s}_{i'})<c_n\}$ the cardinality of which is controlled by $n(c_n/\delta_n)^2$ where $\delta_n$ is defined in Conditions (II) and 
  \begin{align*}
  \sum_{d(\bm{s}_i,\bm{s}_{i'})>c_n} &:= \{(i,i'):1\le i, i' \le n, d(\bm{s}_i,\bm{s}_{i'})>c_n\}\\
  & \subseteq \cup_{m=c_n}^{\infty} \{(i,i'):1\le i, i' \le n, m < d(\bm{s}_i,\bm{s}_{i'}) \le m+1 \}.
  \end{align*}
  Therefore,  we write
  \begin{align}
     |V_2| & \le M_{13} n \frac{ |\bm{H}|^{1/2} \max\{w_i N_i\}^2 }{|\bm{H}|^{1/2}} \left(\frac{c_n}{\delta_n} \right)^2  + \frac{M_{13}  \max\{w_i N_i\}^2 }{|\bm{H}|^{1/2}} \sum_{m=c_n}^{\infty} |\bm{H}|^{\frac{-\varrho}{2(2+\varrho)}} n \left(\frac{m+1}{\delta_n}\right)^2 \left[ \alpha(m) \right]^{\frac{\varrho}{2+\varrho}} \nonumber\\
    & \le M_{13}  \frac{|\bm{H}|^{1/2} n\max\{w_i N_i\}^2}{|\bm{H}|^{1/2}} \left(\frac{c_n}{\delta_n} \right)^2  +  \frac{M_{13}  n\max\{w_i N_i\}^2}{|\bm{H}|^{1/2}}  |\bm{H}|^{\frac{-\varrho}{2(2+\varrho)}} (\delta_n)^{-2} \sum_{m=c_n}^{\infty} m^2 \left[ \alpha(m) \right]^{\frac{\varrho}{2+\varrho}}.
  \label{eq:p2-last1}
  \end{align}
  Denote $\gamma_{n1}=n \frac{ \max\{w_i N_i\}^2}{|\bm{H}|^{1/2}}$. Let $c_n$ be the integer part of $\left(|\bm{H}|^{\frac{\varrho}{2(2+\varrho)}} \delta_n^{2} \gamma_{n1}^{-1}  \right)^{-1/\kappa}$ for $\kappa>0$, specified in Condition (II)(i), by which the second part of (\ref{eq:p2-last1}) tends to zero as $n \rightarrow \infty$. Note that
  \begin{align*}
     \gamma_{n1} |\bm{H}|^{1/2}(c_n/\delta_n)^2 &= \gamma_{n1} |\bm{H}|^{1/2} \left(\left(|\bm{H}|^{\frac{\varrho}{2(2+\varrho)}} \delta_n^{2} \gamma_{n1}^{-1}\right)^{-1/\kappa}\delta_n^{-1}\right)^2\\ &= \gamma_{n1}^{1+2/\kappa} \delta_n^{-2(1+2/\kappa)} |\bm{H}|^{\frac{1}{2}-\frac{\varrho}{\kappa(2+\varrho)}},
  \end{align*}
  which tends to zero following from the Condition (IV)(i). Therefore, the term $ V_2$  goes to zero as $n \rightarrow \infty$.
  
  Now, we consider the term $V_1$ in (\ref{eqn:p1-varsplit}) as defined in (\ref{eqn:p1-varsplit1})
    \begin{align*}
       V_1 &= \sum_{i=1}^n w_i^2 \text{var} \left(\sum_{j=1}^{N_i} K_H(\bm{\mathcal{Y}}_{ij}-\bm{y}_t) \right) \\ &=  \sum_{i=1}^n  w_i^2\sum_{j=1}^{N_i} \text{var}\left(  K_H(\bm{\mathcal{Y}}_{ij}-\bm{y}_t)\right)  \\& \qquad + \sum_{i=1}^n w_i^2  \sum_{1 \le j \neq j' \le N_i} \text{cov} \left\{K_H(\bm{\mathcal{Y}}_{ij}-\bm{y}_t), K_H(\bm{\mathcal{Y}}_{ij'}-\bm{y}_t)  \right\}.
    \end{align*}
    Standard arguments yield that
    \begin{align*}
       E(K_H^2(\bm{\mathcal{Y}}_{ij}-\bm{y}_t)) - \left(E K_H(\bm{\mathcal{Y}}_{ij}-\bm{y}_t)\right)^2 = |\bm{H}|^{-1/2}d(K_2) g(\bm{y}_t) + o(|\bm{H}|^{-1/2}),
    \end{align*}
    where $d(K_2)=\int K_2^2(\bm{u})d\bm{u}$. Analogously, we obtain
    \begin{align*}
      \sum_{1\le j \neq j' \le N_i} \text{cov} \left\{ K_H(\bm{\mathcal{Y}}_{ij}-\bm{y}_t), K_H(\bm{\mathcal{Y}}_{ij'}-\bm{y}_t)  \right\} &= N_i(N_i-1) \bigg( h_y^{-1}  g_3(t,y,t) +o(h_y^{-1})\bigg),
    \end{align*}
    where $g_3(\cdot)$ is the joint distribution of $T_{ij}$, $T_{ik}$ and $Y_i$ as mentioned in Conditions (I)(ii)(a). 
    Therefore, we can write the term $V_1$ in (\ref{eqn:p1-varsplit1}) as
    \begin{align*}
      V_1 &=   \frac{\sum_{i=1}^n w_i^2 N_i}{|\bm{H}|^{1/2}} d(K_2) g(\bm{y}_t) +  \frac{\sum_{i=1}^n N_i(N_i-1) w_i^2}{h_y}   g_3(t,y,t) + o(1).
    \end{align*}
  Consequently, it follows from Condition (IV)(i) and condition $\frac{\sum_{i=1}^n N_i(N_i-1) w_i^2}{h_y} \rightarrow 0$ as $n \rightarrow \infty$, that
  \begin{align*}
      ( V_1 +  V_2) \rightarrow 0.
  \end{align*}
  

  We now turn our attention to term (b) in (\ref{eqn:p2-b})
  and show that it goes to zero as $n \rightarrow \infty$.
  
 Set $L_n= |\bm{H}|^{-1/(4+\varrho)} $ for  $\varrho>0$ defined in Condition (II)(i). Define $\eta_{ij}=\eta_{ij1}+ \eta_{ij2}$ where $\eta_{ij1}:= \eta_{ij}I_{\{|\eta_{ij}| \le L_n\}}$ and $\eta_{ij2} := \eta_{ij}I_{\{|\eta_{ij}| > L_n\}}$. Let
  \begin{align*}
      \chi_{ijk}(\bm{y}_t) := \eta_{ijk}K_H(\bm{\mathcal{Y}}_{ij}-\bm{y}_t) \text{  and   } \varpi_{ijk}(\bm{y}_t) := \chi_{ijk}(\bm{y}_t) - E \chi_{ijk}(\bm{y}_t), \qquad k=1,2.  
  \end{align*}
Observe that $\varpi_{ij}(\bm{y}_t)=\varpi_{ij1}(\bm{y}_t)+\varpi_{ij2}(\bm{y}_t)$. Therefore,
\begin{align}
    E\varpi_{ij}(\bm{y}_t)\varpi_{i'j'}(\bm{y}_t) &= E\varpi_{ij1}(\bm{y}_t)\varpi_{i'j'1}(\bm{y}_t) + E\varpi_{ij1}(\bm{y}_t)\varpi_{i'j'2}(\bm{y}_t) \nonumber\\ & \qquad \qquad +E\varpi_{ij2}(\bm{y}_t)\varpi_{i'j'1}(\bm{y}_t)+ E\varpi_{ij2}(\bm{y}_t)\varpi_{i'j'2}(\bm{y}_t). \label{eqn:var-split-trun}
\end{align}
  By Cauchy-Schwartz inequality and Condition (I)(ii)(c), we first note that
  \begin{align*}
      |E\varpi_{ij1}(\bm{y}_t)&\varpi_{i'j'2}(\bm{y}_t)| \\ & \le  \{E \chi_{ij1}^2(\bm{y}_t) \}^{1/2} \{E \chi_{i'j'2}^2(\bm{y}_t) \}^{1/2} \\ & \le M_{21} |\bm{H}|^{-1/4}  \{ E(E(|\eta_{i'j'}|^{2} I_{\{|\eta_{i'j'}| > L_n\}} |\bm{\mathcal{Y}}_{i'j'}) K_H^2(\bm{\mathcal{Y}}_{i'j'}-\bm{y}_t)) \}^{1/2} \\
      & \le M_{21} |\bm{H}|^{-1/4}  \{ |\bm{H}|^{-1/2} L_n^{-\varrho} E(|\eta_{i'j'}|^{2+\varrho} |\bm{y}_t) \}^{1/2} \\
      & \le M_{22} |\bm{H}|^{-1/2} L_n^{-\varrho/2}= M_{22} |\bm{H}|^{-2/(4+\varrho)},
  \end{align*}
  where $M_{21}$ and $M_{22}$ are some constants. Analogous arguments yield
  \begin{align*}
      E\varpi_{ij2}(\bm{y}_t)\varpi_{i'j'1}(\bm{y}_t) &\le M_{22} |\bm{H}|^{-1/2} L_n^{-\varrho/2}= M_{22} |\bm{H}|^{-2/(4+\varrho)} \qquad \text{ and } \\
      E\varpi_{ij2}(\bm{y}_t)\varpi_{i'j'2}(\bm{y}_t) &\le M_{22} |\bm{H}|^{-1/2} L_n^{-\varrho} = M_{22} |\bm{H}|^{(\varrho-4)/(2(4+\varrho))}.
  \end{align*}
  Now it remains to show that
  \begin{align*}
    E\varpi_{ij1}(\bm{y}_t)\varpi_{i'j'1}(\bm{y}_t)  & =    \text{cov}(\eta_{ij1} K_{H}(\bm{\mathcal{Y}}_{ij}-\bm{y}_t),  \eta_{i'j'1}K_H(\bm{\mathcal{Y}}_{i'j'}-\bm{y}_t))  \\
    & = E(\eta_{ij1} K_{H}(\bm{\mathcal{Y}}_{ij}-\bm{y}_t)  \eta_{i'j'1}K_H(\bm{\mathcal{Y}}_{i'j'}-\bm{y}_t))\\ & \qquad \qquad-E
    (\eta_{ij1} K_{H}(\bm{\mathcal{Y}}_{ij}-\bm{y}_t)) E(\eta_{i'j'1}K_H(\bm{\mathcal{Y}}_{i'j'}-\bm{y}_t))\\
    & = \int  K_2(\bm{u})K_2(\bm{v}) \\
    & \qquad \qquad \times \{\Psi_{i,i',1}(\bm{w}_1,\bm{w}_2) g_{i,i'}(\bm{w}_1,\bm{w}_2)- \Psi_1(\bm{w}_1) g(\bm{w}_1)\Psi_1(\bm{w}_2) g(\bm{w}_2) \} d\bm{u}d\bm{v},
  \end{align*}
  where $\Psi_{i,i',1}(\bm{w}_1,\bm{w}_2) := E(\eta_{ij1}\eta_{i'j'1}|\bm{\mathcal{Y}}_{ij}=\bm{w}_1, \bm{\mathcal{Y}}_{i'j'}=\bm{w}_2)$, $\Psi_1(\bm{w}_1) := E(\eta_{ij1}|\bm{\mathcal{Y}}_{ij}=\bm{w}_1)$, with $\bm{w}_1=\bm{y}_t+ \bm{H}^{1/2}\bm{u}$ and $\bm{w}_2=\bm{y}_t + \bm{H}^{1/2}\bm{v}$. Since, by definition $|\eta_{ij1}| \le L_n$, we have $|\Psi_{i,i',1}(\bm{w}_1,\bm{w}_2)|\le L_n^2$
   and $|\Psi_1(\bm{w}_1) \times \Psi_1(\bm{w}_2)| \le L_n^2$. Thus
\begin{align}
  |\Psi_{i,i',1}&(\bm{w}_1,\bm{w}_2) g_{i,i'}(\bm{w}_1,\bm{w}_2)- \Psi_1(\bm{w}_1) g(\bm{w}_1)\Psi_1(\bm{w}_2) g(\bm{w}_2)| \nonumber\\
  & \le |\Psi_{i,i',1}(\bm{w}_1,\bm{w}_2) (g_{i,i'}(\bm{w}_1,\bm{w}_2)- g(\bm{w}_1)g(\bm{w}_2))| \nonumber \\
  & \qquad + |(\Psi_{i,i',1}(\bm{w}_1,\bm{w}_2)-\Psi_1(\bm{w}_1)\Psi_1(\bm{w}_2)) g(\bm{w}_1)g(\bm{w}_2)|\nonumber \\ &\le  L_n^2 |g_{i,i'}(\bm{w}_1,\bm{w}_2)- g(\bm{w}_1)g(\bm{w}_2)|+ 2L_n^2g(\bm{w}_1)g(\bm{w}_2)  \label{eq:den-cond}
\end{align}
Therefore, for some constant $M_{23}$,
\begin{align*}
    |E\varpi_{ij1}(\bm{y}_t)\varpi_{i'j'1}(\bm{y}_t)| \le M_{23}L_n^2.  
\end{align*}
Consequently, from (\ref{eqn:var-split-trun}), we obtain
\begin{align}
    |E\varpi_{ij}(\bm{y}_t)\varpi_{i'j'}(\bm{y}_t)| \le M_{22} |\bm{H}|^{-1/2} L_n^{-\varrho/2} + M_{23}L_n^2 = M_{24} |\bm{H}|^{-2/(4+\varrho)}. \label{eq:th1b-fcov}
\end{align}
Hence, it follows from (\ref{eq:th1b-fcov})  that
\begin{align*}
  \left\vert \text{cov}(\sum_{j=1}^{N_i}\eta_{ij} K_{H}(\bm{\mathcal{Y}}_{ij}-\bm{y}_t), \sum_{j'=1}^{N_{i'}} \eta_{i'j'}K_H(\bm{\mathcal{Y}}_{i'j'}-\bm{y}_t)) \right\vert &\le M_{24} N_i N_{i'}|\bm{H}|^{-2/(4+\varrho)}.  
\end{align*}
Further, applying Lemma \ref{lem:mixing} we obtain that, for $\varrho >0$, 
\begin{align*}
  \vert \text{cov}&(\eta_{ij} K_{H}(\bm{\mathcal{Y}}_{ij}-\bm{y}_t),  \eta_{i'j'}K_H(\bm{\mathcal{Y}}_{i'j'}-\bm{y}_t)) \vert\\
&\le C_{\alpha} \Vert \eta_{ij} K_{H}(\bm{\mathcal{Y}}_{ij}-\bm{y}_t)   \Vert_{2+\varrho} \Vert \eta_{i'j'}K_H(\bm{\mathcal{Y}}_{i'j'}-\bm{y}_t) \Vert_{2+\varrho} \left[ \alpha(d(\bm{s}_i,\bm{s}_{i'})) \right]^{\frac{\varrho}{2+\varrho}}.
\end{align*}
Calculations due to change of variable yield
\begin{align*}
  \Vert \eta_{ij}& K_{H}(\bm{\mathcal{Y}}_{ij}-\bm{y}_t)   \Vert_{2+\varrho} \\&=\left(E|\eta_{ij} K_{H}(\bm{\mathcal{Y}}_{ij}-\bm{y}_t)|^{2+\varrho} \right)^{\frac{1}{2+\varrho}}\\
  &= \left(\int E(|\eta_{ij}|^{2+\varrho}|\bm{\mathcal{Y}}_{ij}=\bm{u})  K_H(\bm{u}-\bm{y}_t)^{2+\varrho} g(\bm{u})d\bm{u}  \right)^{\frac{1}{2+\varrho}} \\
  &=|\bm{H}|^{\frac{-(1+\varrho)}{2(2+\varrho)}}\left(\int E(|\eta_{ij}|^{2+\varrho}|\bm{\mathcal{Y}}_{ij}=\bm{y}_t+ \bm{H}^{1/2}\bm{v})  K(\bm{v})^{2+\varrho} g(\bm{y}_t + \bm{H}^{1/2}\bm{v})d\bm{v}  \right)^{\frac{1}{2+\varrho}}\\
  &\le M_{25} |\bm{H}|^{\frac{-(1+\varrho)}{2(2+\varrho)}},
\end{align*}
where the last step follows from the Conditions (I)[(ii)(a), (ii)(c)] and (III) and $M_{25}$ is some positive constant. This implies
\begin{align*}
  \vert \text{cov}(\eta_{ij} K_{H}(\bm{\mathcal{Y}}_{ij}-\bm{y}_t),  \eta_{i'j'}K_H(\bm{\mathcal{Y}}_{i'j'}-\bm{y}_t)) \vert &\le M_{26} |\bm{H}|^{\frac{-(1+\varrho)}{(2+\varrho)}}\left[ \alpha(d(\bm{s}_i,\bm{s}_{i'})) \right]^{\frac{\varrho}{2+\varrho}},
\end{align*}
for some positive constant $M_{26}$.
Thus is follows from the definition of term (b) in (\ref{eq:var-split-ab})
\begin{align*}
   |(b)| &\le 2M_{27} \sum_{i=1}^{n-1}\sum_{i'=i+1}^n w_i w_{i'} \sum_{j=1}^{N_i} \sum_{j'=1}^{N_{i'}}  \left\vert \text{cov}(\eta_{ij} K_{H}(\bm{\mathcal{Y}}_{ij}-\bm{y}_t), \eta_{i'j'}K_H(\bm{\mathcal{Y}}_{i'j'}-\bm{y}_t))\right\vert\\
  &\le M_{27} \sum_{0 <d(\bm{s}_i,\bm{s}_{i'})<c_n} w_i w_{i'} \sum_{j=1}^{N_i} \sum_{j'=1}^{N_{i'}} \left\vert \text{cov}(\eta_{ij} K_{H}(\bm{\mathcal{Y}}_{ij}-\bm{y}_t), \eta_{i'j'}K_H(\bm{\mathcal{Y}}_{i'j'}-\bm{y}_t))\right\vert\\
  & \qquad + M_{27} \sum_{d(\bm{s}_i,\bm{s}_{i'})>c_n} w_i w_{i'} \sum_{j=1}^{N_i} \sum_{j'=1}^{N_{i'}} \left\vert \text{cov}(\eta_{ij} K_{H}(\bm{\mathcal{Y}}_{ij}-\bm{y}_t), \eta_{i'j'}K_H(\bm{\mathcal{Y}}_{i'j'}-\bm{y}_t))\right\vert\\
  &\le M_{27}\sum_{0 <d(\bm{s}_i,\bm{s}_{i'})<c_n} |\bm{H}|^{-2/(4+\varrho)} w_iN_iw_{i'}N_{i'} \\ & \qquad + M_{27} \sum_{d(\bm{s}_i,\bm{s}_{i'})>c_n} |\bm{H}|^{\frac{-(1+\varrho)}{2+\varrho}} w_i N_i w_{i'}N_{i'} \left[ \alpha(d(\bm{s}_i,\bm{s}_{i'})) \right]^{\frac{\varrho}{2+\varrho}},
\end{align*}
where $\sum_{0 <d(\bm{s}_i,\bm{s}_{i'})<c_n}$ stands for the summation over $\{(i,i'):1 \le i,i'\le n, 0 <d(\bm{s}_i,\bm{s}_{i'})<c_n\}$ the cardinality of which is controlled by $n(c_n/\delta_n)^2$ where $\delta_n$ is defined in Conditions (II) and 
\begin{align*}
\sum_{d(\bm{s}_i,\bm{s}_{i'})>c_n} &:= \{(i,i'):1\le i, i' \le n, d(\bm{s}_i,\bm{s}_{i'})>c_n\}\\
& \subseteq \cup_{m=c_n}^{\infty} \{(i,i'):1\le i, i' \le n, m < d(\bm{s}_i,\bm{s}_{i'}) \le m+1 \}.
\end{align*}
Therefore,
\begin{align}
  |(b)| & \le M_{27} \gamma_{n1} (c_n/\delta_n)^2 |\bm{H}|^{\frac{\varrho}{2(4+\varrho)}}  \nonumber \\ & \qquad \qquad \qquad + M_{27} \gamma_{n1} \sum_{m=c_n}^{\infty}  |\bm{H}|^{\frac{-\varrho}{2(2+\varrho)}} \left(\frac{m+1}{\delta_n}\right)^2 \left[ \alpha(m) \right]^{\frac{\varrho}{2+\varrho}} \nonumber\\
  & \le \gamma_{n1} |\bm{H}|^{\varrho/(2(4+\varrho))} (c_n/\delta_n)^2 \nonumber \\ & \qquad \qquad \qquad + \gamma_{n1} |\bm{H}|^{\frac{-\varrho}{2(2+\varrho)}} \delta_n^{-2} \sum_{m=c_n}^{\infty} m^2 \left[ \alpha(m) \right]^{\frac{\varrho}{2+\varrho}}.
\label{eq:b-last1}
\end{align}
Let $c_n$ be the integer part of $\left(\gamma_{n1}^{-1} |\bm{H}|^{\frac{\varrho}{2(2+\varrho)}} \delta_n^{2}\right)^{-1/\kappa}$ for $\kappa>0$, specified in Condition (II)(i), by which the second part of (\ref{eq:b-last1}) tends to zero as $n \rightarrow \infty$. Note that
\begin{align*}
  \gamma_{n1} |\bm{H}|^{\frac{\varrho}{2(4+\varrho)}} \left(\frac{c_n}{\delta_n} \right)^2 &= \gamma_{n1} |\bm{H}|^{\frac{\varrho}{2(4+\varrho)}} \left(\left( \gamma_{n1}^{-1} |H|^{\frac{\varrho}{2(2+\varrho)}} \delta_n^{2}\right)^{-1/\kappa}\delta_n^{-1}\right)^2\\ &= \gamma_{n1}^{1+2/\kappa} \delta_n^{-2(1+2/\kappa)} |\bm{H}|^{\frac{\varrho}{2(4+\varrho)}-\frac{\varrho}{\kappa(2+\varrho)}} \\ &= \gamma_{n1}^{1+2/\kappa} \delta_n^{-2(1+2/\kappa)} |\bm{H}|^{\frac{\varrho}{2(4+\varrho)}\left(1-\frac{2(4+\varrho)}{\kappa(2+\varrho)}\right)},
\end{align*}
which tends to zero following from the Condition (IV)(ii) for $\kappa > \frac{2(4+\varrho)}{2+\varrho}$.
  \end{proof}


\begin{proposition}\label{proposition:uniform:convergence:u00} (Convergence of $u_{00}-$type terms)
Suppose that conditions (I)[(i),(ii)], (III), (IV)(ii) and (V)(ii) hold.  Let $g(\cdot, \cdot)$ be a $B$-Lipschitz function for some $B>0$ and $h_t$ and $h_y$ are of the same order $h$.
Define $S_n(t,y)=\sum_{i=1}^n w_i\sum_{j=1}^{N_i} g(T_{ij}-t, Y_{ij}-y)K_H(T_{ij}-t, Y_{ij}-y)$. 
Then it follows that 
\begin{align*}
\sup_{t\in [0, 1],y\in \mathfrak{B}}|S_n(t,y)-\ev[S_n(t,y)]| &= o_P(1)
\end{align*}
where  $\mathfrak{B}$ is a compact set of $\mathbb{R}$.
\end{proposition}
\begin{proof}
The proof is divided into two steps.

\noindent \textbf{Step 1:}
By definition of $\Delta_n$ in (\ref{eqn:sc-2}) and without loss of generality, we can assume $\bfs_1,\ldots, \bfs_n\in [0, C\Delta_n]^2:=\Omega_n$ for some $C>0$. If not, we can shift the points. For each pair $i, k$, we define rectangles
\begin{eqnarray*}
A_{ku}^{(1)}&=&\{\bfs=(s_1, s_2): 2(k-1)p\leq s_{1}<(2k-1)p, 2(u-1)p\leq s_{2}<(2u-1)p\},\\
A_{ku}^{(2)}&=&\{\bfs=(s_1, s_2): 2(k-1)p\leq s_{1}<(2k-1)p, (2u-1)p\leq s_{2}<2up\},\\
A_{ku}^{(3)}&=&\{\bfs=(s_1, s_2): (2k-1)p\leq s_{1}<2kp, 2(u-1)p\leq s_{2}<(2u-1)p\},\\
A_{ku}^{(4)}&=&\{\bfs=(s_1, s_2): (2k-1)p\leq s_{1}<2kp, (2u-1)p\leq s_{2}<2up\}.
\end{eqnarray*}
Correspondingly, let us define
\begin{eqnarray*}
V_{ku}^{(1)}&=&\sum_{i=1}^n I(\bfs_i \in A_{ku}^{(1)})w_i\sum_{j=1}^{N_i} g(T_{ij}-t, Y_{ij}-y)K_H(T_{ij}-t, Y_{ij}-y),\\
V_{ku}^{(2)}&=&\sum_{i=1}^n I(\bfs_i \in A_{ku}^{(2)})w_i\sum_{j=1}^{N_i} g(T_{ij}-t, Y_{ij}-y)K_H(T_{ij}-t, Y_{ij}-y),\\
V_{ku}^{(3)}&=&\sum_{i=1}^n I(\bfs_i \in A_{ku}^{(3)})w_i\sum_{j=1}^{N_i} g(T_{ij}-t, Y_{ij}-y)K_H(T_{ij}-t, Y_{ij}-y),\\
V_{ku}^{(4)}&=&\sum_{i=1}^n I(\bfs_i \in A_{ku}^{(4)})w_i\sum_{j=1}^{N_i} g(T_{ij}-t, Y_{ij}-y)K_H(T_{ij}-t, Y_{ij}-y).
\end{eqnarray*}
It is not difficult to verify that
\begin{eqnarray}
S_n(t,y)=\sum_{k=1}^M \sum_{u=1}^M(V_{ku}^{(1)}+V_{ku}^{(2)}+V_{ku}^{(3)}+V_{ku}^{(4)}),\label{eq:proposition:uniform:convergence:u00:step1:eq0}
\end{eqnarray}
where $M$ is an integer such that $M\leq C\Delta_n/p$. 

Notice that 
\begin{align*}
|{V}_{ku}^{(1)}|\leq B h_t^{-1}h_y^{-1}\sum_{i=1}^n I(\bfs_i\in A_{ku}^{(1)})w_iN_i\leq B h_t^{-1}h_y^{-1}Cp^2\delta_n^{-2} \sup_{i}w_iN_i \leq C n^{-1}h_t^{-1}h_y^{-1}p^2\delta_n^{-2}
\end{align*}
, Lemma 1.2 of \cite{bosq2012nonparametric} implies that there is a sequence of independent variables $\{W_{ku}\}$ such that ${V}_{ku}^{(1)}$ and $W_{ku}$ share the same distribution, and further
\begin{eqnarray}
\pr(|[{V}_{ku}^{(1)}-\ev({V}_{ku}^{(1)})]-[W_{ku}-\ev(W_{ku})]|>x)&\leq&C \sqrt{\frac{C n^{-1}h_t^{-1}h_y^{-1}p^2\delta_n^{-2}}{x}} \alpha(A_{ku}^{(1)},\Omega_n/A_{ku}^{(1)})\nonumber\\
&\leq& C \sqrt{\frac{C n^{-1}h_t^{-1}h_y^{-1}p^2\delta_n^{-2}}{x}} \psi(Cp^2\delta_n^{-2} ,n)\phi(p).\nonumber\\
\label{eq:proposition:uniform:convergence:u00:step1:eq1}
\end{eqnarray}

Similar to the proof in  Theorem \ref{th:1}, we can show that
\begin{eqnarray}
\sum_{k=1}^M \sum_{u=1}^M var(W_{ku})\leq C\left(\sum_{i=1}^n w_i^2 N_i h_t^{-1}h_y^{-1} + \sum_{i=1}^n N_i(N_i-1)w_i^2h \right):= C V_n^2 \label{eq:proposition:uniform:convergence:u00:step1:eq2}
\end{eqnarray}
and
\begin{eqnarray}
|W_{ku}|\leq B h_t^{-1}h_y^{-1}\sum_{i=1}^n I(\bfs_i\in A_{ku}^{(1)})w_iN_i\leq B h_t^{-1}h_y^{-1}Cp^2\delta_n^{-2} \sup_{i}w_iN_i \leq C n^{-1}h_t^{-1}h_y^{-1}p^2\delta_n^{-2}. \label{eq:proposition:uniform:convergence:u00:step1:eq3}
\end{eqnarray}
For any $\epsilon>0$, simple calculation leads to
\begin{eqnarray*}
&&\pr(|\sum_{k=1}^M \sum_{u=1}^M [V_{ku}^{(1)}-\ev(V_{ku}^{(1)})]|>2\epsilon)\\
&\leq&  \pr(|\sum_{k=1}^M \sum_{u=1}^M [V_{ku}^{(1)}-\ev(V_{ku}^{(1)})]|>2\epsilon, |[V_{ku}^{(1)}-\ev(V_{ku}^{(1)})]-[W_{ku}-\ev(W_{ku})]|\leq \epsilon/M^2, \textrm{ for all } k,u)\\
&&+\sum_{k=1}^M \sum_{u=1}^M\pr( |[V_{ku}^{(1)}-\ev(V_{ku}^{(1)})]-[W_{ku}-\ev(W_{ku})]|> \epsilon/M^2)\\
&\leq&  \pr(|\sum_{k=1}^M \sum_{u=1}^M [W_{ku}-\ev(W_{ku})]|>\epsilon)+\sum_{k=1}^M \sum_{u=1}^M\pr( |[V_{ku}^{(1)}-\ev(V_{ku}^{(1)})]-[W_{ku}-\ev(W_{ku})]|> \epsilon/M^2)\\
&:=&R_1+R_2.
\end{eqnarray*}
Using (\ref{eq:proposition:uniform:convergence:u00:step1:eq2}) and (\ref{eq:proposition:uniform:convergence:u00:step1:eq3}), Bernstein inequality implies that
\begin{eqnarray*}
R_1\leq 2\exp\left\{-\frac{\epsilon^2}{2V_n^2+\frac{2}{3}C n^{-1}h_t^{-1}h_y^{-1}p^2\delta_n^{-2}\epsilon}\right\}.
\end{eqnarray*}
The definition of $M$ and (\ref{eq:proposition:uniform:convergence:u00:step1:eq1}) together show that
\begin{eqnarray*}
R_2\leq CM^2 \sqrt{\frac{C M^2n^{-1}h_t^{-1}h_y^{-1}p^2\delta_n^{-2}}{\epsilon}}  \psi(Cp^2\delta_n^{-2} ,n)\phi(p)\leq \frac{C\Delta_n^2}{p^2} \psi(Cp^2\delta_n^{-2} ,n)\phi(p))\sqrt{\frac{C n^{-1}h_t^{-1}h_y^{-1}\Delta_n^2\delta_n^{-2}}{\epsilon}}.
\end{eqnarray*}
Combining the above three in equalities, we conclude that
\begin{eqnarray}
\pr(|\sum_{k=1}^M \sum_{u=1}^M [V_{ku}^{(1)}-\ev(V_{ku}^{(1)})]|>2\epsilon)&\leq& 2\exp\left\{-\frac{\epsilon^2}{2V_n^2+\frac{2}{3}C n^{-1}h_t^{-1}h_y^{-1}p^2\delta_n^{-2}\epsilon}\right\}\nonumber\\
&&+ \frac{C\Delta_n^2}{p^2} \psi(Cp^2\delta_n^{-2} ,n)\phi(p))\sqrt{\frac{C n^{-1}h_t^{-1}h_y^{-1}\Delta_n^2\delta_n^{-2}}{\epsilon}}.\nonumber\\ \label{eq:proposition:uniform:convergence:u00:step1:eq4}
\end{eqnarray}

Similar inequalities as (\ref{eq:proposition:uniform:convergence:u00:step1:eq4}) can be proved for $V_{ku}^{(2)}, V_{ku}^{(3)}, V_{ku}^{(4)}$ by the same argument. Hence,  using union bound and  (\ref{eq:proposition:uniform:convergence:u00:step1:eq0}), we have
\begin{eqnarray}
\pr(|S_n(t,y)-\ev[S_n(t,y)]|>8\epsilon)&\leq& C\exp\left\{-\frac{\epsilon^2}{2V_n^2+\frac{2}{3}C n^{-1}h_t^{-1}h_y^{-1}p^2\delta_n^{-2}\epsilon}\right\}\nonumber\\
&&+ \frac{C\Delta_n^2}{p^2} \psi(Cp^2\delta_n^{-2} ,n)\phi(p))\sqrt{\frac{C n^{-1}h_t^{-1}h_y^{-1}\Delta_n^2\delta_n^{-2}}{\epsilon}}.\nonumber\\\label{eq:proposition:uniform:convergence:u00:step1:eq5}
\end{eqnarray}

\noindent \textbf{Step 2:}
Let $t_1, \ldots, t_{r}$ and $y_1,\ldots, y_r$ be equidistance partitions of $[0, 1]$ and $\mathfrak{B}$. Hence, it follows that
\begin{eqnarray*}
\sup_{(t,y)\in [0, 1]\times \mathfrak{B}}|S_n(t, y)-\ev[S_n(t, y)]|&\leq& \sup_{1\leq i,j \leq r}|S_n(t_i, y_j)-\ev[S_n(t_i, y_j)]|\\
&&+\sup_{|t-\widetilde{t}|\leq C/r, |y-\widetilde{y}|\leq C/r}|S_n(t,y)-S_n(\widetilde{t}, \widetilde{y})|\\
&&+\sup_{|t-\widetilde{t}|\leq C/r, |y-\widetilde{y}|\leq C/r}|\ev[S_n(t,y)]-\ev[S_n(\widetilde{t}, \widetilde{y})]|\\
\end{eqnarray*}
The Lipschitz conditions of kernel and and the rate $h_t = h_y= O(h)$ imply that
\begin{eqnarray*}
\sup_{|t-\widetilde{t}|\leq C/r, |y-\widetilde{y}|\leq C/r}|S_n(t,y)-S_n(\widetilde{t}, \widetilde{y})|\leq  \sum_{i=1}^n w_i \sum_{j=1}^{N_i}\frac{C}{h_th_y} r^{-1}\sqrt{h_t^{-2}+h_y^{-2}}=O_P\left(h^{-3}r^{-1}\right).
\end{eqnarray*}
Similarly, we also have
\begin{eqnarray}
\sup_{|t-\widetilde{t}|\leq C/r, |y-\widetilde{y}|\leq C/r}|\ev[S_n(t,y)]-\ev[S_n(\widetilde{t}, \widetilde{y})]|=O_P\left(h^{-3}\right).\nonumber
\end{eqnarray}

Using (\ref{eq:proposition:uniform:convergence:u00:step1:eq5}), we have
\begin{eqnarray*}
&&\pr(\sup_{1\leq i,j \leq r}|S_n(t_i, y_j)-\ev[S_n(t_i, y_j)]|>8\epsilon)\\
&\leq& Cr^2\exp\left\{-\frac{\epsilon^2}{2V_n^2+\frac{2}{3}C n^{-1}h_t^{-1}h_y^{-1}p^2\delta_n^{-2}\epsilon}\right\}\nonumber\\
&&+ \frac{Cr^2\Delta_n^2}{p^2} \psi(Cp^2\delta_n^{-2} ,n)\phi(p))\sqrt{\frac{C n^{-1}h_t^{-1}h_y^{-1}\Delta_n^2\delta_n^{-2}}{\epsilon}}.
\end{eqnarray*}
Combining the above inequalities, we can show that 
\begin{eqnarray*}
\sup_{(t,y)\in [0, 1]\times \mathfrak{B}}|S_n(t, y)-\ev[S_n(t, y)]|=o_P(1),
\end{eqnarray*}
if the following rate conditions hold
\begin{eqnarray*}
h^3 r\to \infty,\quad V_n^2 \to  0,\quad \frac{p^2}{nh^2\delta_n^2}\to0,\quad \frac{\Delta_n^6 r^4}{nh^2\delta_n^2 p^4}  \psi^2(Cp^2\delta_n^{-2} ,n)\phi^2(p)\to 0.
\end{eqnarray*}
Taking $r=h^{-3}\log(n)$ and using the rate conditions given, the above conditions become
\begin{eqnarray*}
V_n^2\to 0,\quad \frac{p^2}{nh^2\delta_n^2}\to 0,\quad \frac{\Delta_n^6 \log^4(n) }{nh^{14} \delta_n^5 p^{2\rho}} \to 0.
\end{eqnarray*}
To guarantee the existence of such $p$, sufficient conditions are
\begin{eqnarray*}
V_n^2\to 0,\quad nh^2\delta_n^2&\to&\infty,\quad \frac{\Delta_n^6 \log^4(n) }{nh^{14}\delta_n^5 }<<n^\rho h^{2\rho} \delta_n^{2\rho},
\end{eqnarray*}
which are equivalent to 
\begin{eqnarray*}
V_n^2\to 0,\quad nh^2\delta_n^2\to\infty,\quad \frac{n^{1+\rho}h^{14+2\rho}\delta_n^{5+2\rho}}{\Delta_n^6 \log^4(n)}\to\infty,
\end{eqnarray*}
and they follow from Conditions (IV)(i) and (V)(ii).
\end{proof}

With the help of Proposition \ref{proposition:uniform:convergence:u00}, we now provide the $L_2$ convergence rate for the conditional covariance $\widehat{R}_e(0,t_1,t_2)$.

\begin{proof}[Proof for Lemma \ref{lm:2}]
For simplicity, assume $Y_i \in \mathfrak{B}$, $i=1,\ldots,n$. Note that
\begin{align*}
\widehat{R}_e &= \frac{1}{n}\sum_{i=1}^n \widehat{m}(t_1,Y_i)\widehat{m}(t_2,Y_i) -\frac{1}{n}\sum_{i=1}^n \widehat{m}(t_1,Y_i) \frac{1}{n}\sum_{j=1}^n \widehat{m}(t_2,Y_j).
\end{align*}
Define
\begin{align*}
    \widetilde{R}_e &= \frac{1}{n}\sum_{i=1}^n m(t_1,Y_i)m(t_2,Y_i) -\frac{1}{n}\sum_{i=1}^n m(t_1,Y_i) \frac{1}{n}\sum_{j=1}^n m(t_2,Y_j).
\end{align*}
We write $\widehat{R}_e - R_e = (\widehat{R}_e - \widetilde{R}_e) + (\widetilde{R}_e - R_e)$. First consider the term
\begin{align}
    \widehat{R}_e - \widetilde{R}_e &= \left( \frac{1}{n}\sum_{i=1}^n \widehat{m}(t_1,Y_i)\widehat{m}(t_2,Y_i) - \frac{1}{n}\sum_{i=1}^n m(t_1,Y_i)m(t_2,Y_i) \right) \nonumber\\
    & \qquad - \left( \frac{1}{n}\sum_{i=1}^n \widehat{m}(t_1,Y_i) \frac{1}{n}\sum_{j=1}^n \widehat{m}(t_2,Y_j) - \frac{1}{n}\sum_{i=1}^n m(t_1,Y_i) \frac{1}{n}\sum_{j=1}^n m(t_2,Y_j) \right) \nonumber \\
    &= (I) + (II). \label{eqn:re-hs-split}
\end{align}
Reorganizing (I), we obtain $(I)= S_1 +S_2 + S_3$, where
\begin{align}
\begin{split}
    S_1 &=  \frac{1}{n}\sum_{i=1}^n [\widehat{m}(t_1,Y_i)- m(t_1,Y_i)][\widehat{m}(t_2,Y_i)-m(t_2,Y_i)]  \\
    S_2 &= \frac{1}{n}\sum_{i=1}^n m(t_1,Y_i)[\widehat{m}(t_2,Y_i)-m(t_2,Y_i)] \\
    S_3 &= \frac{1}{n}\sum_{i=1}^n [\widehat{m}(t_1,Y_i)- m(t_1,Y_i)]m(t_2,Y_i). 
    \end{split}
    \label{eqn:lem-re-1s}
\end{align}
Now we evaluate the $L_2$ norm for each term in (\ref{eqn:lem-re-1s}). First consider the $L_2$ norm of $S_1$ which is
\begin{align*}
    \Vert S_1 \Vert_{L_2} &= \left[ \int \int  \left\{ \frac{1}{n}\sum_{i=1}^n [\widehat{m}(t_1,Y_i)- m(t_1,Y_i)][\widehat{m}(t_2,Y_i)-m(t_2,Y_i)] \right\}^2 dt_1 dt_2 \right]^{1/2} \\
    &\le \left[ \int \int   \frac{1}{n}\sum_{i=1}^n [\widehat{m}(t_1,Y_i)- m(t_1,Y_i)]^2\frac{1}{n}\sum_{i=1}^n [\widehat{m}(t_2,Y_i)-m(t_2,Y_i)]^2  dt_1 dt_2 \right]^{1/2} \\
    &= \frac{1}{n} \sum_{i=1}^n \Vert \widehat{m}(t,Y_i)- m(t,Y_i)  \Vert_{L_2}^2.
\end{align*}
From Proposition \ref{proposition:uniform:convergence:u00} and Theorem \ref{th:1} and using arguments analogous to Lemma 3 and Theorem 4.2 in \citet{zhang2016sparse}, we obtain
\begin{align}
    \Vert \widehat{m}(t,y)- m(t, y) \Vert_{HS} =& \int \int_{\mathfrak{B}}  [\widehat{m}(t,y)- m(t, y)]^2 dy dt \nonumber\\ 
    &= O_p\left(h^2 + \sqrt{\frac{\sum_{i=1}^n N_i w_i^2}{h^2} + \frac{\sum_{i=1}^nN_i(N_i-1)w_i^2}{h} }\right). \label{eqn:re-hs-s1-l2}
\end{align}
Therefore,
\begin{align*}
    \Vert S_1 \Vert_{L_2} &= O_p\left(h^2 + \sqrt{\frac{\sum_{i=1}^n N_i w_i^2}{h^2} + \frac{\sum_{i=1}^nN_i(N_i-1)w_i^2}{h} }\right).
\end{align*}
Similar calculations yields that both $\Vert S_2 \Vert_{L_2}$ and $\Vert S_3 \Vert_{L_2}$ are of the same order as in (\ref{eqn:re-hs-s1-l2}). Hence,  
\begin{align}
    \widehat{R}_e - \widetilde{R}_e &= O_p\left(h^2 + \sqrt{\frac{\sum_{i=1}^n N_i w_i^2}{h^2} + \frac{\sum_{i=1}^nN_i(N_i-1)w_i^2}{h} }\right).
\end{align}
Moreover,  for part (II) in (\ref{eqn:re-hs-split}) we obtain $\widetilde{R}_e - R_e=O_p(1/\sqrt{n})$. Thus the result is proved. 
\end{proof}

\section{Simulation study} \label{sec:numer}

We illustrate the numerical performance of the proposed methodology using simulations. Our simulation design is similar to \citet{zhang2022unified} in terms of generating spatially correlated functional data. The functional data are generated from (\ref{eqn:x-meas-err}) in the spatial domain $\mathcal{D}=[0,1]^2$ and time domain $\mathcal{I}=[0,1]$, as 
\begin{align*}
  X(\bm{s}; t) &= \mu(t) + \sum_{j=1}^3 A_j(\bm{s}) \pi_j(t),
\end{align*} 
where $\mu(t)=2t\sin(2\pi t)$, $\pi_1(t)=\cos(2\pi t)$, $\pi_2(t)=\sin(2\pi t)$, and $\pi_3(t)=\cos(4\pi t)$. The principal component scores $A_j(\bm{s})$, $j=1,2,3$, are Gaussian random fields generated using the \emph{RandomFields} package \citep{randomfields} in \texttt{R}, where their spatial covariance functions are members of the Matern family, 
\begin{align*}
    \mathcal{C}_j(u; v,r) &=  V_{A,j} \frac{2^{1-v}}{\Gamma(v)} (\sqrt{2v}u/r)^v K_v(\sqrt{2v}u/r),
\end{align*}
where the variances of $A_j(\cdot)$, $j=1,2,3$, are $(V_{A,1},V_{A,2},V_{A,3})=(1, 1.5, 1)$ and $K_v(\cdot)$ is the modified Bessel function of the second kind with degree $v$.  The shape parameter $v$  and the scale parameter $r$ are set to be $0.2, 0.1, 0.05$ and $1,0.5,1$, for the three principal components, respectively. The spatial coordinates $\{\bm{s}_i\}$ are drawn uniformly in the square $\{(0,0),(0,1),(1,0),(1,1) \}$ with a hole inside the coordinates $\{(0.6,0.2), (0.4,0.2)$,$(0.4,0.8),(0.6,0.8) \}$ as shown in Figure \ref{fig:sp-sample}. The measurement error $\{e_{ij}\}$ is generated as i.i.d. $N(0, 0.1^2)$. We also consider a functional nugget effect $U_i(t)=\sum_{j=1}^2A_{\text{nug},j}(\bm{s}_i)\pi_{\text{nug},j}$, with $\pi_{\text{nug},1}= \sqrt{3}t$ and $\pi_{\text{nug},2}=-2\sqrt{15/7}t^2+\sqrt{15/7}$. We consider $A_{\text{nug},j} \sim N(0, V_{\text{nug},j})$, $j=1,2$ with $(V_{\text{nug},1},V_{\text{nug},2})=(0.5,1)$. Then the response variable is generated as
\begin{align*}
     Y_i &= 3+ f(\langle\beta, X_i\rangle)+ \epsilon_i, \qquad i=1,\ldots,n,
\end{align*}
with $f(x)=x/(1+\exp(x))$, $\beta(t)=\sqrt{2}\sin(3\pi t/2)$, $Z_{ij}=X_{ij}+U_{ij}+e_{ij}$, and $\epsilon_i \sim N(0, 0.1^2)$.

We apply a 3-fold cross-validation procedure for bandwidth selection.  Primarily, our simulation design consists of the following four combinations: (nugget, no nugget) $\times$ (sparse, dense). The number of repeated measurements is drawn from a discrete uniform distribution $[3,7]$ for  sparse  and  from $[10, 15+ \lfloor n/10 \rfloor ]$ for dense functional data,  where  $n=\{30, 50, 100\}$.   We estimate e.d.r. directions for each of the four combinations using both $\widehat{R}(0,\cdot,\cdot)$ and $\widehat{\Gamma}(\cdot,\cdot)$ under both weighting schemes SUBJ and OBS. Hereafter, we denote the proposed estimation method using $\widehat{R}(0,\cdot,\cdot)$  as SFSIR, and the regular estimation approach using   $\widehat{\Gamma}(\cdot,\cdot)$ as FSIR \citep{jiang2014inverse}. The simulations are carried out on a linux server using \texttt{R} programming language \citep{Rsoftware}. For comparison, we also estimate $\beta$ using functional regression model (\texttt{FLM1} function in the \emph{fdapace} package \citep{fdapace} in \texttt{R}); we refer to it as FLM.

\begin{figure}
    \centering
    \includegraphics[scale=0.4]{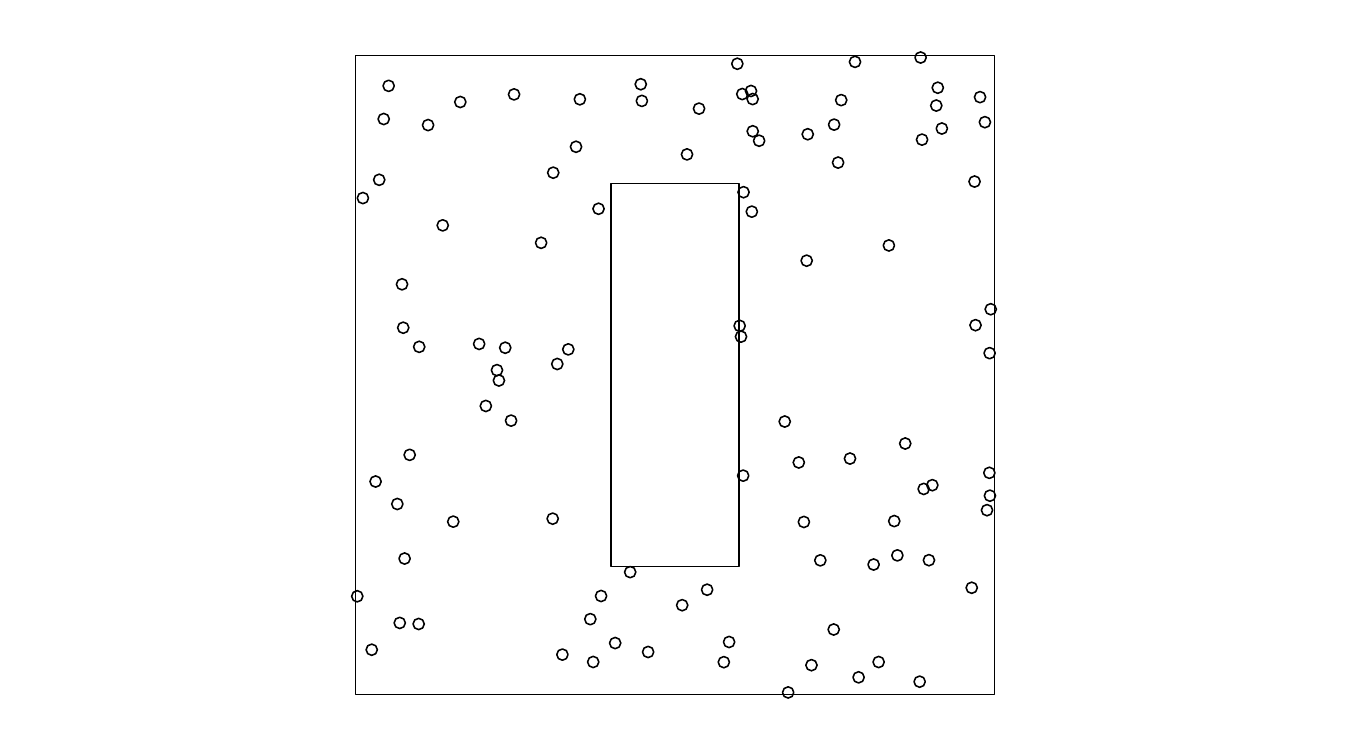}
    \caption{Locations of generated spatial sites on $[0,1]^2$. \texttt{runifpoint} function in \emph{RandomFields} package is utilized to generate the data.}
    \label{fig:sp-sample}
\end{figure}
To assess the performance of aforementioned methods, we compute the integrated squared bias (ISB), integrated variance (IVAR), and mean integrated  squared error (MISE) which are defined in \citet{jiang2014inverse}. To reduce the computational burden for the estimation of $\widehat{R}(\cdot; \cdot,\cdot)$, which includes the computation of $C_{ij,i'j'} :=[Z_{ij}-\widehat{\mu}(T_{ij})][Z_{i'j'}-\widehat{\mu}(T_{i'j'})]$, we consider the following strategies:
\begin{itemize}
    \item We compute the cross-products across functions, $(i,i')$, of $i$th function with top 20\% of the nearest functions.
    \item We use the \texttt{binning} function in the \textit{sm} package \citep{smpackage} in \texttt{R} to reduce the number of observations. 
\end{itemize}

\begin{table}[ht]
\centering
\begin{tabular}{lllrrr|rrr}
  \hline
   &  &  & \multicolumn{3}{c}{SUBJ} & \multicolumn{3}{c}{OBS}\\
   
   & & & ISB & IVAR & MISE & ISB & IVAR & MISE \\ 
  \hline
\multirow{6}{*}{Sparse} & \multirow{3}{*}{Nugget} & FSIR  & 0.39 & 0.78 & 1.18 & 0.37 & 0.78 & 1.15 \\ 
   &  & SFSIR & 0.22 & 0.60 & 0.82 & 0.25 & 0.61 & 0.87 \\ 
   &  & FLM & 0.71 & 0.55 & 1.27 & 0.71 & 0.55 & 1.27 \\ 
   & \multirow{3}{*}{No Nugget} & FSIR & 0.27 & 0.57 & 0.84 & 0.28 & 0.56 & 0.84 \\ 
   &  & SFSIR & 0.19 & 0.62 & 0.81 & 0.18 & 0.61 & 0.80 \\ 
   &  & FLM & 0.45 & 2.47 & 2.92 & 0.45 & 2.47 & 2.92 \\ 
   \hline \\
  \multirow{6}{*}{Dense} & \multirow{3}{*}{Nugget} & FSIR & 0.39 & 0.65 & 1.04 & 0.44 & 0.60 & 1.03 \\
   &  & SFSIR  & 0.24 & 0.53 & 0.77 & 0.25 & 0.59 & 0.83 \\ 
   &  & FLM & 0.79 & 0.24 & 1.02 & 0.79 & 0.24 & 1.02 \\ 
   & \multirow{3}{*}{No Nugget}  & FSIR & 0.20 & 0.56 & 0.76 & 0.20 & 0.50 & 0.71 \\ 
   &  & SFSIR& 0.18 & 0.57 & 0.75 & 0.20 & 0.57 & 0.77 \\ 
   &  & FLM & 0.49 & 0.32 & 0.81 & 0.49 & 0.32 & 0.81 \\ 
   
   \hline
\end{tabular}
\caption{Here $n=50$. FSIR: \citet{jiang2014inverse}. SFSIR: Proposed method. FLM: \texttt{FLM1} function in \emph{fdapace} package. }
\label{tab:fifty}
\end{table}

\begin{table}[ht]
\centering
\begin{tabular}{lllrrr|rrr}
  \hline
   &  &  & \multicolumn{3}{c}{SUBJ} & \multicolumn{3}{c}{OBS}\\
   
   & & & ISB & IVAR & MISE & ISB & IVAR & MISE \\ 
  \hline
\multirow{6}{*}{Sparse} & \multirow{3}{*}{Nugget} & FSIR  & 0.29 & 0.58 & 0.87 & 0.38 & 0.58 & 0.95 \\ 
   &  & SFSIR & 0.17 & 0.55 & 0.72 & 0.16 & 0.56 & 0.72 \\  
   &  & FLM & 0.78 & 0.41 & 1.18 & 0.77 & 0.40 & 1.18 \\
   & \multirow{3}{*}{No Nugget} & FSIR& 0.27 & 0.48 & 0.75 & 0.26 & 0.46 & 0.72 \\
   &  & SFSIR & 0.15 & 0.49 & 0.64 & 0.14 & 0.54 & 0.68 \\
   &  & FLM & 0.41 & 1.25 & 1.66 & 0.41 & 1.24 & 1.65 \\
   \hline \\
  \multirow{6}{*}{Dense} & \multirow{3}{*}{Nugget} & FSIR & 0.53 & 0.45 & 0.98 & 0.52 & 0.44 & 0.96 \\ 
   &  & SFSIR & 0.19 & 0.45 & 0.64 & 0.21 & 0.45 & 0.66 \\ 
   &  & FLM & 0.89 & 0.07 & 0.97 & 0.89 & 0.07 & 0.96 \\  
   & \multirow{3}{*}{No Nugget}  & FSIR & 0.12 & 0.35 & 0.47 & 0.15 & 0.34 & 0.49 \\
   &  & SFSIR & 0.20 & 0.46 & 0.67 & 0.17 & 0.44 & 0.61 \\
   &  & FLM & 0.50 & 0.02 & 0.52 & 0.50 & 0.02 & 0.52 \\ 
   
   \hline
\end{tabular}
\caption{Here $n=100$. FSIR: \citet{jiang2014inverse}. SFSIR: Proposed method. FLM: \texttt{FLM1} function in \emph{fdapace} package. }
\label{tab:hundred}
\end{table}

We evaluate the mean and covariance functions at 101 grid points between 0 and 1 and  use 100 replications for each combination of the simulation design. For each replication, for both SFSIR and  FSIR, the bandwidths are chosen based on 3-fold cross-validation and they are used to estimate the eigendirections. For FLM, the \emph{fdapace} package uses 10\% of the support by default for bandwidth selection. The simulation results for sample sizes (n = 30, 50, 100)  are provided in Tables S.1, \ref{tab:fifty} and \ref{tab:hundred}. The plus or minus 1 standard deviation bands are also provided in Figures S.1, \ref{fig:sub50} and S.3 for the SUBJ scheme and in Figures S.2, \ref{fig:obs50} and S.4 for the OBS weighting scheme in Supplementary Material \citep{suppl}. We observe that the proposed SFSIR yields the smallest MISE when the data includes nugget effect for both sparse and dense data and for both weighting schemes SUBJ and OBS. On the other hand, the proposed method does not lose much  efficiency, especially for sparse data when the nugget effect is not included.  Similarly to the results from \citet{zhang2016sparse}, the OBS scheme outperforms the SUBJ scheme for sparse data and the SUBJ scheme works better for dense functional data.   On the whole, as expected, the estimation improves with the sample size.

To sum up, the proposed method yields good performance in our simulation study. Although we adopted some strategies to reduce the computational burden of the proposed method, they are far from being optimal. A more systematic approach is needed to carefully device an efficient algorithm. 
\begin{figure}
    \centering
    \includegraphics[scale=0.4]{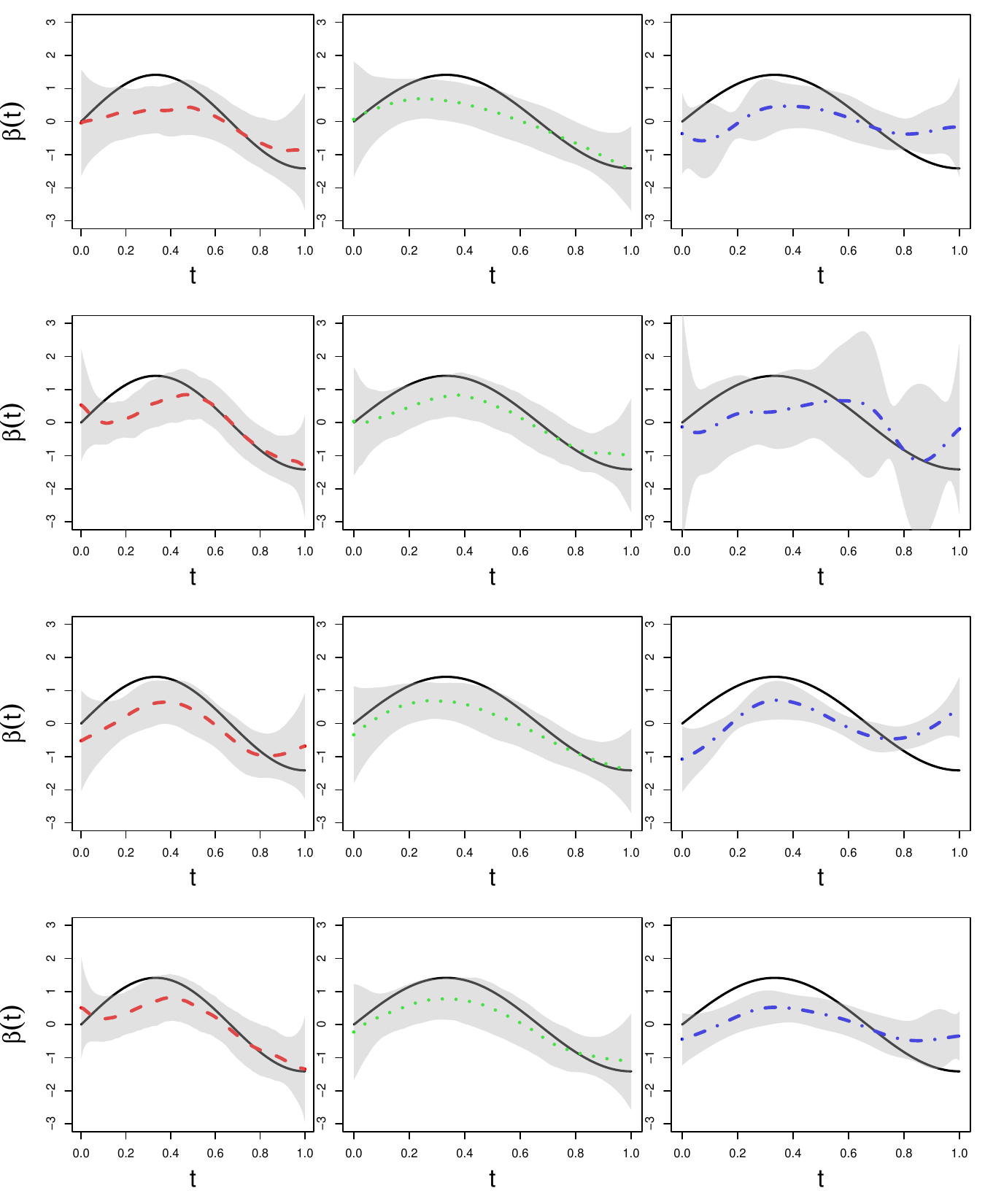}
    \caption{(n=50) Estimated mean curves with SUBJ weighting scheme along with plus or minus 1 standard deviation. Columns: (left, middle, right) = (FSIR, SFSIR, FLM). Rows: (1,2,3,4)=(sparse+nugget, sparse+no nugget, dense+nugget, dense+no nugget). Solid line denotes the true function}
    \label{fig:sub50}
\end{figure}
\begin{figure}
    \centering
    \includegraphics[scale=0.4]{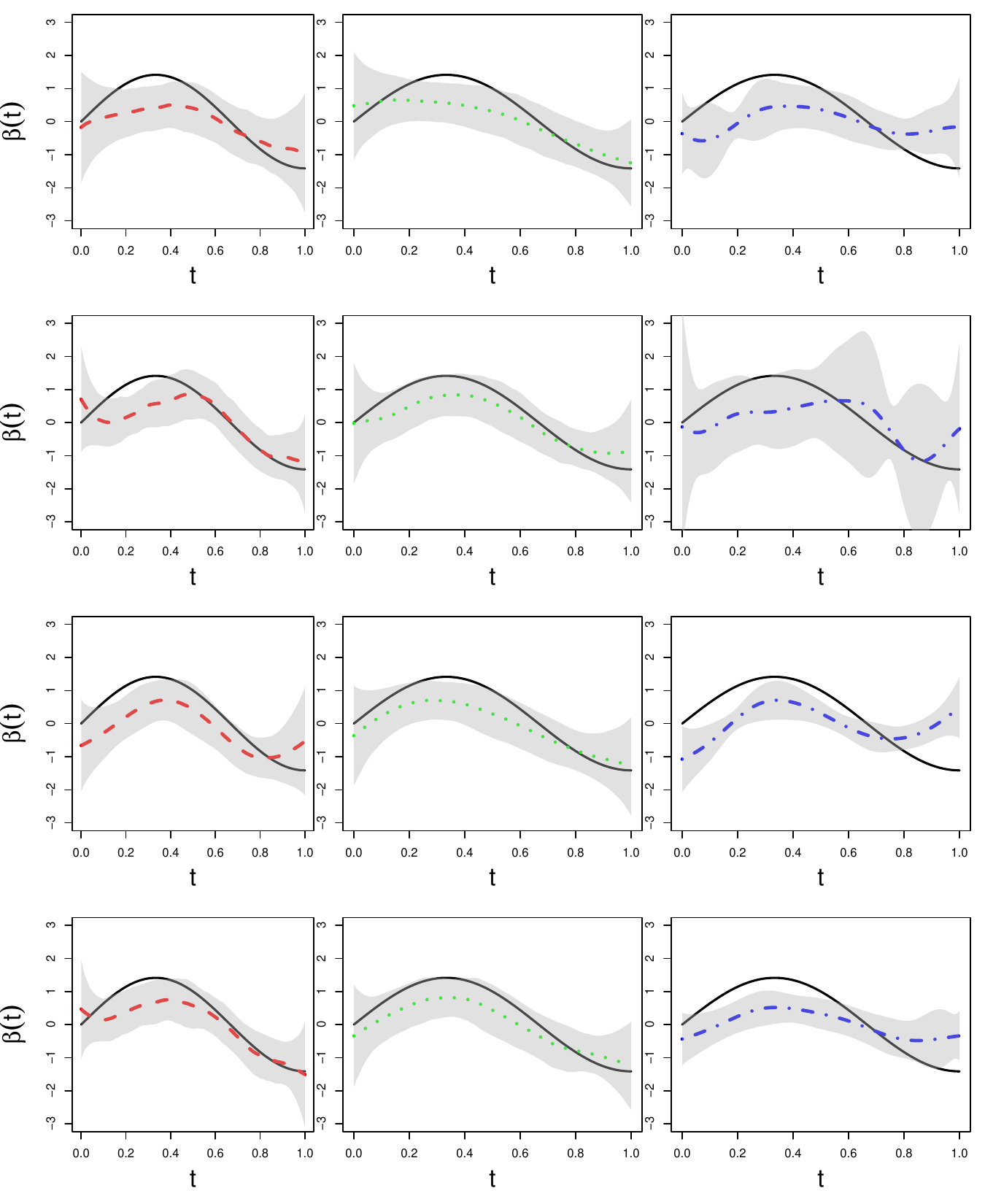}
    \caption{(n=50) Estimated mean curves with OBS weighting scheme along with plus or minus 1 standard deviation. Columns: (left, middle, right) = (FSIR, SFSIR, FLM). Rows: (1,2,3,4)=(sparse+nugget, sparse+no nugget, dense+nugget, dense+no nugget). Solid line denotes the true function}
    \label{fig:obs50}
\end{figure}

\begin{table}[ht]
\centering
\begin{tabular}{lllrrr|rrr}
  \hline
   &  &  & \multicolumn{3}{c}{SUBJ} & \multicolumn{3}{c}{OBS}\\
   
   & & & ISB & IVAR & MISE & ISB & IVAR & MISE \\ 
  \hline
\multirow{6}{*}{Sparse} & \multirow{3}{*}{Nugget} & FSIR  & 0.42 & 0.83 & 1.26 & 0.50 & 0.88 & 1.38 \\  
   &  & SFSIR & 0.38 & 0.72 & 1.11 & 0.34 & 0.70 & 1.05 \\  
   &  & FLM & 0.77 & 1.29 & 2.06 & 0.77 & 1.29 & 2.06 \\
   & \multirow{3}{*}{No Nugget} & FSIR & 0.26 & 0.70 & 0.96 & 0.24 & 0.66 & 0.90 \\ 
   &  & SFSIR & 0.22 & 0.69 & 0.91 & 0.17 & 0.63 & 0.80 \\
   &  & FLM & 0.47 & 2.71 & 3.18 & 0.47 & 2.71 & 3.18 \\ 
   \hline \\
  \multirow{6}{*}{Dense} & \multirow{3}{*}{Nugget} & FSIR & 0.33 & 0.77 & 1.10 & 0.45 & 0.80 & 1.25 \\ 
   &  & SFSIR & 0.26 & 0.64 & 0.91 & 0.30 & 0.70 & 1.00 \\ 
   &  & FLM & 0.77 & 0.40 & 1.17 & 0.77 & 0.40 & 1.17 \\  
   & \multirow{3}{*}{No Nugget}  & FSIR& 0.26 & 0.60 & 0.86 & 0.24 & 0.59 & 0.83 \\
   &  & SFSIR & 0.19 & 0.61 & 0.80 & 0.18 & 0.54 & 0.72 \\ 
   &  & FLM & 0.46 & 2.93 & 3.39 & 0.46 & 2.93 & 3.39 \\ 
   
   \hline
\end{tabular}
\caption{Here $n=30$. FSIR: \citet{jiang2014inverse}. SFSIR: Proposed method. FLM: \texttt{FLM1} function in \emph{fdapace} package. }
\label{tab:thirty}
\end{table}

\begin{figure}
    \centering
    \includegraphics[scale=0.4]{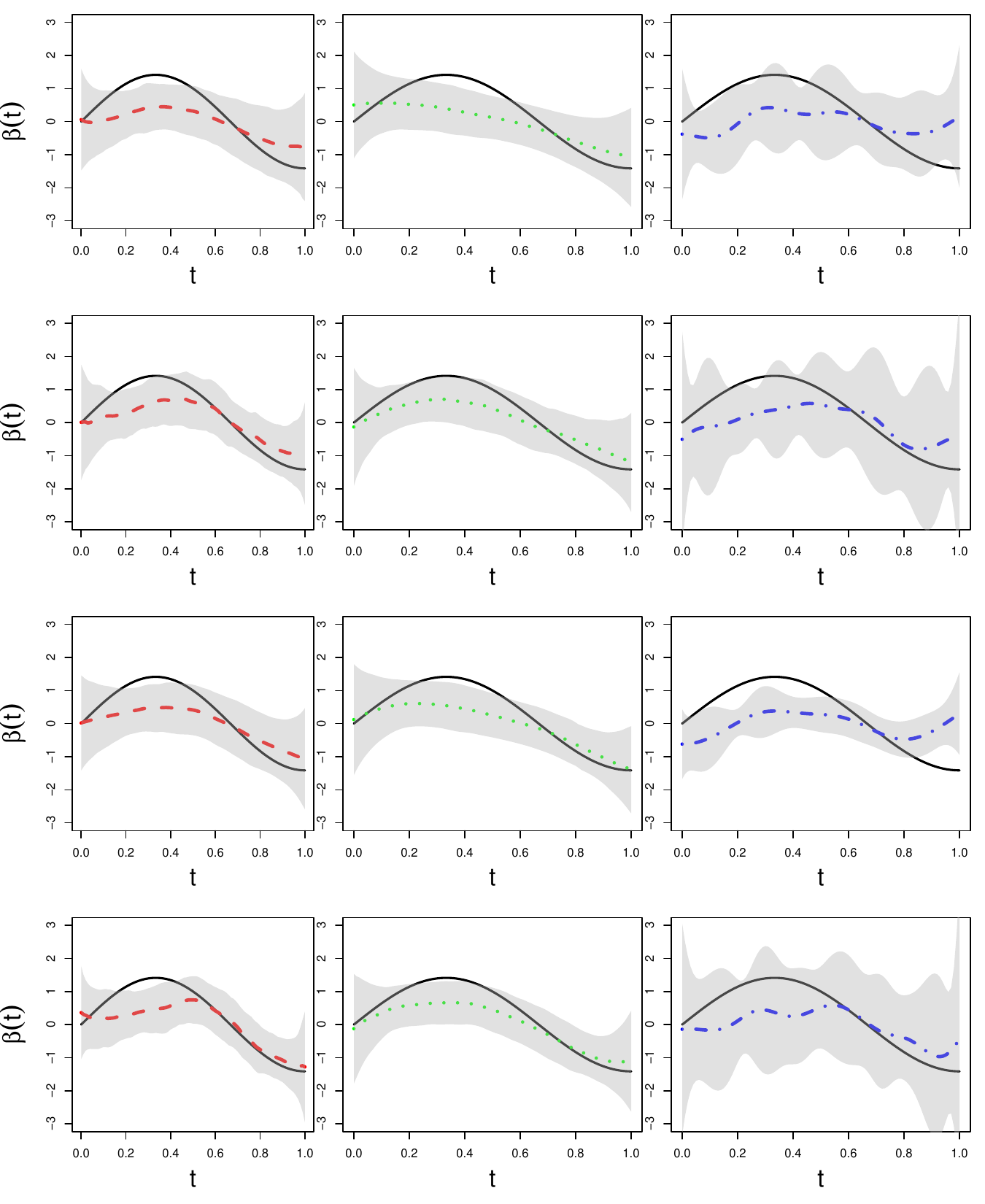}
    \caption{(n=30) Estimated mean curves with SUBJ weighting scheme along with plus or minus 1 standard deviation. Columns: (left, middle, right) = (FSIR, SFSIR, FLM). Rows: (1,2,3,4)=(sparse+nugget, sparse+no nugget, dense+nugget, dense+no nugget). Solid line denotes the true function}
    \label{fig:sub30}
\end{figure}
\begin{figure}
    \centering
    \includegraphics[scale=0.4]{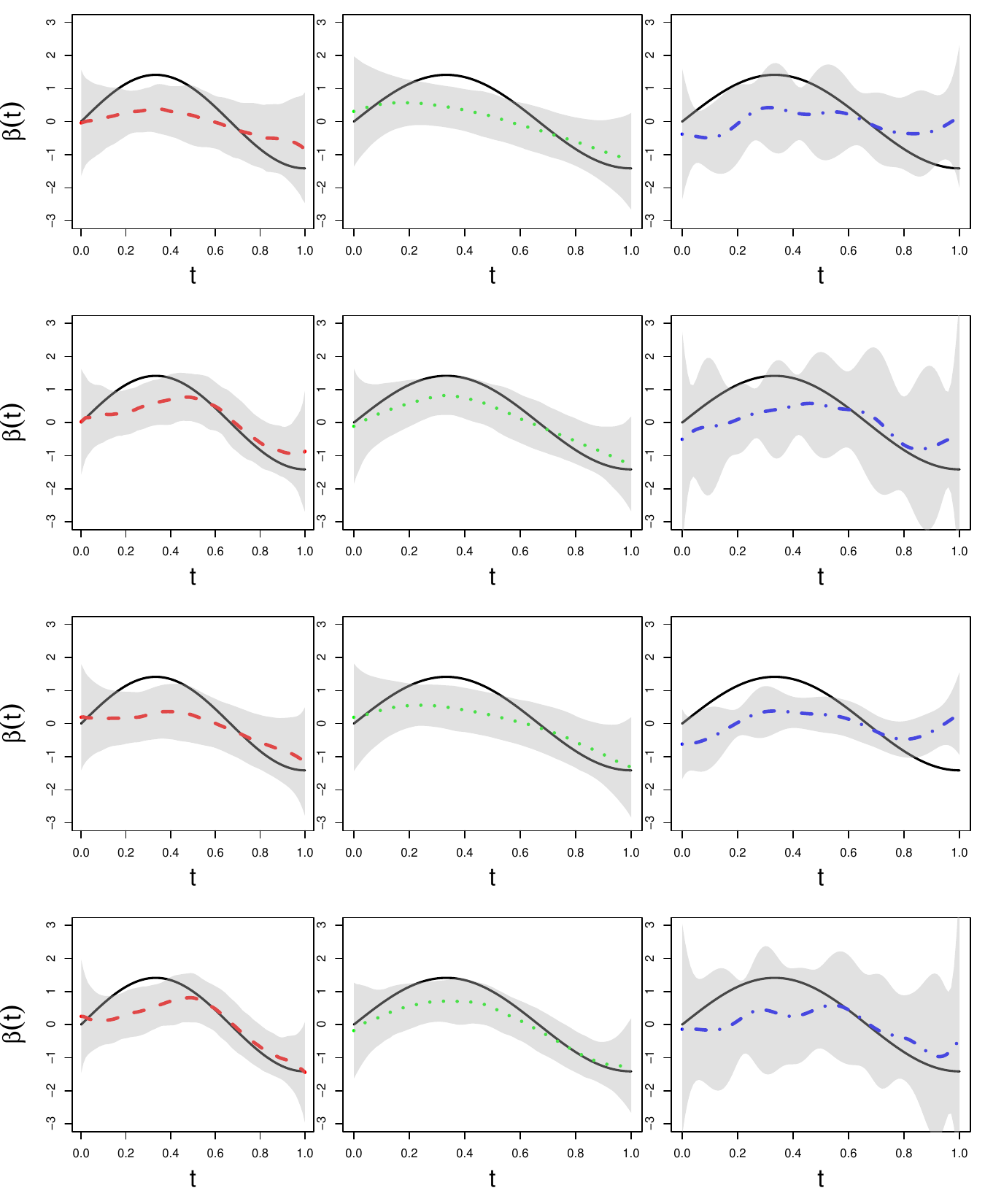}
    \caption{(n=30) Estimated mean curves with OBS weighting scheme along with plus or minus 1 standard deviation. Columns: (left, middle, right) = (FSIR, SFSIR, FLM). Rows: (1,2,3,4)=(sparse+nugget, sparse+no nugget, dense+nugget, dense+no nugget). Solid line denotes the true function}
    \label{fig:obs30}
\end{figure}

\begin{figure}
    \centering
    \includegraphics[scale=0.4]{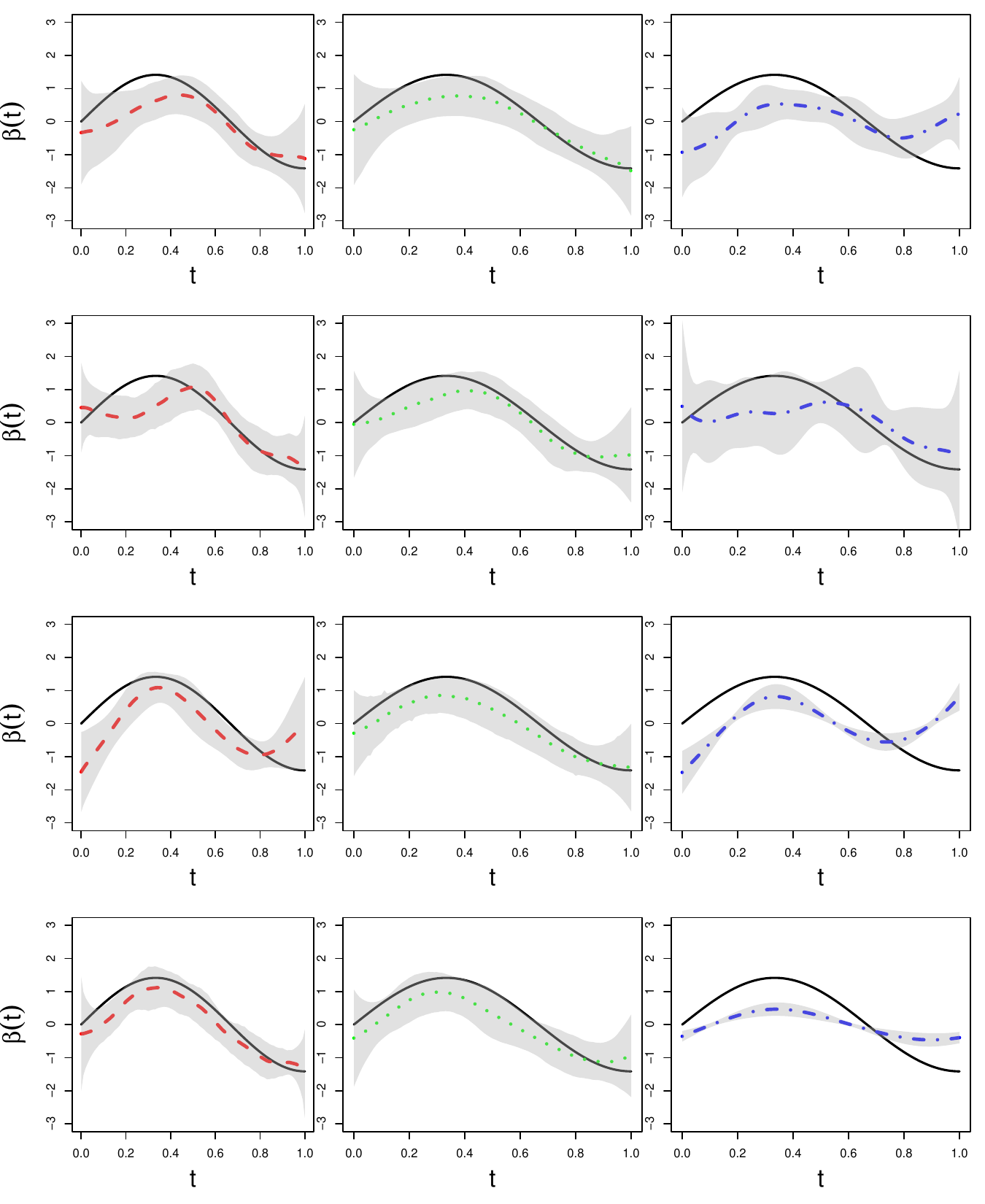}
    \caption{(n=100) Estimated mean curves with SUBJ weighting scheme along with plus or minus 1 standard deviation. Columns: (left, middle, right) = (FSIR, SFSIR, FLM). Rows: (1,2,3,4)=(sparse+nugget, sparse+no nugget, dense+nugget, dense+no nugget). Solid line denotes the true function}
    \label{fig:sub100}
\end{figure}
\begin{figure}
    \centering
    \includegraphics[scale=0.4]{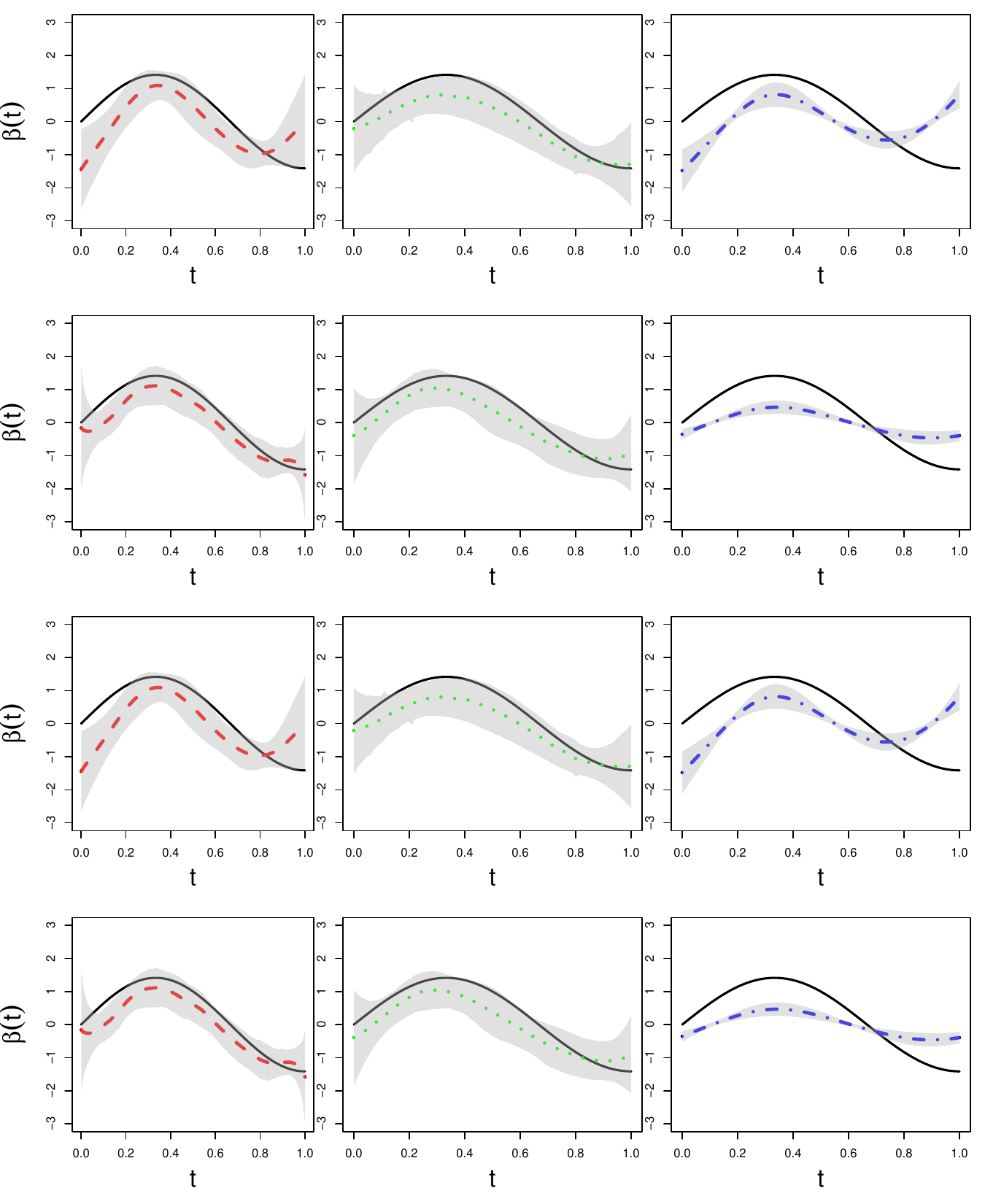}
    \caption{(n=100) Estimated mean curves with OBS weighting scheme along with plus or minus 1 standard deviation. Columns: (left, middle, right) = (FSIR, SFSIR, FLM). Rows: (1,2,3,4)=(sparse+nugget, sparse+no nugget, dense+nugget, dense+no nugget). Solid line denotes the true function}
    \label{fig:obs100}
\end{figure}

\section{Real data application}\label{sec:application}
To showcase the usefulness of the proposed sufficient dimension reduction method, we analyzed the weather data from British Columbia (BC) region. The \texttt{R} package \emph{rnoaa} \citep{chamberlain2022package} is used to extract the weather data for British Columbia region in Canada. This package helps accessing data from Global Historical Climatology Network  daily (GHCNd)  from the National Ocenic and Atmospheoric Administration Information (NOAA) server. The GHCNd is an integrated database of daily weather summaries from land surface stations all over the globe. 
\begin{figure}[h!]
  \centering
  \includegraphics[scale=0.4]{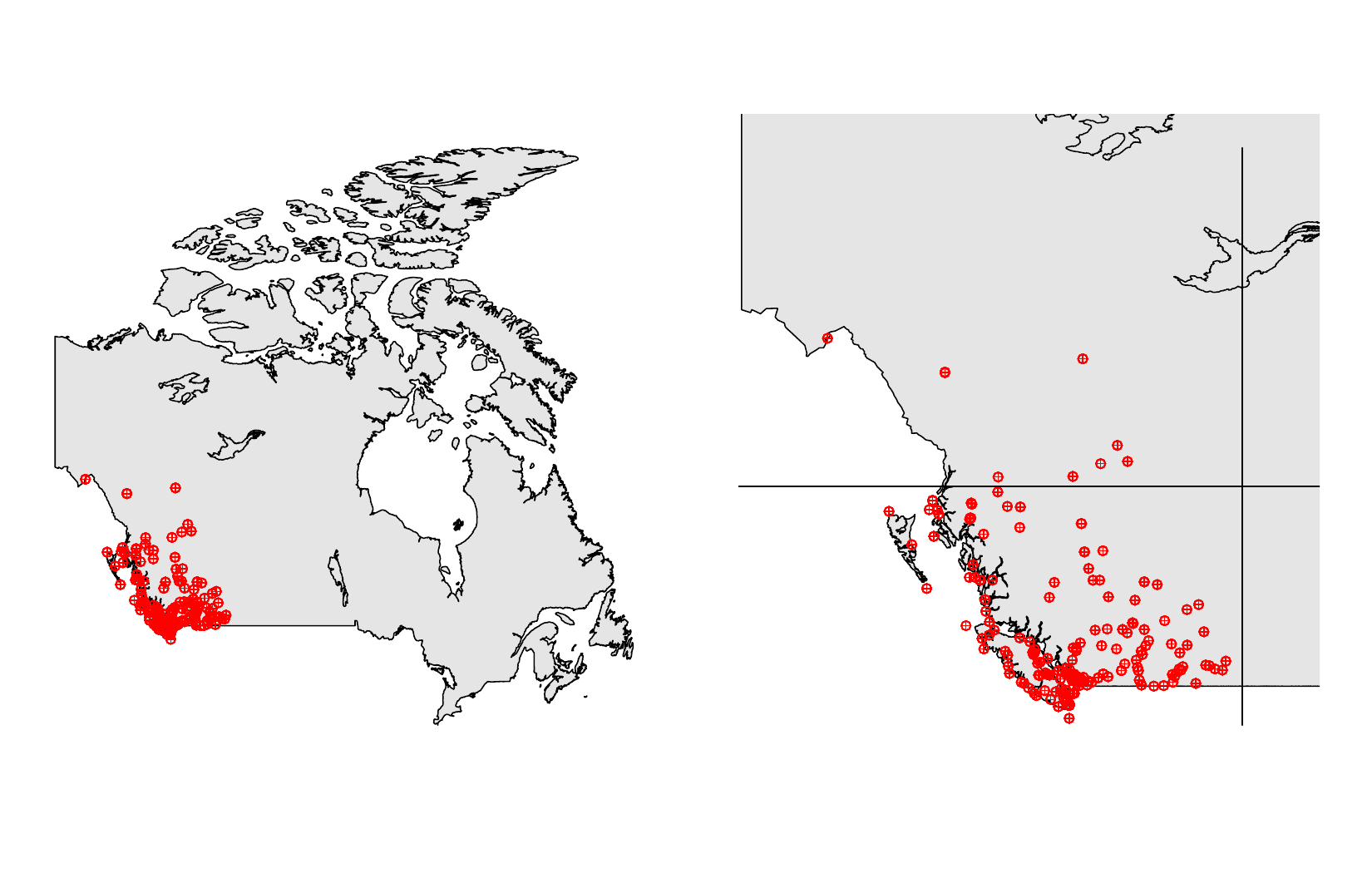}
  \caption{ (left) Locations of 184 weather stations considered from the BC region (Canada map) (right) Enlarged map focusing on British Columbia region with marked outliers (7 outlying stations).}
  \label{fig:bc-loc}
\end{figure}

The locations of 184 weather stations across BC region are shown in the right side plot in Figure \ref{fig:bc-loc}. We notice that the locations of the weather stations are irregularly positioned. For computational simplicity, in this study, we focus on the monthly data for the period starting from \texttt{"2020-06-01"} to \texttt{"2022-05-31"}, for a randomly sampled 100 locations of stations shown in \ref{fig:bc-loc}. The observed data includes information on \emph{precipitation} and \emph{temperature}. 
\begin{figure}[h!]
  \centering
  \includegraphics[scale=0.45]{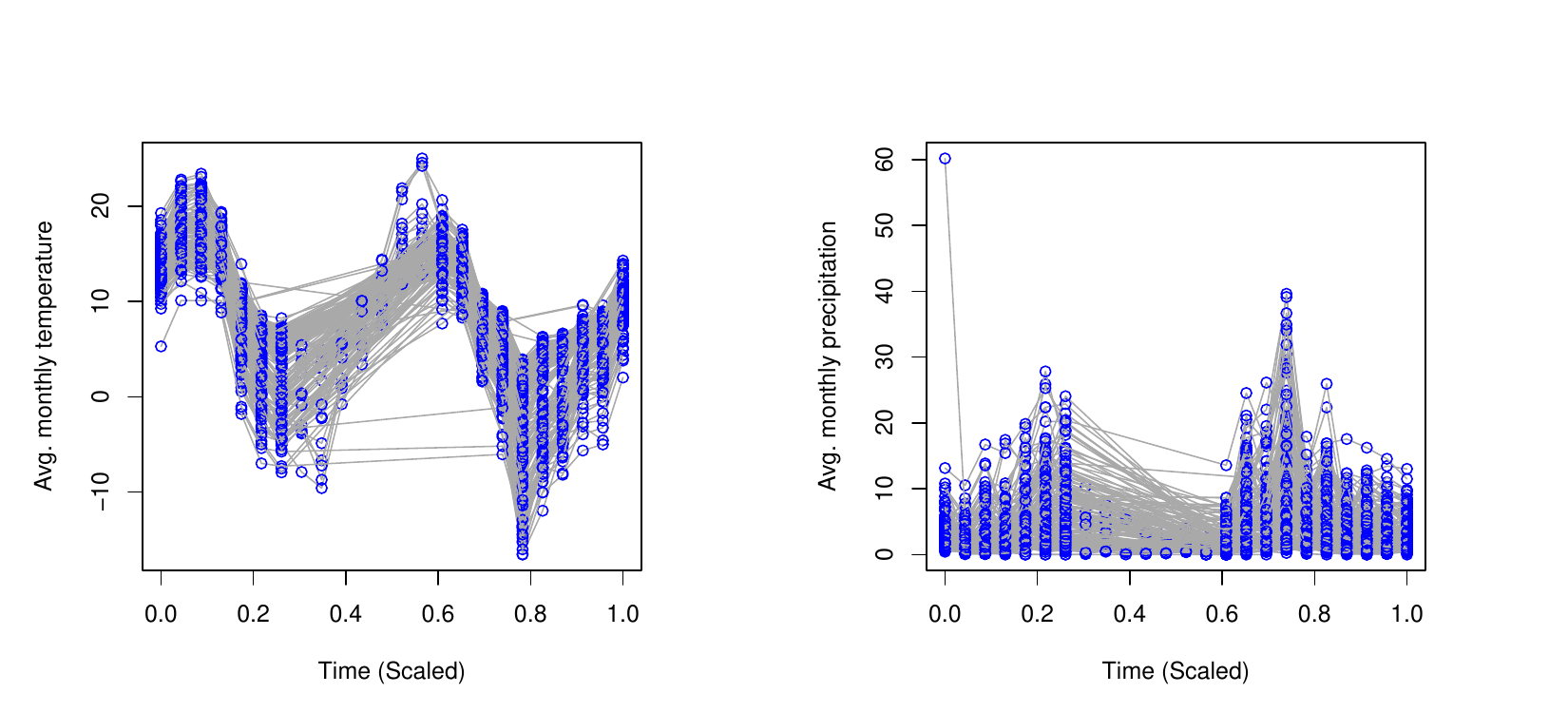}
  \caption{ Average (a) monthly temperatures and  (b) monthly precipitations  (divided by 10), for 24 months starting from \texttt{"2020-06-01"} to \texttt{"2022-05-31"}  across 177 weather stations in BC region.}
  \label{fig:fplot-pt}
\end{figure}
Figure \ref{fig:fplot-pt} shows the monthly temperatures and precipitations across  the 177 weather stations from Figure \ref{fig:bc-loc}.   From Figure \ref{fig:fplot-pt}, we observe that the  monthly temperature exhibits a cyclical pattern. Monthly precipitation has two peak months and but is relatively stable across all the remaining months. Although, we show the monthly precipitations, in the actual model (shown below) we consider the average precipitation across all the locations as our response variable. We also note that not all of the weather stations have data available for the entire 24 months. The exact counts for the number of time points are provided in Figure \ref{fig:ftab}.
\begin{figure}[h!]
  \centering
\includegraphics[scale=0.55]{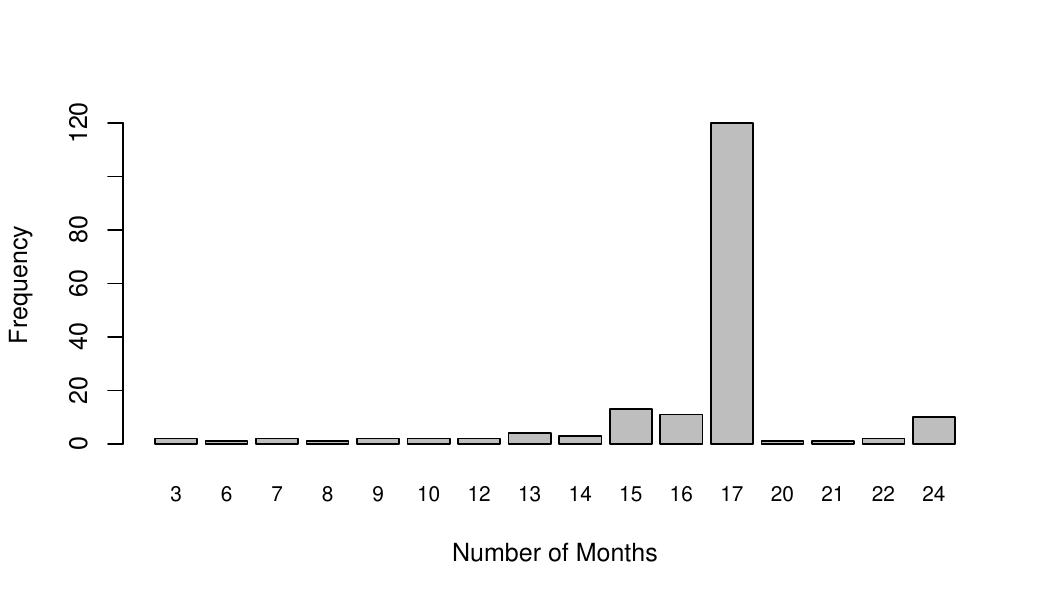}
\caption[]{Frequency table for the functions with total number of time points where the data is available.}
\label{fig:ftab}
\end{figure}
We consider average precipitation as the scalar response and the monthly temperature as the functional predictor in the inverse regression framework. Formally, we consider the following model 
\begin{align}
  AveragePrecipitation_i &= f(\langle\beta_1(t), Temperature_i(t) \rangle, \ldots, \langle\beta_K(t),Temperature_i(t)\rangle, \epsilon_i),
  \label{eqn:weather-model}
\end{align}
for $ i=1,\ldots,100$. Hereafter, we denote the proposed estimation method as SFSIR. For comparison, we also include the inverse regression approach for i.i.d. data \citep{jiang2014inverse} and functional linear model approach (\texttt{FLM1} function in \emph{fdapace} package \citep{fdapace}). We denote both these methods as FSIR and FLM, respectively.

An important concern that needs to be addressed before fitting the proposed method is the bandwidth selection. For both the proposed SFSIR and FSIR methods, we use 3-fold cross-validation separately for both SUBJ and OBS weighting schemes. 
For the method FLM, defaults from the \emph{fdapace} package which uses 10\% of the support are considered. Epanechnikov kernel is used throughout and the functions are evaluated on a grid of 51 points. We consider 50 random samples of size 100 and on each sample
\begin{itemize}
    \item the optimal bandwidths for SFSIR and FSIR are computed based on the 3-fold cross-validation, with 10\% of the support for FLM; and 
    \item the methods SFSIR, FSIR, and FLM are applied with the corresponding optimal bandwidths.
\end{itemize}

The average estimated e.d.r. directions/coefficient functions from the above considered three methods across 50 samples are shown in Figures \ref{fig:weather-subj-obs} for both SUBJ and OBS schemes. The plots include mean curves from each method along with plus or minus one standard deviation. While we can extract multiple e.d.r. directions from the SFSIR and FSIR methods, we can only obtain one coefficient function from the FLM method since it is a linear model. We notice from Figure \ref{fig:fplot-pt} that, since the monthly temperature exhibits a decreasing and cyclical pattern, we expect the e.d.r.  $\widehat{\beta}_1$ may shows an increasing and cyclical pattern that  complements the pattern of the temperature data. As expected, the estimated e.d.r. directions in Figure \ref{fig:weather-subj-obs} shows an increasing trend with a cyclical pattern. The estimated function from the FLM approache exhibit more fluctuations and look different from the estimated directions from the  SFSIR and FSIR approaches. The differences are more visible at the both the beginning  and ending parts of the times series  where more data available (more stations provided data for those time points). Similar to the findings observed from the simulation study, the proposed method SFSIR provides better performance when the underlying noise (nugget effect) is large.

\begin{figure}[h!]
  \centering
\includegraphics[scale=0.35]{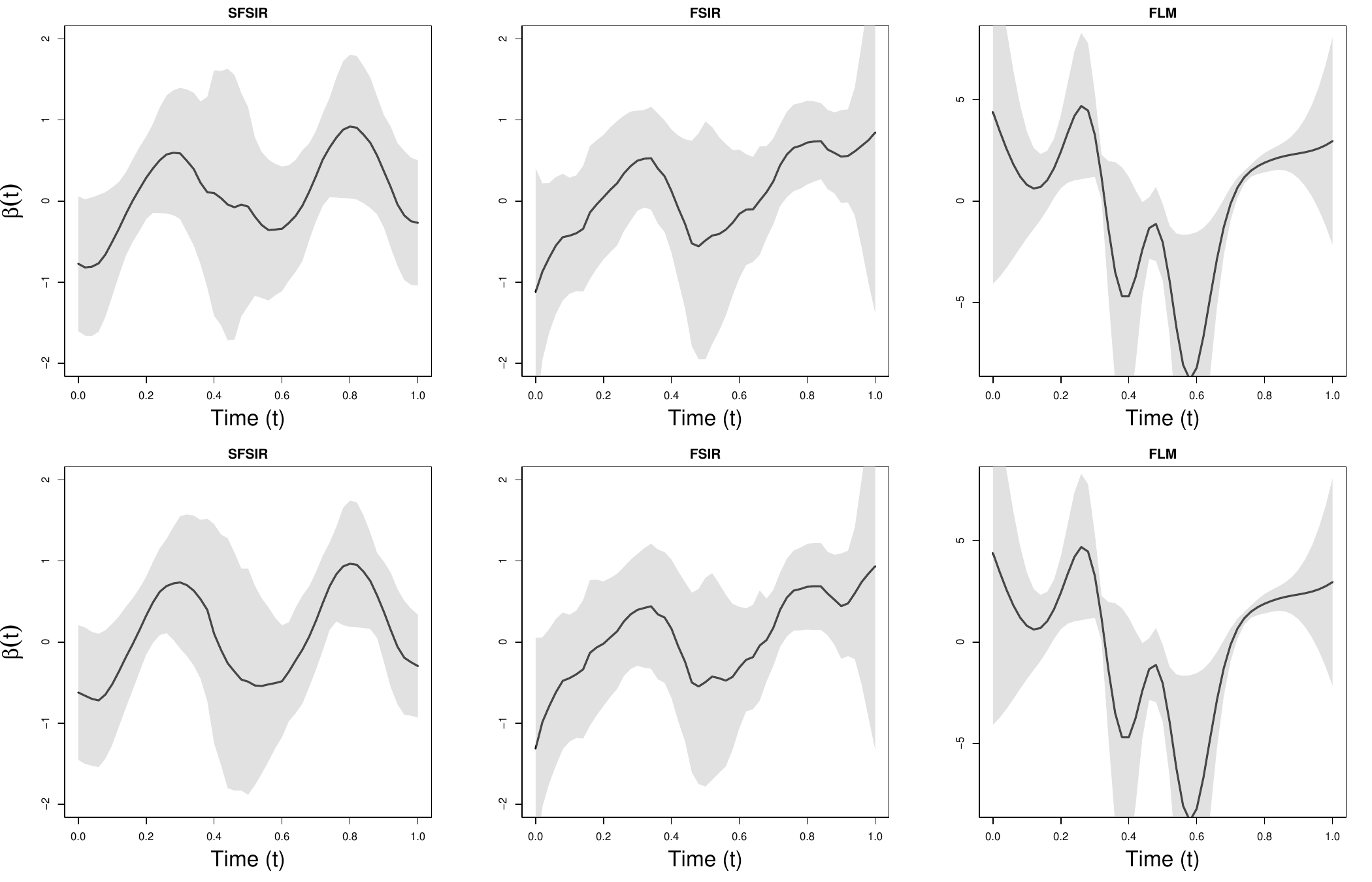}
\caption[]{Estimated e.d.r. directions using  the proposed SFSIR, FSIR , and FLM approaches. The solid lines represent the mean curves across 50 replications and the shaded area represents mean plus or minus one standard deviation. The top row plots use SUBJ weighting scheme while the bottom row plots use OBS weighting scheme.}
\label{fig:weather-subj-obs}
\end{figure}
%
%

\end{document}